\documentclass[11pt]{amsart} 
\usepackage[left=3cm,right=3cm,top=3cm,bottom=3cm]{geometry} 
\usepackage{amsmath,amssymb} 
\usepackage{amsmath, amsthm,amssymb,amsfonts,latexsym}
\usepackage{amsthm}
\usepackage{tikz}
\usepackage{lipsum}
\usepackage{color}

\usepackage{hyperref}
\usepackage[capitalise]{cleveref}
\usetikzlibrary{patterns}
\usetikzlibrary{intersections}
\usetikzlibrary{arrows,positioning}
\usepackage{graphicx}
\usepackage{pgfplots}
 \pgfplotsset{compat=1.14}
\usetikzlibrary{plotmarks}
  \usepgfplotslibrary{patchplots}
  \usepackage{grffile}
\usepgfplotslibrary{fillbetween}
\usepackage{float}
\usepackage{caption}
\usepackage{subcaption}

\newcommand*{\rom}[1]{\expandafter\@slowromancap\romannumeral #1@}

\newcommand{\gettikzxy}[3]{%
  \tikz@scan@one@point\pgfutil@firstofone#1\relax
  \edef#2{\the\pgf@x}%
  \edef#3{\the\pgf@y}%
}

\def\namedlabel#1#2{\begingroup
	#2%
	\def\@currentlabel{#2}%
	\phantomsection\label{#1}\endgroup
}

\newcommand{\sB}{\mathcal B}
\newcommand{\sC}{\mathcal P}
\newcommand{\sD}{\mathcal D}
\newcommand{\sE}{\mathcal E}
\newcommand{\sF}{\mathcal F}

\newcommand{\sI}{\mathcal I}
\newcommand{\sL}{\mathcal L}
\newcommand{\sM}{\mathcal M}
\newcommand{\sN}{\mathcal N}
\newcommand{\sP}{\mathcal P}

\newcommand{\sU}{\mathcal U}

\newcommand{\R}{\mathbb R}
\newcommand{\E}{\mathbb E}

\newcommand{\F}{\mathbb F}

\newcommand{\Prob}{\mathbb P}

\newcommand{\argsup}{\mbox{argsup}}
\newcommand{\Leb}{\mbox{Leb}}
\newtheorem{thm}{Theorem}[section]
\newtheorem{prop}{Proposition}[section]
\newtheorem{eg}{Example}[section]
\newtheorem{lem}{Lemma}[section]
\newtheorem{cor}{Corollary}[section]
\newtheorem{rem}{Remark}[section]

\newtheorem{defn}{Definition}[section]

\usepackage[foot]{amsaddr}
\makeatletter
\renewcommand{\email}[2][]{%
	\ifx\emails\@empty\relax\else{\g@addto@macro\emails{,\space}}\fi%
	\@ifnotempty{#1}{\g@addto@macro\emails{\textrm{(#1)}\space}}%
	\g@addto@macro\emails{#2}%
}
\makeatother
\makeatletter
\@namedef{subjclassname@2020}{%
	\textup{2020} Mathematics Subject Classification}
\makeatother
\numberwithin{equation}{section}

\begin{document}
\title[Explicit construction of the increasing supermartingale coupling]{A potential-based construction of the increasing supermartingale coupling}
\author{Erhan Bayraktar}
\email{erhan@umich.edu}
\author{Shuoqing Deng}
\email{shuoqing@umich.edu}
\author{Dominykas Norgilas}
\email{dnorgila@umich.edu}
\address{Department of Mathematics, University of Michigan}
\thanks{E. Bayraktar is partially supported by the National Science Foundation under grant  DMS-2106556 and by the Susan M. Smith chair.} 

\keywords{Couplings, supermartingales, optimal transport}
\subjclass[2020]{Primary: 60G42; Secondary: 49N05.}


\begin{abstract}
The increasing supermartingale coupling, introduced by Nutz and Stebegg (Canonical supermartingale couplings, Ann. Probab., 46(6):3351--3398, 2018) is an extreme point of the set of `supermartingale' couplings between two real probability measures in convex-decreasing order. In the present paper we provide an explicit construction of a triple of functions, on the graph of which the increasing supermartingale coupling concentrates. In particular, we show that the increasing supermartingale coupling can be identified with the left-curtain martingale coupling and the antitone coupling to the left and to the right of a uniquely determined regime-switching point, respectively.

Our construction is based on the concept of the \textit{shadow} measure. We show how to determine the potential of the shadow measure associated to a supermartingale, extending the recent results of Beiglb\"ock et al. (The potential of the shadow measure, Electron. Commun. Probab., 27, paper no. 16, 1--12, 2022) obtained in the martingale setting.
\end{abstract}
\maketitle
\tableofcontents

\section{Introduction}
The joint distribution $\pi$ of two real-valued random variables $X$ and $Y$ is called a supermartingale coupling if $\E^\pi[Y\lvert X]\leq X$. The classical result by Strassen \cite{Strassen:65} states that, for two probability measures $\mu$ and $\nu$ on $\R$, the set of supermartingale couplings of $X\sim\mu$ and $Y\sim\nu$ is non-empty if and only if $\mu\leq_{cd}\nu$, i.e., $\mu$ is smaller than $\nu$ with respect to the convex-decreasing order. The natural question is then whether there is any canonical choice to couple $\mu$ and $\nu$. Nutz and Stebegg \cite{NutzStebegg:18} recently introduced the \textit{increasing} supermartingale coupling, denoted by $\pi_I$, and proved that it is canonical in several ways.

First, $\pi_I$ solves the supermartingale optimal transport (SOT) problem for a class of cost functions $c:\R\times\R\mapsto\R$ (essentially those $c$ that are such that $c(x_2,\cdot)-c(x_1,\cdot)$ is strictly decreasing and strictly convex for all $x_1<x_2$):
$$
\textrm{minimise}\quad\E^\pi[c(X,Y)]\quad\textrm{subject to }X\sim\mu,Y\sim\nu,
$$
where the infimum is taken over all supermartingale couplings $\pi$ (in this context the couplings $\pi$ are often called \textit{transport plans}). Second, the optimality of $\pi_I$ is closely linked to the monotonicity properties of its support. In particular, $\pi_I$ is a unique supermartingale coupling whose support is both, first-order right-monotone (see Definition \ref{def:forightmon}) and second-order left monotone (see Definition \ref{def:lmon}). Finally, the increasing supermartingale coupling has one further, order-theoretic, characterisation: it is canonical with respect to the convex-decreasing order. More precisely, for a transport plan $\pi$ and any real number $t$ denote by $\nu^\pi_t$ the terminal law of $\mu\lvert_{(-\infty,t]}$ within $\nu$ when a coupling $\pi$ is used. Then $\pi_{I}$ is such that, for each $t$, $\nu^{\pi_{I}}_t\leq_{cd}\nu^\pi_t$ for all supermartingale couplings $\pi$.

\textbf{Literature Review.} From the optimal transportation point of view, the SOT is a classical Monge-Kantorovich optimal transport (OT) problem with an additional moment constraint. On the other hand, the basic martingale optimal transport (MOT) problem (introduced by Beiglb\"{o}ck et al. \cite{BeiglbockHenryLaborderePenkner:13} (in a discrete time setting) and Galichon et al. \cite{GalichonHenryLabordereTouzi:14} (in continuous time), and first solved by Hobson and Neuberger~\cite{HobsonNeuberger:12} and Hobson and Klimmek~\cite{HobsonKlimmek:15} for the specific cost functions $c(x,y)=-|y-x|$ and $c(x,y)=|y-x|$, respectively) is to construct a martingale $M$, with $M_1\sim\mu$ and $M_2\sim\nu$, and such that $\E[c(M_1,M_2)]$ is minimal. Since the martingale requirement can be expressed in terms of couplings $\pi$ satisfying $\E^\pi[M_2\lvert M_1]=M_1$, the SOT problem can be also seen as a relaxation of an MOT problem. Given the connectedness of the aforementioned three variations of the transportation problem, we in fact have that they all share a similar theory.

It is a well-known fact in the classical OT setting that the support of an optimal coupling is a $c$\textit{-cyclically monotone} set (see, for example, Villani \cite{villani:08}). Furthermore, if one considers the cost functions $c$ that can be represented as $c(x,y)=h(y-x)$ for a strictly convex $h$ (the so-called Spence-Mirrlees condition), then the optimal coupling $\hat\pi$ is canonical with respect to the first-order stochastic dominance. In particular, $\hat\pi=\pi_{HF}$, where $\pi_{HF}$ is the so-called Hoeffding-Fr\'{e}chet (or \textit{quantile}) coupling, and we have that $\nu^{\pi_{HF}}_t$ is the \textit{left-most} measure within $\nu$ and with total mass $\mu\lvert_{(-\infty,t]}(\R)$. 

In the martingale setting, Beiglb\"{o}ck and Juillet \cite{BeiglbockJuillet:16} introduced the \textit{left-curtain} coupling $\pi_{lc}$ that can be viewed as a martingale counterpart to the monotone quantile coupling $\pi_{HF}$. Some notable similarities are that $\pi_{lc}$ can also be described via three different characterizations: order-theoretic, optimality, monotonicity of the support. First, $\pi_{lc}$ is canonical with respect to the convex order, denoted by $\leq_c$: for each $t$, $\nu^{\pi_{lc}}_t\leq_{c}\nu^\pi_t$  for any martingale coupling $\pi$. Beiglb\"{o}ck and Juillet \cite{BeiglbockJuillet:16} also showed that the left-curtain coupling is optimal for a range of different cost functions. Later Henry-Labord\`{e}re and Touzi~\cite{HenryLabordereTouzi:16} extended their result and proved that $\pi_{lc}$ is optimal for even larger class of cost functions, namely those satisfying the martingale Spence-Mirrlees condition $c_{xyy}>0$. Finally $\pi_{lc}$ can be characterized by its support: it is a unique second-order left-monotone martingale coupling. Several other authors further investigate the properties and extensions of the left-curtain coupling, see Beiglb\"{o}ck et al. \cite{{BeiglbockHenryLabordereTouzi:17},{BeiglbockCox:17},BHN:20}, Juillet \cite{{Juillet:16},Juillet:18}, Hobson and Norgilas \cite{HobsonNorgilas:17}, Nutz et al. \cite{{NutzStebegg:18},NutzStebeggTan:17}, Campi et al. \cite{Campi:17}, Henry-Labord\`{e}re et al.~\cite{HenryLabordereTanTouzi:16} and Br\"{u}ckerhoff at al.~\cite{BHJ:20}.

We will study one further characterisation of the increasing supermartingale coupling $\pi_I$, which is also satisfied by $\pi_{HF}$ and $\pi_{lc}$ in their respective settings. The fundamental result in the theory of OT is Brenier's theorem (see Brenier \cite{Brenier:87} and R\"{u}schendorf and Rachev \cite{RuRach:90}). It considers the optimal transport problem (in $\R^d$ and) in the particular case $c(x,y)=\lvert x-y\lvert^2$, where $\lvert\cdot\lvert$ denotes the Euclidean norm on $\R^d$. Under some regularity conditions on the initial measure $\mu$, the optimal coupling is supported by the graph of the gradient of a convex function. In dimension one, the supporting function is monotonically increasing and the optimal coupling coincides with the quantile coupling $\pi_{HF}$. More precisely, $\pi_{HF}$ is supported on the graph of $G_\nu\circ F_\mu$, where $G_\nu$ is the quantile function of $\nu$ while $F_\mu$ is the cumulative distribution function of $\mu$. In the martingale setting, Beiglb\"{o}ck and Juillet \cite{BeiglbockJuillet:16} established the Brenier-type result for the left-curtain coupling as well. Given that the initial measure $\mu$ does not contain atoms, the authors showed that $\pi_{lc}$ is supported by the graphs of two functions $T_u,T_d:\R\mapsto\R$ satisfying certain monotonicity properties (see Definition \ref{def:lmonfns}). While the result of Beiglb\"{o}ck and Juillet \cite{BeiglbockJuillet:16} is purely an existence result, Henry-Labord\`{e}re and Touzi~\cite{HenryLabordereTouzi:16} used an ordinary differential equation approach and showed how to explicitly determine $T_d$ and $T_u$ under some further assumptions on $\mu$ and $\nu$ of technical nature. The most general result regarding the functional representation of $\pi_{lc}$ is due to Hobson and Norgilas \cite{HobsonNorgilas:21}. The authors showed (for arbitrary $\mu$ and $\nu$) how to recover the property that the left-curtain coupling is supported on a graph of two (explicitly constructed) functions, provided that we generalise the notion of a coupling.

\textbf{Our contribution.} In the present paper, our main effort is dedicated to proving the following Brenier-type result that provides the functional representation of the (generalised) increasing supermartingale coupling $\pi_{I}$. 
\begin{thm}
	\label{thm:mainTHM_intro}
	Let $(\Omega,\sF,\Prob) = ((0,1) \times (0,1), \sB(\Omega), \Leb(\Omega))$. Let $\omega = (u,v)$ and let $(U,V)$ be the canonical random variables on $(\Omega,\sF,\Prob)$ given by $(U(\omega),V(\omega))=(u,v)$ so that $U$ and $V$ are independent $U(0,1)$ random variables. Let $\F = (\sF_1 = \sigma(U), \sF_2 = \sigma(U,V) \})$ and set $\mathbf{S} = (\Omega, \sF, \F, \Prob)$.
	
	Fix $\mu\leq_{cd}\nu$ and let $G=G_\mu$ be a quantile function of $\mu$.
	
	Then there exists $u^*\in[0,1]$ and a triple of functions $R,S:(0,u^*]\mapsto\R$ and $T:(u^*,1)\mapsto\R$ such that
	\begin{itemize}
		\item $R\leq G\leq S$ on $\sI=(0,u^*]$, $S$ is non-decreasing on $\sI$, $R(u')\notin (R(u),S(u))$ for all $u,u'\in\sI$ with $u<u'$,
		\item $T<G$ on $\hat{\sI}=(u^*,1)$, $T$ is non-increasing on $\hat{\sI}$, $T(u')\notin (R(u),S(u))$ for all $u\in\sI$ and $u'\in\hat{\sI}$.
	\end{itemize}
	Furthermore, if we define $X(u,v)=X(u)=G(u)$ and $Y(u,v) \in \{R(u),S(u),T(u) \}$ by 
	\begin{equation*}
	\begin{aligned}
	Y(u,v) &= I_{\{u\leq u^*\}\cap\{R(u)=S(u)\}}G(u)\\&+I_{\{u\leq u^*\}\cap\{R(u)<S(u)\}}\Big\{R(u) I_{\{ v \leq \frac{S(u) - G(u)}{S(u)-R(u)} \}} +  S(u) I_{ \{ v > \frac{S(u) - G(u)}{S(u)-R(u)} \} }\Big\}\\ &+ I_{\{u>u^*\}}T(u),
	\end{aligned}
	\end{equation*}
	then
	$S = (X(U),Y(U,V))$ is a $\mathbf{S}$-supermartingale for which $\sL(X) = \mu$ and $\sL(Y) = \nu$.
\end{thm}

We will prove Theorem \ref{thm:mainTHM_intro} by explicitly constructing the triple of supporting functions $(R,S,T)$, see Figure \ref{fig:RGSfg}. Furthermore we will determine a unique regime-switching point $u^*\in[0,1]$ such that, to the left of $G_\mu(u^*)$, $\pi_I$ coincides with the left-curtain martingale coupling and concentrates on $R$ and $S$, while to the right of $G(u^*)$, $\pi_I$ concentrates on a deterministic decreasing map $T$, and thus corresponds to the classical \textit{antitone} coupling $\pi_{AT}$, which is the symmetric counterpart of $\pi_{HF}$ (if $\pi_{HF}$ concentrates on $G_\nu\circ F_\mu$, then $\pi_{AT}$ concentrates on $G_\nu\circ(1-F_\mu)$). Finally, Nutz and Stebegg \cite{NutzStebegg:18} showed that the set $\mathbb M$ appearing in the definition of the first-order right-monotonicity of $\pi_I$ (see Definition \ref{def:forightmon}) is such that $\pi_I\lvert_{\mathbb M\times\R}$ is a martingale. From or construction it will follow that the set of `martingale points' $\mathbb M$ is in fact an interval $(-\infty,G(u^*)]$ (and extra care will be needed if $\mu$ has an atom at $G_\mu(u^*)$).
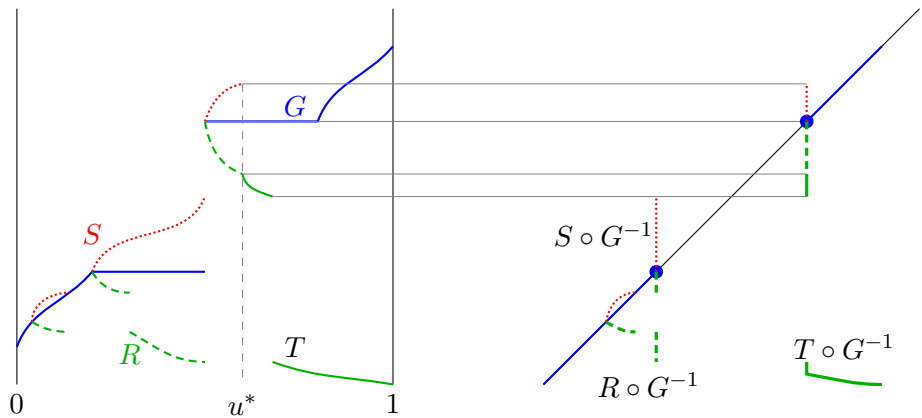
\begin{figure}[H]
	\centering
	\begin{tikzpicture}[scale=1,
	declare function={	
		k1=2.1;
		k2=1;
		a(\x)=((k1-\x)*(\x<k1))-k1-1;
		b(\x)=((k2-\x)*(\x<k2))-k1-1;
		x1=-6.8;
		x2=-6.3;
		z1=-5.2;
		z2=-4.7;
	}]

	\draw[ black] (-7,0)--(-7,5);
	\draw[black] (-2,0)--(-2,5);
	
	\draw[name path=diag, black] (0,0) -- (5,5);

	\draw[blue,thick, name path=g1] (-7,0.5) to[out=70, in=230] (-6,1.5) -- (-4.5,1.5) ;
	\draw[blue,thick, name path=g2] (-4.5,3.5) -- (-3,3.5) to[out=70, in=230] (-2,4.5)  ;
	
	\path [name path=lineA](x1,0) -- (x1,5);
	\path [gray, very thin, name intersections={of=lineA and g1}] (x1,0) -- (intersection-1);
	\coordinate (temp1) at (intersection-1);

	\path [name path=lineA](x2,0) -- (x2,5);
	\path [gray, very thin, name intersections={of=lineA and g1}] (x2,0) -- (intersection-1);
	\coordinate (temp2) at (intersection-1);
	
	\draw[red,thick, densely dotted,name path=s1] (temp1) to[out=90, in=180] (temp2) ;
	\draw[red,thick, densely dotted,name path=s2] (-6,1.5) to[out=70, in=250] (-4.5,2.5) ;
	\draw[red,thick, densely dotted,name path=s3] (-4.5,3.5) to[out=70, in=190] (-4,4) ;
	
	\gettikzxy{(temp1)}{\l}{\k};
	\gettikzxy{(temp2)}{\n}{\m};
	\draw[black!30!green,thick, densely dashed, name path=r1] (temp1) to[out=330, in=180] (\n,0.7) ;	
	\draw[black!30!green,thick, densely dashed, name path=r2] (-6,1.5) to[out=300, in=180] (-5.5,\m) ;
	\draw[black!30!green,thick, densely dashed, name path=r3] (-5.5,0.7) to[out=330, in=180] (-4.5,0.3) ;
	\draw[black!30!green,thick, densely dashed, name path=r4] (-4.5,3.5) to[out=280, in=160] (-4,2.8) ;
	\draw[black!30!green,thick, name path=r4] (-4,2.8) to[out=280, in=160] (-3.6,2.5) ;
	\draw[black!30!green,thick, name path=r5] (-3.6,0.3) to[out=340, in=170] (-2,0) ;
	
	\draw[gray, very thin, name path=r4] (-4,2.8) -- (3.5,2.8) ;
	

	\path[gray, very thin] (\n,0.7) -- (temp2);
	\path[gray, very thin] (\n,0.7) -- (0.7,0.7);
	\path[gray, very thin]  (temp2)--(\m,\m);
	\path[gray, very thin]  (-2,4.5)--(4.5,4.5);

	\path [name path=lineA](-5.5,0.7) -- (-5.5,5);
	\path [gray, very thin, name intersections={of=lineA and s2}] (-5.5,0.7) -- (intersection-1);
	\coordinate (vs2) at (intersection-1);
	\gettikzxy{(vs2)}{\aa}{\bb}
	\path[gray, very thin]  (vs2)--(\bb,\bb);

	\path[gray, very thin] (-4.5,0.3) -- (-4.5,3.5);
	\draw[gray, very thin] (-4.5,3.5) -- (3.5,3.5);
	\path[gray, very thin] (-4.5,0.3) -- (0.3,0.3);
	
	\path[gray, very thin] (-3.6,0.3) -- (-3.6,2.5) ;
	\draw[gray, very thin] (-3.6,2.5) -- (3.5,2.5) ;
	\draw[gray, very thin] (-4,4) -- (3.5,4) ;
	\draw[gray, very thin, dashed] (-4,4) -- (-4,0) ;
	
	\path [name path=lineA](-3.6,2.5) -- (-3.6,5);
	\coordinate (test1) at (intersection-1);
	\gettikzxy{(test1)}{\testx}{\testy}
	
	\path [name path=lineA](-3,3.5) -- (-3,5);
	\path [gray, very thin, name intersections={of=lineA and s3}] (-3,3.5) -- (intersection-1);
	\coordinate (vs3) at (intersection-1);
	\gettikzxy{(vs3)}{\aaa}{\bbb}
	\path[gray, very thin]  (vs3)--(\bbb,\bbb);

	\path [name path=lineA](-3,3.5) -- (-3,0);
	\path [gray, very thin, name intersections={of=lineA and r5}] (-3,3.5) -- (intersection-1);
	\coordinate (vr5) at (intersection-1);
	\gettikzxy{(vr5)}{\aaaa}{\bbbb}
	\path[gray, very thin]  (vr5)--(\bbbb,\bbbb);
	
	\path[gray, very thin]  (vs2)--(\bb,\bb);
	\path[gray, very thin]  (vs2)--(\bb,\bb);
	
	\node [scale=0.5, shape=circle, fill, blue] at (3.5,3.5) {} ;
	\node [scale=0.5, shape=circle, fill, blue] at (1.5,1.5) {} ;

	\draw[blue,thick] (0,0) -- (1.5,1.5);
	
	\draw[blue,thick] (3.5,3.5) -- (4.5,4.5);
	
	\draw[red,densely dotted, thick] (\k,\k) to[out=80, in=180] (\m,\m);
	
	\draw[red,densely dotted, thick] (1.5,1.5) -- (1.5,2.5);
	
	\draw[red,densely dotted, thick] (3.5,3.5) -- (3.5,4);

	\draw[black!30!green, densely dashed,very thick] (\k,\k) to[out=330, in=180] (\m,0.7);
	\draw[black!30!green,densely dashed,very thick] (1.5,1.5) -- (1.5,\m);
	\draw[black!30!green,densely dashed,very thick] (1.5,0.7) -- (1.5,0.3);
	
	\draw[black!30!green,densely dashed,very thick] (3.5,3.5) -- (3.5,2.8);
	\draw[black!30!green,very thick] (3.5,2.8) -- (3.5,2.5);
	
	\draw[black!30!green,very thick] (3.5,0.3) -- (3.5,\bbbb) to[out=350, in=180] (4.5,0);

	\node[red] at (-6,2) {$S$};
	\node[blue] at (-3.3,3.7) {$G$};
	\node[] at (-3.3,0.5) {$T$};

	\node[black!30!green] at (-5.5,0.4) {$R$};
	
	\node[below] at (-7,0) {$0$};
	\node[below] at (-2,0) {$1$};
	\node[below] at (-4,0) {$u^*$};
	\node[black] at (4,0.5) {$T \circ G^{-1}$};
	\node[black] at (0.8,2) {$S \circ G^{-1}$};
	\node[black] at (1.4,0) {$R \circ G^{-1}$};
	\end{tikzpicture}
	\caption{Sketch of $R,G,S$ and $T$ in the case $u^*\in(0,1)$. For each $u\in (0,u^*]$ the mass at $G(u)$ either remains at $G(u)$ or it splits and is mapped either to $R(u)$ or to $S(u)$, while for each $u\in(u^*,1)$ the mass at $G(u)$ is mapped to a single point located at $T(u)<G(u)$. On the atoms of $\mu$, $G$ is flat, and $R \circ G^{-1}$, $S \circ G^{-1}$ and $T \circ G^{-1}$ are multi-valued, but $R,S$ and $T$ remain well-defined.}
	\label{fig:RGSfg}
\end{figure}

There are two special cases of Theorem \ref{thm:mainTHM_intro}, namely, $u^*=0$ and $u^*=1$. When $u^*=1$, then we have that, for the given $\mu,\nu$ with $\mu\leq_{cd}\nu$, the set of supermartingale couplings coincides with the set of martingale couplings. In this case the function $T$ is irrelevant. In particular, our construction then corresponds to the generalised, or lifted, left-curtain martingale coupling. Hobson and Norgilas \cite{HobsonNorgilas:21} constructed $R$ and $S$ for (generalised version of) $\pi_{lc}$ on a single `irreducible' component only. Our construction, on the other hand, does not place any irreducibility conditions. The second special case is when $u^*=0$, so that the functions $R$ and $S$ do not play any role. In particular, the generalised increasing supermartingale coupling then concentrates on the deterministic decreasing map $T$ and we have that  $\pi_I=\pi_{AT}$. To achieve this we will show that $u^*=0$ if and only if the support of $\mu$ is (strictly) to the right of the support of $\nu$, and thus no part of $\mu$ can be embedded into $\nu$ using a martingale. 

Nutz and Stebegg \cite{NutzStebegg:18} introduced the notion of positive convex-decreasing order of two measures, denoted by $\leq_{pcd}$, which compares measures of possibly different total mass. If a pair of measures $\mu,\nu$ is such that $\mu\leq_{pcd}\nu$, then the set of measures $\eta$ with $\mu\leq_{cd}\eta\leq\nu$ is non-empty, and each such $\eta$ corresponds to a terminal law of a supermartingale that embeds $\mu$ into $\nu$. Nutz and Stebegg \cite{NutzStebegg:18} also proved that there exists a canonical choice of such $\eta$ with respect to $\leq_{cd}$: the shadow of $\mu$ in $\nu$, denoted by $S^\nu(\mu)$, is a unique measure such that $S^\nu(\mu)\leq_{cd}\eta$ for all $\eta$ satisfying $\mu\leq_{cd}\eta\leq\nu$. Our interest in the shadow measure lies in the fact that the increasing supermartingale coupling can be defined as a unique measure $\pi_I$ on $\R^2$ such that, for each $x\in\R$, $\pi_I\lvert_{(-\infty,x]\times\R}$ has
the first marginal $\mu\lvert_{(-\infty,x]}$ and the second marginal $S^\nu(\mu\lvert_{(-\infty,x]})$. One of our main achievement is that, for arbitrary $\mu$ and $\nu$ with $\mu\leq_{pcd}\nu$, we are able to explicitly construct the potential function of the shadow measure (and then the shadow measure itself can be identified as the second derivative of the potential function in the sense of distributions). This generalises the recent results of Beiglb\"{o}ck et al. \cite{BHN:20}, where the authors showed how to construct the potential of the shadow measure in the martingale setting. Surprisingly, in both, supermartingale and martingale cases, the potential of the shadow measure has the same functional representation.

The ability to explicitly determine the shadow measure will be the key asset in constructing the triple $(R,S,T)$ that supports the increasing supermartingale coupling. In particular, for each $x\in\R$, the graph of the potential function of $S^\nu(\mu\lvert_{(-\infty,x]})$ will suggest a candidate locations to which the mass of $\mu$ at $x$ should be mapped.
The remarkable property of the supermartingale shadow measure is that it is able to determine both, the left-curtain martingale coupling in the case $u^*=1$ and also the antitone coupling in the case $u^*=0$. This is of independent interest.

The paper is structured as follows. In Section \ref{sec:prelims} we discuss the relevant notions of probability measures and (positive) convex-decreasing order, and some important (for our main theorems) results regarding the convex hull of a function. In Section \ref{sec:shadowLC} we introduce the supermartingale shadow measure and the increasing supermartingale coupling. Section \ref{sec:Slc} is dedicated to our main results. In Section \ref{sec:u^*} we first determine the regime-switching point $u^*\in[0,1]$. Then in Section \ref{sec:Constr} we prove Theorem \ref{thm:mainTHM_intro}, first in the case $u^*=1$, then we cover the case $u^*=0$, and finally we prove the general case $u^*\in(0,1)$. Some proofs are deferred until the appendix.
\section{Preliminaries}\label{sec:prelims}
\subsection{Measures and Convex order}
\label{prelims:convex}
The law of a random variable $X$ will be denoted by $\sL(X)$.

Let $\sM$ (respectively $\sP$) be the set of measures (respectively probability measures) on $\R$ with finite total mass and finite first moment, i.e., if $\eta\in\sM$, then $\eta(\R)<\infty$ and $\int_\R\lvert x\lvert\eta(dx)<\infty$. Given a measure $\eta\in\sM$ (not necessarily a probability measure), define $\overline{\eta} = \int_\R x \eta(dx)$ to be the first moment of $\eta$ (and then $\overline{\eta}/\eta(\R)$ is the barycentre of $\eta$). Let $\sI_\eta$ be the smallest interval containing the support of $\eta$, and let $\{ \ell_\eta, r_\eta \}$ be the endpoints of $\sI_\eta$. If $\eta$ has an atom at $\ell_\eta$ then $\ell_\eta$ is included in $\sI_\eta$, and otherwise it is excluded, and similarly for $r_\eta$. 

For $\eta\in\sM$, by $F_\eta:\R\mapsto[0,\eta(\R)]$ we denote the right-continuous cumulative distribution function of $\eta$. Let $G_\eta:(0,\eta(\R))\mapsto\R$ be a quantile function of $\eta$, i.e., a generalised inverse of $F_\eta$. There are two canonical versions of $G_\eta$: the left-continuous and right-continuous versions correspond to $G_\eta^-(u)=\sup\{k\in\R: F_\eta(k)<u\}$ and $G^+_\eta(u)=\inf\{k\in\R: F_\eta(k) > u\}$, for $u\in(0,\eta(\R))$, respectively. However any $G$ with $G_\eta^-(u)\leq G(u)\leq G^+_\eta(u)$, for all $u\in(0,\eta({\R})$, is still called a quantile function of $\eta$, which is motivated by the fact that for any such $G$ we have that $\sL(G(U))=\eta$, where $U\sim U[0,1]$. (Note that $G_\eta$ may take values $-\infty$ and $\infty$ at the left and right end-points of $[0,\eta(\R)]$, respectively.)

For $\alpha \geq 0$ and $\beta \in \R$ let $\sD(\alpha, \beta)$ denote the set of non-negative, non-decreasing and convex functions $f:\R \mapsto \R_+$ such that
\[ \lim_{ z \downarrow -\infty}  \{ f(z) \} =  0, \hspace{10mm} \lim_{z \uparrow \infty} \{ f(z) - (\alpha z- \beta) \}   =0. \]
Then, when $\alpha = 0$, $\sD(0,\beta)$ is empty unless $\beta = 0$ and then $\sD(0,0)$
contains one element, the zero function.

For $\eta\in\sM$, define the functions $P_\eta,C_\eta : \R \mapsto \R^+$ by
\begin{equation*}
P_\eta(k) := \int_{\R} (k-x)^+ \eta(dx),\quad k\in\R,\hspace{10mm}C_\eta(k) := \int_{\R} (x-k)^+ \eta(dx),\quad k\in\R,
\end{equation*}
respectively. Then $P_\eta(k) \geq 0 \vee  (\eta(\R) k - \overline{\eta} )$ and $C_\eta(k) \geq 0 \vee (\overline{\eta} - \eta(\R)k)$. Also $C_\eta(k) - P_\eta(k) = (\overline{\eta}- \eta(\R)k)$.

The following properties of $P_\eta$ can be found in Chacon~\cite{Chacon:77}, and Chacon and Walsh~\cite{ChaconWalsh:76}: $P_\eta \in \sD(\eta(\R), \overline{\eta})$ and $\{k : P_{\eta}(k) > (\eta(\R)k - \overline{\eta})^+  \} =  \{k : C_{\eta}(k) > (\overline{\eta}- \eta(\R)k)^+\} =(\ell_\eta,r_\eta)$. Conversely (see, for example, Proposition 2.1 in Hirsch et al. \cite{peacock}),  if $h \in \sD(k_m,k_f)$ for some numbers $k_m \geq 0$ and $k_f\in\R$ (with $k_f = 0$ if $k_m=0$), then there exists a unique measure $\eta\in\sM$, with total mass $\eta(\R)=k_m$ and first moment $\overline{\eta}=k_f$, such that $h=P_{\eta}$. In particular, $\eta$ is uniquely identified by the second derivative of $h$ in the sense of distributions. Furthermore, $P_\eta$ and $C_\eta$ are related to the potential $U_\eta$, defined by
\begin{equation*}
U_\eta(k) : =  - \int_{\R} |k-x| \eta(dx),\quad k\in\R,
\end{equation*}
by $-U_\eta=C_\eta+P_\eta$. We will call $P_\eta$ (and $C_\eta$) a modified potential. Finally note that all three second derivatives $C^{\prime\prime}_{\eta},P^{\prime\prime}_{\eta}$ and $-U_\eta^{\prime\prime}/2$ identify the same underlying measure  $\eta$.

For $\eta,\chi\in\sM$, let $\sD_\eta:=\{\sD(\eta(\R),\bar\eta-q):q\in\R_+\}$ and 
$$
\sC(\eta,\chi):=\{\tilde{P} \in \sD_\eta:P_\chi-\tilde{P}\textrm{ is convex and }P_{\eta}\leq \tilde P\}.
$$

For $\eta,\chi\in\sM$, we write $\eta\leq\chi$ if $\eta(A) \leq \chi(A)$ for all Borel measurable subsets $A$ of $\R$, or equivalently if
\begin{equation*}
\int fd\eta\leq\int fd\chi,\quad \textrm{for all non-negative }f:\R\mapsto\R_+.
\end{equation*}
Since $\eta$ and $\chi$ can be identified as second derivatives of $P_\chi$ and $P_\eta$ respectively, we have $\eta\leq\chi$ if and only if $P_\chi-P_\eta$ is convex, i.e., $P_\eta$ has a smaller curvature than $P_\chi$.

Two measures $\eta,\chi\in\sM$ are in convex (resp. convex-decreasing) order, and we write $\eta\leq_c\chi$ (resp. $\eta \leq_{cd} \chi$), if
\begin{equation}\label{eq:cd}
\int fd\eta\leq\int fd\chi,\quad\textrm{for all convex (resp. convex and non-increasing) }f:\R\mapsto\R.
\end{equation}
Since we can apply \eqref{eq:cd} to all constant functions, including $f\equiv -1$ and $f\equiv1$, we have that if $\eta\leq_{c}\chi$ or $\eta\leq_{cd}\chi$ then $\eta(\R)=\chi(\R)$. On the other hand, applying \eqref{eq:cd} to $f(x)=-x$ gives that $\bar{\eta}\leq\bar{\chi}$. However, a reversed inequality holds only in the case $\eta\leq_c\chi$ (since $f(x)=x$ is strictly increasing).

For $\eta,\chi\in\sP$, let ${\Pi}(\eta,\chi)$ be the set of probability measures on $\R^2$ with the first marginal $\eta$ and second marginal $\chi$. Let ${\Pi}_S(\eta,\chi)$ be the set of supermartingale couplings of $\eta$ and $\chi$.
Then
\begin{equation*}
{\Pi}_S(\eta,\chi) = \big\{ \pi \in {\Pi}(\eta,\chi) : \mbox{\eqref{eq:martingalepi} holds} \big\},
\end{equation*}
where \eqref{eq:martingalepi} is the supermartingale condition
\begin{equation}
\int_{x \in B} \int_{y \in \R} y \pi(dx,dy) \leq \int_{x \in B} \int_{y \in \R} x \pi(dx,dy) = \int_B x \eta(dx), \quad
\mbox{$\forall$ Borel $B \subseteq \R$}.
\label{eq:martingalepi}
\end{equation}
Equivalently, $\Pi_S(\eta,\chi)$ consists of all transport plans $\pi$ (i.e., elements of ${\Pi}(\eta,\chi)$) such that the disintegration in probability measures $(\pi_x)_{x\in\R}$ with respect to $\eta$ satisfies $\int_\R y\pi_x(dy)\leq x$ for $\eta$-almost every $x$.

The following is classical (see, for example, F\"{o}llmer and Schied \cite[Theorem 2.58]{FS:16}).
\begin{lem}\label{lem:put_cd}
Let $\eta,\chi\in\sP$. The following are equivalent:\begin{enumerate}
	\item $\eta\leq_{cd}\chi$,
	\item $\eta(\R)=\chi(\R)$ and $P_\eta\leq P_\chi$ on $\R$,
	\item $\Pi_S(\eta,\chi)\neq\emptyset$.
\end{enumerate}
\end{lem}

\begin{rem}\label{rem:cVScd}
If $\eta,\chi\in\sP$ with $\eta\leq_{cd}\chi$, but $\bar\eta=\bar\chi$, then $\Pi_S(\eta,\chi)$ reduces to the set of martingale couplings, denoted by $\Pi_M(\eta,\chi)$ (i.e., elements of $\Pi(\eta,\chi)$ for which \eqref{eq:martingalepi} holds with equality). Indeed, any supermartingale with constant mean is a martingale. In this case $\eta\leq_c\chi$ (see Strassen~\cite{Strassen:65}).
\end{rem}

For our purposes in the sequel we need a generalisation of the convex (resp. convex-decreasing) order of two measures. We follow Beiglb\"{o}ck and Juillet~\cite{BeiglbockJuillet:16} (resp. Nutz and Stebegg \cite{NutzStebegg:18}) and say $\eta,\chi\in\sM$ are in a \textit{positive} convex (resp. \textit{positive} convex-decreasing) order, and write $\eta\leq_{pc}\chi$ (resp. $\eta\leq_{pcd}\chi$), if $\int fd\eta\leq\int fd\chi$, for all non-negative and convex (resp. non-negative, convex and non-increasing) $f:\R\mapsto\R_+$.
If $\eta\leq_{c}\chi$ (resp. $\eta\leq_{cd}\chi$) then also $\eta\leq_{pc}\chi$ (resp. $\eta\leq_{pcd}\chi$), since non-negative and convex (resp. non-negative, convex and non-increasing) functions are convex (resp. convex and non-increasing). If $\eta\leq\chi$ then both,  $\eta\leq_{pc}\chi$ and $\eta\leq_{pcd}\chi$, since non-negative, convex and non-increasing functions are non-negative and convex, and thus also non-negative. Note that, if $\eta\leq_{pc}\chi$ or $\eta\leq_{pcd}\chi$, then $\eta(\R)\leq\chi(\R)$ (apply the function $f(x)\equiv1$ in the definition of $\leq_{pc}$ and $\leq_{pcd}$). On the other hand, if $\eta(\R)=\chi(\R)$, then $\eta\leq_{pc}\chi$ (resp. $\eta\leq_{pcd}\chi$) is equivalent to $\eta\leq_{c}\chi$ (resp. $\eta\leq_{cd}\chi$).
\begin{eg}\label{eg1}
	Let $\eta,\chi\in\sM$ with $\eta\leq_{cd}\chi$ (resp. $\eta\leq_{c}\chi$). Fix a Borel set $B\subseteq\R$, and let $\eta\lvert_B\in\sM$ be a restriction of $\eta$ to $B$. Then $\eta\lvert_B\leq_{pcd}\chi$ (resp. $\eta\lvert_B\leq_{pc}\chi$).
\end{eg}
The following result allows us to verify (in a special case of disjoint supports) whether $\eta,\chi\in\sM$ satisfy $\eta\leq_{pc}\chi$. (The proof is presented in the appendix.)
	\begin{lem}\label{lem:disjointSupport}
		Let $\eta,\chi\in\sM$ be such that $\eta$ and $\chi$ have disjoint supports.  If
		$$0<\eta(\R)\leq	(\chi((-\infty,\ell_\eta])\wedge\chi([r_\eta,\infty)))
		$$
		then $\eta\leq_{pc}\chi$.		
	\end{lem}

For a pair of measures $\eta,\chi\in\sM$, let the function $ D_{\eta,\chi}: \R \mapsto \R$ be defined by $D_{\eta,\chi}(k) = P_\chi(k) - P_\eta(k)$. Note that if $\eta,\chi$ have equal mass then $\eta \leq_{cd} \chi$ is equivalent to $D_{\eta,\chi} \geq 0$ on $\R$.

The following result (see Hobson~\cite[page 254]{Hobson:98b} or Beiglb\"{o}ck and Juillet ~\cite[Section A.1]{BeiglbockJuillet:16}) tells us that, if $\eta\leq_{cd}\chi$ and $D_{\eta,\chi}(x)=0$ for some $x$, then in any supermartingale coupling of $\eta$ and $\chi$ no mass can cross $x$. (Hobson~\cite{Hobson:98b} and Beiglb\"{o}ck and Juillet ~\cite{BeiglbockJuillet:16} considered the case $\eta\leq_c\chi$, however the same arguments also work in the case $\eta\leq_{cd}\chi$, see Nutz and Stebegg \cite[Proposition 3.2]{NutzStebegg:18}.)

\begin{lem} Suppose $\eta$ and $\chi$ are probability measures with $\eta \leq_{cd} \chi$. Suppose that $D_{\eta,\chi}(x)=0$. If $\pi \in {\Pi_S}(\eta,\chi)$, then we have $\pi((-\infty,x),(x,\infty)) + \pi((x,\infty),(-\infty,x))=0$.
	\label{lem:chacon}
\end{lem}

It follows from Lemma \ref{lem:chacon} that, if there is a point $x\in(\ell_\eta,r_\eta)$ such that $D_{\eta,\chi}(x)=0$, then the problem of constructing supermartingale couplings of $\eta$ to $\chi$ can be separated into a pair of subproblems involving mass to the left and right of $x$ (and by carefully choosing how to allocate the mass of $\chi$ at $x$). In particular, if there are multiple $\{ x_i \}_{i\geq1}$ with $D_{\eta,\chi}(x_i)=0$, then the problem splits into a sequence of `irreducible' problems, each taking place on an interval $\sI_j$ such that $D_{\eta,\chi}>0$ on the interior of $\sI_j$ and $D=0$ at the endpoints. All mass starting in the interior of $\sI_j$ is transported to a point in $\sI_j$. This is summarised in the following lemma.

\begin{lem}[{Nutz and Stebegg \cite[Proposition 3.4]{NutzStebegg:18}}]\label{lem:irreducibleComp}
	Let $\eta,\chi\in\sM$ with $\eta\leq_{cd}\chi$. Set $x_{\eta,\chi}:=\sup\{k\in\R:D_{\eta,\chi}(k)=0\}\in[-\infty,\infty]$.
	
	Define $\sI^{\eta,\chi}_0:=(x_{\eta,\chi},\infty)$ and let $(\sI_i^{\eta,\chi})_{i\geq 1}$ be the open components of $\{k\in\R:D_{\eta,\chi}(k)>0\}\cap(-\infty,x_{\eta,\chi})$. Define a closed set $\sI^{\eta,\chi}_{-1}$ by $\sI^{\eta,\chi}_{-1}=\R\setminus\bigcup_{i\geq0}\sI^{\eta,\chi}_i$.
	
	Set $\eta_i=\eta\lvert_{\sI^{\eta,\chi}_i}$, $i\geq-1$, so that $\eta=\sum_{i\geq-1}\eta_i$.
	
	Then there exists a unique decomposition $\chi=\sum_{i\geq-1}\chi_i$ such that
	$$
	\eta_{-1}=\chi_{-1},\quad \eta_0\leq_{cd}\chi_0\quad\textrm{and}\quad\eta_i\leq_c\chi_i,\textrm{ for all }i\geq1.
	$$
	Set $\sI^{\eta,\chi}_i=(\ell_i,r_i)$, $i\geq0$ (here $r_i\leq\ell_0=x_{\eta,\chi}\leq r_0=\infty$, $i\geq1$).
	
	Then for all $i\geq0$,
	$$
	D_{\eta_i,\chi_i}>0\textrm{ on } \sI^{\eta,\chi}_i\quad\textrm{and}\quad D_{\eta_i,\chi_i}=0\textrm{ on }\R\setminus\sI^{\eta,\chi}_i,
	$$
	and there exists a unique choice for $0\leq\alpha_i\leq\nu(\{\ell_i\})$ and $0\leq\beta_i\leq\nu(\{r_i\})$ satisfying
	$$
	\chi_i=\chi\lvert_{\sI^{\eta,\chi}_i}+\alpha_i\delta_{\ell_i}+\beta_i\delta_{r_i}.
	$$
	
	Finally, any $\pi\in\Pi_S(\eta,\chi)$ admits a unique decomposition $\pi=\sum_{i\geq-1}\pi_i$ such that $\pi_0\in\Pi_S(\eta_0,\chi_0)$ and $\pi_i\in\Pi_M(\eta_i,\chi_i)$ for all $i\geq-1$ with $i\neq0$.
\end{lem}
\begin{rem}
	\label{rem:martIrreduc}
	When $\eta\leq_{c}\chi$ then $\Pi_S(\eta,\chi)=\Pi_M(\eta,\chi)$ and therefore we also have that $\eta_0\leq_c\chi_0$.
\end{rem}

{\em Notation:} For $x \in \R$ let $\delta_x$ denote the unit point mass at $x$. For real numbers $c,x,d$ with $c \leq x \leq d$ define the probability measure $\chi_{c,x,d}$ by $\chi_{c,x,d} = \frac{d-x}{d-c}\delta_c + \frac{x-c}{d-c} \delta_d$ with $\chi_{c,x,d} = \delta_x$ if $(d-x)(x-c)=0$. Note that $\chi_{c,x,d}$ has mean $x$ and is the law of a Brownian motion started at $x$ evaluated on the first exit from $(c,d)$.
\subsection{Convex hull}\label{prelims:hull}

Our key results will be expressed in terms of the convex hull. For $f : \R \mapsto (-\infty,\infty)$ let $f^c$ be the largest convex function which lies below $f$. In our typical application $f$ will be non-negative and this property will be inherited by $f^c$. However, in general we may have $f^c$ equal to $-\infty$ on $\R$, and the results of this section are stated in a way which includes this case.  Note that if a function $g$ is equal to $-\infty$ (or $\infty$) everywhere, then we deem it to be both linear and convex, and set $g^c$ equal to $g$.

Recall the definition of the sub-differential $\partial h(x)$ of a convex function $h:\R \mapsto \R$ at $x$:
\[ \partial h (x) = \{ \phi \in\R: h(y) \geq h(x) + \phi(y-x) \mbox{ for all } y \in \R \}. \]
We extend this definition to non-convex functions $f$ so that the sub-differential of $f$ at $x$ is given by
\[ \partial f (x) = \{ \phi\in\R : f(y) \geq f(x) + \phi(y-x) \mbox{ for all } y \in \R \}. \]
If $h$ is convex then $\partial h$ is non-empty everywhere, but this is not the case for non-convex functions. Instead we have that $\partial f(x)$ is non-empty if and only if $f(x)=f^c(x)$ and then $\partial f^c(x) = \partial f(x)$. We also write $f'(\cdot-)$ and $f'(\cdot+)$ for the left and right derivatives (provided they exist) of a function $f$.

Fix $x,z\in\R$ with $x \leq z$, and define $L^f_{x,z}:\R\mapsto\R$ by
\begin{equation}\label{eq:L1}
L^f_{x,z}(y)=\begin{cases}
f(x)+\frac{f(z)-f(x)}{z-x}(y-x),&\textrm{ if }x<z,\\
f(x),&\textrm{ if }x=z.
\end{cases}
\end{equation}
Then, see Rockafellar \cite[Corollary 17.1.5]{Rockafellar:72},
\begin{equation}\label{eq:chull}
f^c(y)=\inf_{x\leq y\leq z}L^f_{x,z}(y),\quad y\in\R.
\end{equation}

Moreover, it is not hard to see (at least pictorially, by drawing the graphs of $f$ and $f^c$) that $f^c$ replaces the non-convex segments of $f$ by straight lines. (The proof of the following lemma is standard, see, for example, Hobson and Norgilas \cite[Lemma 2.2]{HobsonNorgilas:21}.)

\begin{lem}
	\label{lem:linear}
	Let $f:\R\mapsto\R$ be lower semi-continuous. Suppose $f>f^c$ on $(a,b)\subseteq\R$. Then $f^c$ is linear on $(a,b)$.
\end{lem}

The following definition, for a given function $f$ and $y \in \R$, will allow us to identify the values $x,z\in\R$ with $x\leq y \leq z$ which attain the infimum in \eqref{eq:chull}.
\begin{defn}\label{defn:convh}
	Let $f:\R\mapsto(-\infty,\infty]$ be a measurable function and $f^c$ denote its convex hull. For $y\in\R$, define
	\begin{align*}
	X^f(y)=X(y)&=\sup\{x:x\leq y, f^c(x)=f(x)\},\\
	Z^f(y)=Z(y)&=\inf\{z:z\geq y, f^c(z)=f(z)\},
	\end{align*}
	with the convention that $\sup\emptyset=-\infty$ and $\inf\emptyset = \infty$.
\end{defn}
\begin{rem}\label{rem:XYloc}
If $f$ is continuous, then $f^c(y)=L^f_{X(y),Z(y)}(y)$ (see Hobson and Norgilas \cite[Lemma 2.4]{HobsonNorgilas:21}). Note, however, that $X^f$ and $Z^f$ may take values $-\infty$ and $\infty$, respectively. Hence for the aforementioned equality to remain valid, one has to carefully extend the definition of $L^f_{x,z}$ allowing for $x=-\infty$ and $z=\infty$.
	\end{rem}

The following lemmas are the main ingredients in the proofs of Theorem \ref{thm:shadow_potential} and Proposition \ref{thm:shadow_assoc} (the proof of Lemma \ref{lem:CSpotential} is postponed until the appendix, while the proofs of the remaining two can be found in Beiglb\"{o}ck et al. \cite{BHN:20}).
\begin{lem}\label{lem:diff}
	Let $f,g:\R\mapsto\R$ be convex and lower semi-continuous. Define $\psi:\R\mapsto(-\infty,\infty]$ by
	$
	\psi =g-(g-f)^c
	$.
	Then $\psi$ is convex.
\end{lem}

\begin{lem}\label{lem:equal_hulls}
	Let $f:\R\mapsto\R$ be measurable and let $g:\R\mapsto\R$ be convex. Then
	$$
	(f-g)^c=(f^c-g)^c.
	$$
\end{lem}

\begin{lem}\label{lem:CSpotential} Assume that
	$f\in \sD(\alpha, \beta)$ and $g\in \sD(a,b)$ for some $\alpha, a \geq 0$, $\beta, b \in \R$. Let $h:\R\mapsto\R$ be defined by $h(k):=(a-\alpha)k-(b-\beta)$. Define $\eta:=\sup_{k\in\R}\{h(k)-g(k)+f(k)\}$.

	 Suppose that $g\geq f$. Then $\alpha \leq a$. If $\alpha = a$ then $\beta \geq b$. Furthermore, $\eta\in[0,\infty)$ and $ (g-f)^c\in \sD(a-\alpha, {b-\beta+\eta}).$
	
	
\end{lem}

\section{The shadow measure and $\pi_I$}\label{sec:shadowLC}
\subsection{The maximal element}\label{sec:max}
Let $\mu,\nu\in\sM$ with $\mu\leq_{pcd}\nu$ and define $\sM^\nu_\mu=\{\eta\in\sM:\mu\leq_{cd}\eta\leq\nu\}$. Then $\sM^\nu_\mu$ is a set of terminal laws of a supermartingale that embeds $\mu$ into $\nu$. Note that $\eta\in\sM^\nu_\mu$ if and only if $P_\eta\in\sC(\mu,\nu)$.

Nutz and Stebegg \cite{NutzStebegg:18} showed that there exists a measure in $\sM^\nu_\mu$ that is minimal w.r.t. $\leq_{cd}$ (see section \ref{sec:shadow}). The first step in their proof is to show that $\sM^\nu_\mu\neq\emptyset$ by constructing the `left-most' measure $\theta\leq\nu$ of mass $\mu(\R)$. In this section we show that this left-most measure is indeed the largest measure (w.r.t. $\leq_{cd}$) in $\sM^\nu_\mu$.

Let $G_\nu:[0,\nu(\R)]\mapsto\R$ be a quantile function of $\nu$. Define $T^\nu(\mu)\in\sM$ by
\begin{equation}\label{eq:leftmost}
T^\nu(\mu):=\nu\lvert_{(-\infty,G_\nu(\mu(\R)))}+(\mu(\R)-\nu((-\infty,G_\nu(\mu(\R))))\delta_{G_\nu(\mu(\R))}.
\end{equation}
Note that $T^\nu(\mu)$ does not depend on the version of $G_\nu$.
\begin{rem}\label{rem:Ecx0}
	Let $\mu,\nu\in\sM$ with $\mu \leq_{cd} \nu$. Then $\mu(\R)=\nu(\R)$ and therefore $T^\nu(\mu)=\nu$. This is consistent with a fact that if $\mu \leq_{cd} \nu$ then $\{ \eta: \mu \leq_{cd} \eta \leq \nu \}$ is the singleton $\{\nu\}$.
\end{rem}

\begin{prop}\label{thm:maximal}
	Let $\mu,\nu \in\sM$ with $\mu \leq_{pcd} \nu$. Then $T^\nu(\mu)$, defined in \eqref{eq:leftmost}, is a unique measure with the following properties:
	\begin{enumerate}
		\item\label{item1max} $\mu\leq_{cd}T^\nu(\mu)$,
		\item\label{item2max} $T^\nu(\mu)\leq\nu$,
		\item\label{item3max} If $\eta$ is another measure satisfying $\mu \leq_{cd} \eta \leq \nu$ then $\eta\leq_{cd}T^\nu(\mu)$.
	\end{enumerate}
\end{prop}
\begin{proof} First note that $T^\nu(\mu)(\R)=\mu(\R)$. Furthermore, $T^\nu(\mu)$ satisfies Property~\ref{item2max} by definition, while the uniqueness is a direct consequence of Property \ref{item3max}.	
	
	Now define $h:\R\mapsto\R$ by
	$$
	h(k)=\mu(\R)k-\overline{T^\nu(\mu)},\quad k\in\R.
	$$
	Then $h$ is the line tangent to $P_\nu$ at $G_\nu(\mu(\R))$. In particular, $P_{T^\nu(\mu)}$ is convex, $P_{T^\nu(\mu)}=P_\nu$ on $(-\infty,G_\nu(\mu(\R))]$, $P_{T^\nu(\mu)}=h$ on $(G_\nu(\mu(\R)),\infty)$ and $\lim_{\lvert x\lvert\to\infty}\{P_{T^\nu(\mu)}(k)-h^+(k)\}=0$. Note that, since $P_\nu\geq P_\mu$ everywhere, $\overline{T^\nu(\mu)}\leq\bar\mu$.
	
	Now, since $P_{T^\nu(\mu)}=P_\nu\geq P_\mu$ on $(-\infty,G_\nu(\mu(\R))]$, we have that $P_{T^\nu(\mu)}(G_\nu(\mu(\R)))\geq P_\mu(G_\nu(\mu(\R)))$. However, $P_\mu$ is convex and $P'_\mu(G_\nu(\mu(\R))+)\leq\mu(\R)=P'_{T^\nu(\mu)}(G_\nu(\mu(\R))+)$. In particular, if $P_\mu(k)> P_{T^\nu(\mu)}(k)=h(k)$ for some $k\in(G_\nu(\mu(\R)),\infty)$, then $P_\mu'(k+)>\mu(\R)$, a contradiction to the fact that $P_\mu\in\sD(\mu(\R),\overline\mu)$. It follows that $P_{T^\nu(\mu)}\geq P_\mu$ everywhere, and thus Property \ref{item1max} holds.
	
	Finally, we verify Property \ref{item3max}. Let $\theta\in\sM^\nu_\mu$. Then, since $T^\nu(\mu)=\nu\geq\theta$ on $(-\infty,G_\nu(\mu(\R)))$, and by the continuity of $P_{T^\nu(\mu)}$ and $P_\theta$, we have that $P_{T^\nu(\mu)}\geq P_\theta$ on $(-\infty,G_\nu(\mu(\R))]$. Now suppose that there exists $k\in(G_\nu(\mu(\R)),\infty)$ with $P_{T^\nu(\mu)}(k)=h(k)=\mu(\R)k-\overline{T^\nu(\mu)}<P_\theta(k)$. Then, since $P'_\theta(k-)\leq\mu(\R)=P'_{T^\nu(\mu)}(k-)$, convexity of $P_\theta$ ensures that $P_\theta(G_\nu(\mu(\R)))>P_{T^\nu(\mu)}(G_\nu(\mu(\R)))$, a contradiction. We conclude that $P_{T^\nu(\mu)}\geq P_\theta$ everywhere, and since $\theta\in\sM^\nu_\mu$ was arbitrary, Property \ref{item3max} holds.
\end{proof}

The proof of the following lemma is similar to the proof of Beiglb\"{o}ck et al. \cite[Lemma 6]{BHN:20} (replace $\leq_{pc}$ with $\leq_{pcd}$ and $\leq_{c}$ with $\leq_{cd}$), and thus is omitted.
\begin{lem}\label{lem:order}
	Suppose $\mu, \nu \in \sM$. The following are equivalent:
	\begin{itemize}
		\item[(i)] $\mu \leq_{pcd} \nu$;
		\item[(ii)] there exists $\eta \in \sM$ such that $\mu \leq_{cd} \eta \leq \nu$;
		\item[(iii)] there exists $\chi \in \sM$ such that $\mu \leq \chi \leq_{cd} \nu$.
	\end{itemize}
\end{lem}
\begin{rem}\label{rem:orderCX}
	If we replace $\leq_{pcd}$ with $\leq_{pc}$ and $\leq_{cd}$ with $\leq_c$ in Lemma \ref{lem:order}, then its statement remains true, see Beiglb\"{o}ck et al. \cite[Lemma 6]{BHN:20}. 
\end{rem}
Let $\mu=\mu_1+\mu_2$ for some $\mu_1,\mu_2\in\sM$ and $\nu\in\sM$ with $\mu\leq_{pcd}\nu$. Then $\sM^\nu_{\mu_1}\neq\emptyset$ and, in particular, we can embed $\mu_1$ into $\nu$ using any supermartingale coupling $\pi\in\Pi_S(\mu_1,T^\nu(\mu_1))$. A natural question is then whether $\sM_{\mu_2}^{\nu-T^\nu(\mu_1)}$ is non-empty, so that the remaining mass $\mu_2$ can also be embedded in what remains of $\nu$.
\begin{eg}\label{eg2}
	Let $\mu=\delta_0$ and $\nu=\frac{1}{2}(\delta_{-2}+\delta_1)$. Then $\mu\leq_{cd}\nu$. Consider $\mu_1=\mu_2=\frac{1}{2}\delta_0$. Then $T^\nu(\mu_1)=\frac{1}{2}\delta_{-2}$. However, $\mu_2\leq_{cd}\nu-T^\nu(\mu_1)$ does not hold. Indeed, $\nu-T^\nu(\mu_1)=\frac{1}{2}\delta_1\leq_{cd}\frac{1}{2}\delta_0=\mu_2$.
\end{eg}
As Example \ref{eg2} demonstrates, for $\mu_1,\mu_2,\nu\in\sM$ with $\mu_1+\mu_2=\mu\leq_{pcd}\nu$,
if we first transport $\mu_1$ to $T^\nu(\mu_1)$, then we cannot,
in general, embed $\mu_2$ in $\nu - T^\nu(\mu_1)$ in a way which respects the supermartingale property. As a consequence, for arbitrary measures in convex-decreasing order we cannot expect the maximal element to \textit{induce} a supermartingale coupling. In Section \ref{sec:shadow} we study the minimal element of $\sM^\nu_\mu$, namely the shadow measure. The shadow measure has the property that if $\mu_1+\mu_2=\mu\leq_{pcd}\nu$ and we transport $\mu_1$ to the shadow $S^{\nu}(\mu_1)$ of $\mu_1$ in $\nu$, then $\mu_2$ is in positive convex-decreasing order with what remains of $\nu$, i.e., $\mu_2 \leq_{pcd} \nu - S^\nu(\mu_1)$.

Let $\mu,\nu\in\sM$ with $\mu\leq_{pcd}\nu$ and define $\tilde{\sM}^\nu_\mu:=\{\eta\in\sM:\mu\leq_c\eta\leq\nu\}\subseteq\sM^\nu_\mu$, so that $\tilde{\sM}^\nu_\mu$ is a set of terminal laws of a martingale that embeds $\mu$ into $\nu$. Note that, by Lemma \ref{lem:order} and Remark \ref{rem:orderCX}, if $\tilde{\sM}^\nu_\mu\neq\emptyset$ then $\mu\leq_{pc}\nu$. We close this section with the following result that allows us to check whether $\leq_{pcd}$ is equivalent to $\leq_{pc}$.
\begin{lem}\label{prop:pcd=pc} Let $\mu,\nu\in\sM$ with $\mu\leq_{pcd}\nu$. Then $\mu\leq_{pc}\nu$ if and only if $C_\mu\leq C_\nu$ everywhere.
\end{lem}
\begin{proof}
		If $\mu\leq_{pc}\nu$ then $C_\mu\leq C_\nu$ since, for each $k\in\R$, $x\mapsto (x-k)^+$ is non-negative and convex, and hence we are done.
		
		Now suppose that $C_\mu\leq C_\nu$ everywhere. Let $\tilde{P}^\nu_\mu:\R\mapsto\R$ be defined by
		$$
		\tilde{P}^{\nu}_\mu(k)=\min\{P_\nu(k),C_\nu(k) +  (\mu(\R)k-\overline{\mu})\},\quad k\in\R.
		$$ We will show that $(\tilde{P}^\nu_\mu)^c\in\sD(\mu(\R),\overline\mu)$, $P_\nu-(\tilde{P}^\nu_\mu)^c$ is convex and $P_\mu\leq(\tilde{P}^\nu_\mu)^c$ everywhere, proving that the second derivative of $(\tilde{P}^\nu_\mu)^c$ corresponds to $\eta\in\sM$ with $\mu\leq_c\eta\leq\nu$.
		
		First, since $C_\mu(k) - P_\mu(k) = (\overline{\mu}- \mu(\R)k)$,
		$$
		\tilde{P}^\nu_\mu(k)=P_\mu(k)+\min\{P_\nu(k)-P_\mu(k),C_\nu(k)-C_\mu(k)\}\geq P_\mu(k),\quad k\in\R.
		$$
		Then from the convexity of $P_\mu$ and the definition of the convex hull it follows that $P_\mu\leq  (\tilde{P}^\nu_\mu)^c\leq\tilde{P}^\nu_\mu$ everywhere.
		
		Second, since $\lim_{k\to-\infty}(P_\nu(k)-P_\mu(k))=\lim_{k\to\infty}(C_\nu(k)-C_\mu(k))=0$, we have that $\lim_{\lvert k\lvert\to\infty}\tilde{P}^\nu_\mu(k)=\lim_{\lvert k\lvert\to\infty}{P}_\mu(k)$. Then, since $P_\mu\in\sD(\mu(\R), \overline{\mu})$, convexity of $(\tilde{P}^\nu_\mu)^c$ ensures that $(\tilde{P}^\nu_\mu)^c \in \sD(\mu(\R), \overline{\mu})$.
		
		Finally we prove the convexity of $P_\nu-(\tilde{P}^\nu_\mu)^c$. First note that $\tilde{P}^\nu_\mu(k) = P_\nu(k) - ((\nu(\R)-\mu(\R))k - (\bar{\nu} - \bar{\mu}))^+$, $k\in\R$. Then, since $p$ given by $p(k)=((\nu(\R)-\mu(\R))k - (\bar{\nu} - \bar{\mu}))^+$ is convex, we can apply Lemma~\ref{lem:diff}, with $g=P_\nu$ and $f=p$. It follows that $P_\nu - (P_\nu - p)^c = P_\nu - (\tilde{P}^\nu_\mu)^c$ is convex, as required.
		
		We showed that $\mu\leq_c\eta\leq\nu$, where $\eta=((\tilde{P}^\nu_\mu)^c)''\in\sM$. By Lemma \ref{lem:order} and Remark \ref{rem:orderCX} it follows that $\mu\leq_{pc}\nu$.
	\end{proof}

\subsection{The shadow measure}\label{sec:shadow}

\begin{defn}[Shadow measure]\label{defn:shadow}
	Let $\mu,\nu\in\sM$ and assume $\mu\leq_{pcd}\nu$. The shadow of $\mu$ in $\nu$, denoted by $S^\nu(\mu)$, has the following properties
	\begin{enumerate}
		\item\label{s1} $\mu\leq_{cd}S^\nu(\mu)$,
		\item\label{s2} $S^\nu(\mu)\leq\nu$,
		\item \label{s3}If $\eta$ is another measure satisfying $\mu \leq_{cd} \eta\leq \nu$, then $S^\nu(\mu)\leq_{cd}\eta$.
	\end{enumerate}
	\label{lem:shadowDef}
\end{defn}

\begin{lem}[Nutz and Stebegg \cite{NutzStebegg:18}, Lemma 6.2]
	For $\mu,\nu\in\sM$ with $\mu\leq_{pcd}\nu$, $S^\nu(\mu)$ exists and is unique.
\end{lem}
\begin{rem}\label{rem:Ecx}
	If $\mu \leq_{cd} \nu$ then, in the light of Remark \ref{rem:Ecx0}. $S^\nu(\mu)=\nu=T^\nu(\mu)$.
\end{rem}




Given $\mu$ and $\nu$ with $\mu \leq _{pcd} \nu$ (and, by Remark \ref{rem:Ecx}, with $\mu(\R)<\nu(\R)$) our goal in this section is to construct the shadow measure ${S}^\nu(\mu)$. We do this by finding a corresponding (modified) potential function $P_{S^\nu(\mu)}$ (and then $S^\nu(\mu)$ can be identified as the second derivative of $P_{S^\nu(\mu)}$ in the sense of distributions).

\begin{thm}\label{thm:shadow_potential}
	Let $\mu,\nu\in\sM$ with $\mu\leq_{pcd}\nu$. Then the shadow of $\mu$ in $\nu$ is uniquely defined and given by
	\begin{equation}\label{eq:shadow_potential}
	P_{S^\nu(\mu)}=P_\nu-(P_\nu-P_{\mu})^c. 
	\end{equation}
\end{thm}

\begin{proof}
	If $\mu \leq_{pcd} \nu$ and $\mu(\R) = \nu(\R)$ then $\mu \leq_{cd} \nu$ and $\overline{\nu} \leq \overline{\mu}$. Then $\lim_{k\to-\infty}\{P_\nu(k)-P_\mu(k)\}=0$ and $\lim_{k\to\infty}\{P_\nu(k)-P_\mu(k)\}=\bar\mu-\bar\nu\geq0$. Since $0\leq(P_{\nu} - P_{\mu})^c\leq P_{\nu} - P_{\mu}$, it follows that $(P_{\nu} - P_{\mu})^c$ is the zero function which is a unique element in $\sD(0,0)$. Then \eqref{eq:shadow_potential} gives that $P_{S^\nu(\mu)}=P_\nu$, and thus $S^\nu(\mu)=\nu$, which is consistent with Remark \ref{rem:Ecx}. In the rest of the proof we assume that $\mu \leq_{pcd} \nu$ with $\mu(\R) < \nu(\R)$.
	
	First we will rephrase Definition \ref{defn:shadow} in terms of the modified potential function: $h$ is the potential of the shadow of $\mu$ in $\nu$ if
	\begin{enumerate}
		\item[0]\label{shadow0} $h\in \sD(\mu(\R), \overline\mu-c)$, for some $c\in[0,\infty)$,
		\item[1]\label{shadow1} $P_\mu \leq  h$,
		\item[2]\label{shadow2} $P_\nu-h$ is a potential function, i.e., $P_\nu-h\in \sD(\alpha, \beta)$ for some $\alpha\geq 0, \beta\in \R$,
		\item[3]\label{shadow3} If $p$ is another potential function satisfying properties 0,1,2 then $h \leq p$.
	\end{enumerate}
	Equivalently we can write this as
	\begin{enumerate}
		\item[0] $h\in \sD(\mu(\R), \overline\mu-c)$, for some $c\in[0,\infty)$,
		\item[$1'$]\label{shadow1x} $(P_\nu-h) \leq (P_\nu- P_\mu)$,
		\item[2] $P_\nu-h$ is a potential function, i.e., $P_\nu-h\in \sD( \alpha, \beta)$ for some $\alpha\geq 0, \beta\in \R$,
		\item[$3'$]\label{shadow3x} If $p$ is another potential function with properties 0,$1'$,2 then $(P_\nu-h) \geq (P_\nu-p)$.
	\end{enumerate}

	By Lemma \ref{lem:CSpotential} with $g=P_\nu$ and $f= P_\mu$ we have $(P_\nu-P_\mu)^c\in\sD(\nu(\R)-\mu(\R),\overline\nu-\overline\mu+c_{\mu,\nu})$, where $c_{\mu,\nu}:=\sup_{k\in\R}\{(\nu(\R)-\mu(\R))k-(\bar\nu-\bar\mu)-P_\nu(k)+P_\mu(k)\}\in[0,\infty)$.

	Now set $h=P_\nu-(P_\nu-P_\mu)^c$. Since $ P_\nu - h = (P_\nu-P_\mu)^c$, Property $2$ is satisfied. Furthermore, using the definition of the convex hull we have that $P_\nu-h=(P_\nu-P_\mu)^c\leq P_\nu-P_\mu$, and thus Property $1'$ is also satisfied.
	
	 We now verify that $h\in \sD(\mu(\R), \overline\mu-c_{\mu,\nu})$. First note that $h\geq P_\nu-(P_\nu-P_\mu)=P_\mu\geq0$ and $\lim_{k\to-\infty}h(k)=\lim_{k\to-\infty}P_\nu(k)-\lim_{k\to-\infty}(P_\nu-P_\mu)^c(k)=0$. Next, by applying Lemma \ref{lem:diff}, with $g=P_\nu$ and $f=P_\mu$, we have that $h$ is convex (and thus also non-decreasing). We are left to show that the asymptotic slope of $h$ at $\infty$ is equal to $\overline\mu-c_{\mu,\nu}$. But
	 \begin{align*}
	&\lim_{k\to\infty}\{h(k)-\mu(\R)k+(\bar\mu-c_{\mu,\nu})\}\\&=\lim_{k\to\infty}\Big\{\{P_\nu(k)-\nu(\R)k+\bar\nu\}-\{(P_\nu-P_\mu)^c(k)-(\nu(\R)-\mu(\R))k+(\bar\nu-\bar\mu+c_{\mu,\nu})\}\Big\}=0
	 \end{align*}
	 and we conclude that $h\in \sD(\mu(\R), \overline\mu-c_{\mu,\nu})$.
	
Finally we claim that $ P_\nu - h = (P_\nu-P_\mu)^c$ satisfies property $3'$. If $p$ is another potential satisfying properties $0,1',2'$, then $P_\nu- p\leq P_\nu-P_\mu$ and $P_\nu-p$ is convex. Then by the maximality of the convex hull we have that $P_\nu-p\leq P_\nu-h=(P_\nu-P_\mu)^c\leq P_\nu-P_\mu$, which concludes the proof.
\end{proof}
\begin{rem}\label{rem:samePotential}
	Suppose $\mu,\nu\in\sM$ with $\mu\leq_{pc}\nu$. Then replacing $\leq_{pcd}$ with $\leq_{pc}$ in Definition \ref{defn:shadow} we recover the definition of the shadow measure in the martingale case (see Beiglb\"{o}ck et al. \cite[Definition 2]{BHN:20}). Surprisingly, the functional representation \eqref{eq:shadow_potential} of the modified potential $P_{S^\nu(\mu)}$ is the same in both cases (see Beiglb\"{o}ck et al. \cite[Theorem 2]{BHN:20}). 
	\end{rem}
	
We now turn to the associativity of the shadow measure. The proof in the martingale case, given by Beiglb\"{o}ck and Juillet \cite{BeiglbockJuillet:16}, is delicate and relies on the approximation of $\mu$ by atomic measures. On the other hand, Nutz and Stebegg \cite{NutzStebegg:18} only give a comment that the proof in the supermartingale case can be obtained along the lines of Beiglb\"{o}ck and Juillet \cite{BeiglbockJuillet:16}. Thanks to Theorem \ref{thm:shadow_potential}, and similarly to Beiglb\"{o}ck et al. \cite[Theorem 3]{BHN:20}, we are able to give a direct proof of the associativity of the supermartingale shadow.

\begin{prop}\label{thm:shadow_assoc}
	Let $\mu_1,\mu_2,\nu\in\sM$. Suppose $\mu=\mu_1+\mu_2$ and $\mu\leq_{pcd}\nu$. Then $\mu_2\leq_{pcd}\nu-S^\nu(\mu_1)$ and
	\begin{equation}\label{eq:assoc}
	S^\nu(\mu_1+\mu_2)=S^\nu(\mu_1)+S^{\nu-S^\nu(\mu_1)}(\mu_2).
	\end{equation}
\end{prop}
\begin{proof}
	
	
	

	We first prove that  $\mu_2\leq_{pcd}\nu-S^\nu(\mu_1)$. Define $P_\theta:\R\to\R_+$ by
	$$
	P_\theta(k)=(P_\nu-P_{\mu_1})^c(k)-((P_\nu-P_{\mu_1})^c-P_{\mu_2})^c(k),\quad k\in\R.
	$$
	We will show that $P_\theta\in\sC(\mu_2,\nu-S^\nu(\mu_1))$. Then the second derivative of $P_\theta$ corresponds to a measure $\theta\in\sM^{\nu-S^\nu(\mu_1)}_{\mu_2}$, which by Lemma~\ref{lem:order} is enough to prove the assertion.
	
	Convexity of $P_\theta$ is a direct consequence of Lemma~\ref{lem:diff} with $g=(P_\nu-P_{\mu_1})^c$ and $f=P_{\mu_2}$. Moreover, since
	$
	P_{\nu-S^\nu(\mu_1)}=P_\nu-P_{S^\nu(\mu_1)}=(P_\nu-P_{\mu_1})^c,
	$
	we have that
	$$
	P_{\nu-S^{\nu}(\mu_1)}-P_\theta =((P_\nu-P_{\mu_1})^c-P_{\mu_2})^c\leq(P_\nu-P_{\mu_1})^c-P_{\mu_2},
	$$
	and it follows that $(P_{\nu-S^{\nu}(\mu_1)}-P_\theta)$ is convex and $P_{\mu_2}\leq P_\theta$. To prove that $\mu_2 \leq_{pcd} \nu - S^{\nu}(\mu_1)$ it only remains to show that $P_\theta$ has the correct limiting behaviour to ensure that $P_\theta \in \sC(\mu_2,\nu-S^\nu(\mu_1))$. For this we will apply Lemma~\ref{lem:CSpotential} to each of the convex hulls in the definition of $P_\theta$ and then to $P_\theta$ itself.
	
	First, since $\mu_1\leq_{pcd}\nu$ (so that $P_{\mu_1}\leq P_\nu$), and by Lemma \ref{lem:CSpotential} with $g=P_\nu$ and $f=P_{\mu_1}$,
	we have that $(P_\nu-P_{\mu_1})^c\in\sD(\nu(\R)-\mu_1(\R),\overline\nu-\overline{\mu_1}+c_{\mu_1,\nu})$, where we write $c_{\eta,\chi}:=\sup_{k\in\R}\{(\chi(\R)-\eta(\R))k-(\bar\chi-\overline{\eta})-P_\chi(k)+P_{\eta}(k)\}$ for $\eta,\chi\in\sM$ with $\eta\leq_{pcd}\chi$ (recall that $c_{\eta,\chi}\in[0,\infty)$). Similarly, since $\mu_1+\mu_2=\mu\leq_{pcd}\nu$, $(P_\nu-P_{\mu_1}-P_{\mu_2})^c\in\sD(\nu(\R)-\mu_1(\R)-\mu_2(\R),\overline{\nu-\mu_1-\mu_2}+c_{\mu,\nu})$. But, by Lemma \ref{lem:equal_hulls}, with $f=(P_\nu-P_{\mu_1})$ and $ g=P_{\mu_2}$, we have that $((P_\nu-P_{\mu_1})^c-P_{\mu_2})^c=(P_\nu-P_{\mu_1}-P_{\mu_2})^c$. Finally, recall that $P_\theta\geq P_{\mu_2}$ and, since $P_\theta$ is convex, $P_\theta=P_\theta^c$. Therefore, by applying Lemma~\ref{lem:CSpotential} with $g=(P_\nu-P_{\mu_1})^c$ and $f=((P_\nu-P_{\mu_1})^c-P_{\mu_2})^c$, we conclude that $P_\theta\in\sD(\mu_2(\R),\overline{\mu_2}-\bar c)$, where $\bar c=c_{\mu,\nu}-c_{\mu_1,\nu}\geq0$. (Note that $\bar c\geq0$ since $P_\theta\geq P_{\mu_2}$.)
	
	We are left to prove the associativity property \eqref{eq:assoc}. However this follows from similar arguments used in the proof of the associativity in the martingale case, see Beiglb\"{o}ck et al. \cite[Theorem 3]{BHN:20}. Hence we omit the details.
\end{proof}

We give one further result which is easy to prove using Theorem~\ref{thm:shadow_potential} and which describes a
structural property of the shadow. (The proof can be obtained along the lines of the proof of Beiglb\"{o}ck et al. \cite[Proposition 2]{BHN:20} with $\leq_{pcd}$ in place of $\leq_{pc}$.)
\begin{lem}\label{prop:shadowSum}
	Suppose $\xi,\mu,\nu\in\sM$ with $\xi \leq \mu \leq_{pcd} \nu$. Then, $\xi \leq_{pcd} \nu$, $\xi \leq_{pcd} S^\nu(\mu)$ and
	$$
	S^{S^\nu(\mu)}(\xi) = S^\nu(\xi).
	$$
	In particular, $S^\nu(\xi)\leq S^\nu(\mu)$.
\end{lem}

\begin{eg}\label{eg3}
	The assertion of Lemma~\ref{prop:shadowSum} does not hold for $\xi,\mu,\nu\in\sM$ with $\xi\leq_{pcd}\mu\leq_{pcd}\nu$. To see this, let $\xi=\frac{1}{3}\delta_0$, $\mu=\frac{1}{3}(\delta_{-2}+\delta_2)$ and $\nu=\frac{1}{3}(\delta_{-2}+\delta_0+\delta_2)$. Then $S^\nu(\mu)=\mu$ and $S^{S^\nu(\mu)}(\xi)=S^\mu(\xi)=\frac{1}{6}(\delta_{-2}+\delta_2)\neq\xi=S^\nu(\xi)$.
\end{eg}
\subsection{The increasing supermartingale coupling $\pi_{I}$}\label{sec:LCcoupling} 
The left-curtain martingale coupling (introduced by Beiglb\"{o}ck and Juillet~\cite{BeiglbockJuillet:16}), and the increasing supermartingale coupling (introduced by Nutz nad Stebegg \cite{NutzStebegg:18}), and denoted by $\pi_{lc}$ and $\pi_I$ respectively, are the couplings that arise via the shadow measure, created working from left to right. More specifically, when $\mu\leq_{cd}\nu$ (resp. $\mu\leq_c\nu$), $\pi_{I}$ (resp. $\pi_{lc}$) is a unique measure in $\Pi_S(\mu,\nu)$ (resp. $\Pi_M(\mu,\nu)$) which for each $x \in \R$ transports $\mu\lvert_{(-\infty,x]}$ to the shadow $S^\nu(\mu\lvert_{(-\infty,x]})$, see Nutz and Stebegg \cite[Theorem 6.6]{NutzStebegg:18} (resp. Beiglb\"{o}ck and Juillet~\cite[Theorem 4.18]{BeiglbockJuillet:16}). In other words, for each $x$, the first and second marginals of $\pi_{I}\lvert_{(-\infty,x]\times\R}$ (resp. $\pi_{lc}\lvert_{(-\infty,x]\times\R}$) are $\mu\lvert_{(-\infty,x]}$ and $S^\nu(\mu\lvert_{(-\infty,x]})$, respectively. 

An alternative characterization of $\pi_I$ and $\pi_{lc}$ is through their supports. As a consequence of the minimality of the shadow measure with respect to convex order (when $\mu\leq_c\nu$), $\pi_{lc}$ is also the unique martingale coupling  which is second-order left-monotone in the sense of Definition \ref{def:lmon} (see Beiglb\"{o}ck and Juillet ~\cite[Theorem 5.3]{BeiglbockJuillet:16}):

\begin{defn}\label{def:lmon}
	A transport plan $\pi\in\Pi(\mu,\nu)$ is said to be \textit{second-order left-monotone} if there exists $\Gamma\in\sB(\R^2)$ with $\pi(\Gamma)=1$ and such that, if $(x,y^-),(x,y^+),(x^\prime,y^\prime)\in\Gamma$ we cannot have $x<x^\prime$ and $y^-<y^\prime<y^+$.
\end{defn}

While the second-order left-monotonicity (in the case $\mu\leq_c\nu$) can be seen as a martingale counterpart of the \textit{c-cyclical monotonicity} in the classical OT theory, the supermartingale constraint requires a novel distinction of the origins $x$. In particular, when $\mu\leq_{cd}\nu$, the support of the initial measure $\mu$ splits into a set $\mathbb M$ of `martingale points' and their complement (i.e., the `supermartingale points'). In particular, Nutz and Stebegg \cite[Corollary 9.5]{NutzStebegg:18} showed that there exists $(\Gamma,\mathbb M)\in\sB(\R^2)\times\sB(\R)$ such that $\pi_{I}$ is concentrated on $\Gamma$, $\pi_{I}\lvert_{\mathbb M\times\R}$ is a martingale and $\pi_{I}$ is second-order left-monotone (w.r.t. $\Gamma$) and first-order right-monotone (w.r.t. $(\Gamma,\mathbb M)$) in the sense of Definition \ref{def:forightmon}. Furthermore, the converse is also true. Suppose $\mu\leq_{cd}\nu$ and let $\pi\in\Pi_S(\mu,\nu)$. If, for some $(\Gamma,\mathbb M)\in\sB(\R^2)\times\sB(\R)$, $\pi(\Gamma)=1$, $\pi\lvert_{\mathbb M\times\R}$ is a martingale and $\pi$ is both, second-order left-monotone (w.r.t. $\Gamma$) and first-order right-monotone (w.r.t. $(\Gamma,\mathbb M)$), then $\pi=\pi_{I}$ (see Nutz and Stebegg \cite[Theorem 8.1]{NutzStebegg:18}).

\begin{defn}\label{def:forightmon}
A transport plan $\pi\in\Pi(\mu,\nu)$ is said to be \textit{first-order right-monotone} if there exists $\Gamma\in\sB(\R^2)$ and $\mathbb M\in\sB(\R)$ such that $\pi(\Gamma)=1$ and, if $(x_1,y_1),(x_2,y_2)\in\Gamma$ with $x_1<x_2$ and $x_1\notin \mathbb M$, then we cannot have $y_1<y_2$.
\end{defn}

When the initial law $\mu$ is continuous (i.e., $\mu(\{x\})=0$ for all $x\in\R$), and if $\mu\leq_c\nu$, the left-curtain martingale coupling has a rather simple representation. In particular, for $x\in\mathbb{R}$, the element $\pi^x_{lc}(\cdot)$ in the disintegration $\pi_{lc}(dx,dy) = \mu(dx) \pi_{lc}^x(dy)$ is a measure supported on a set of at most two points.

\begin{lem}[{Beiglb\"ock and Juillet \cite[Corollary 1.6]{BeiglbockJuillet:16}}]\label{lem:LC}
	Let $\mu,\nu$ be probability measures in convex order and assume that $\mu$ is continuous. Then there exists a pair of measurable functions $T_d : \R \mapsto \R$ and $T_u : \R \mapsto \R$ such that $T_d(x) \leq x \leq T_u(x)$, such that for all $x<x'$ we have $T_u(x) \leq T_u(x')$ and $T_d(x') \notin (T_d(x),T_u(x))$, and such that, if we define $\bar{\pi}(dx,dy) = \mu(dx) \chi_{T_d(x),x,T_u(x)}(dy)$, then $\bar{\pi} \in {\Pi_M}(\mu,\nu)$ and $\bar{\pi}=\pi_{lc}$.
\end{lem}

Lemma~\ref{lem:LC} is expressed in terms of elements of $\Pi_M$. We can give an equivalent expression in terms of a martingale. First we give an analogue of Definitions~\ref{def:lmon} and \ref{def:forightmon} for functions.

\begin{defn}\label{def:lmonfns}
	Given an interval $I$, $\hat I\subset I$ and an increasing function $g:I \mapsto \R$, a triple of functions $(f,h,l)$, where $f,h : \hat I \mapsto \R$, $l:I\setminus\hat I\mapsto\R$, is said to be second-order left-monotone and first-order right monotone with respect to $(I,\hat I,g)$ if
	\begin{itemize}
		\item $(f,h)$ is second-order left-monotone with respect to $g$ on $\hat I$: $f \leq g \leq h$ on $\hat I$ and for $x,x'\in\hat I$ with $x < x'$ we have $h(x) \leq h(x')$ and $f(x') \notin (f(x), h(x))$,
		\item $l$ is non-increasing and $l\leq g$ on $I\setminus\hat I$, and $l(x')\notin(f(x),h(x))$ for all $x'\in I\setminus \hat I$ and $x\in\hat I$ with $x<x'$.
	\end{itemize} 
\end{defn}

\begin{cor}\label{cor:BJ}
	Let $(\Omega,\sF,\Prob) = (I \times (0,1), \sB(\Omega), \Prob_X \times Leb((0,1)))$ where $\Prob_X((-\infty,x])= \mu((-\infty,x])$.
	Let $\omega = (x,v)$ and let the canonical random variable $(X,V)$ on $(\Omega,\sF,\Prob)$ be given by $(X(\omega),V(\omega))=(x,v)$.
	Then $X$ has law $\mu$, $V$ is a $U(0,1)$ random variable and $X$ and $V$ are independent. Let $\F = (\sF_1 = \sigma(X), \sF_2 = \sigma(X,V) )$ and set $\mathbf{M} = (\Omega, \sF, \F, \Prob)$.
	
	Suppose $\mu$ is continuous.
	Then there exists $T_d,T_u:I\mapsto\R$ such that $(T_d,T_u)$ is second-order left-monotone with respect to the identity function on $I$ and such that
	if we define $Y(x,v) \in \{T_d(x),T_u(x) \}$ by $Y(x,v) = x$ on $T_d(x)=x=T_u(x)$ and
	\begin{equation} Y(x,v) = T_d(x) I_{\{ v \leq \frac{T_u(x) - x}{T_u(x)-T_d((x)} \}} +  T_u(x) I_{ \{ v > \frac{T_u(x) - x}{T_u(x)-T_d(x)} \} }
	\label{eq:YXVdef}
	\end{equation}
	otherwise, then
	$M = (X,Y(X,V))$ is a $\mathbf{M}$-martingale for which $\sL(X) = \mu$ and $\sL(Y) = \nu$.
\end{cor}

In the case with atoms, $T_d$ and $T_u$ cannot be constructed unless we allow them to be multi-valued. By changing their viewpoint Hobson and Norgilas \cite{HobsonNorgilas:21} constructed the generalised lower and upper functions that support $\pi_{lc}$.

\begin{lem}[Hobson and Norgilas~{\cite[Theorem 7.8]{HobsonNorgilas:21}}]
	\label{lem:construction}
	Let $(\Omega,\sF,\Prob) = ((0,1) \times (0,1), \sB(\Omega), \Leb(\Omega))$. Let $\omega = (u,v)$ and let $(U,V)$ be the canonical random variables on $(\Omega,\sF,\Prob)$ given by $(U(\omega),V(\omega))=(u,v)$ so that $U$ and $V$ are independent $U(0,1)$ random variables. Let $\F = ( \sF_1 = \sigma(U), \sF_2 = \sigma(U,V) )$ and set $\mathbf{M} = (\Omega, \sF, \F, \Prob)$.
	
	Fix $\mu\leq_{c}\nu$ and let $G=G_\mu$ be a quantile function of $\mu$.
	
	Then there exists $R,S:(0,1)\mapsto\R$ such that the pair $(R,S)$ is second-order left-monotone with respect to $G$ on $\sI=(0,1)$ and such that if we define $X(u,v)=X(u)=G(u)$ and $Y(u,v) \in \{R(u),S(u) \}$ by $Y(u,v) = G(u)$ on $R(u)=S(u)$ and
	\begin{equation*}
	Y(u,v) = R(u) I_{\{ v \leq \frac{S(u) - G(u)}{S(u)-R(u)} \}} +  S(u) I_{ \{ v > \frac{S(u) - G(u)}{S(u)-R(u)} \} }
	\end{equation*}
	otherwise, then
	$M = (X(U),Y(U,V))$ is a $\mathbf{M}$-martingale for which $\sL(X) = \mu$ and $\sL(Y) = \nu$.
\end{lem}

Our main goal in later sections is to construct the suitable supporting functions for $\pi_{I}$ in the case when $\mu\leq_{cd}\nu$, and similar in form to the construction in Lemma \ref{lem:construction}. In particular, we will show that the set $\mathbb M$ of martingale points is actually an interval $(-\infty,x^*]$ whose right boundary $x^*$ will be determined explicitly. To the left of $x^*$, $\pi_{I}$ is just a martingale left-curtain coupling (and we can use Corollary \ref{cor:BJ} or Lemma \ref{lem:construction} to embed $\mu\lvert_{(-\infty,x*]}$ into $\nu$ using the shadow $S^\nu(\mu\lvert_{(-\infty,x*]})$), while to the right of $x^*$ the supermartingale left-curtain coupling $\pi_{I}$ is concentrated on a deterministic decreasing map (and $\mu\lvert_{(x^*,\infty)}$ is embedded into $\nu$ through the \textit{antitone} coupling). A special care will be needed in the case when $\mu$ has an atom at $x^*$.

\subsection{Lifted supermartingale transport plans}
Just as Corollary~\ref{cor:BJ} has an equivalent expression via Lemma~\ref{lem:LC}, Theorem~\ref{thm:mainTHM_intro} has an equivalent expression in terms of transport plans, provided that we generalise the notion of a supermartingale transport plan.
Let $(\mu_u)_{0 \leq u \leq 1}$ be a family of measures with $\mu_u(\R)=u$, $\mu_1=\mu$ and $\mu_u \leq \mu_v$ for $0 \leq u \leq v \leq 1$, and let $\lambda$ denote the Lebesgue measure on the unit interval. Then a \textit{lift} (Beiglb\"{o}ck and Juillet \cite{BeiglbockJuillet:16, BeiglbockJuillet:16s}) of $\mu$ with respect to $(\mu_u)_{0 \leq u \leq 1}$ is a probability measure $\hat\mu\in\Pi(\lambda,\mu)$ such that, for all $u\in[0,1]$ and Borel $A\subseteq\R$, $\hat\mu([0,u]\times A)=\mu_u(A)$. A {\it lifted supermartingale transport plan} is a probability measure $\hat{\pi} \in \Pi(\hat{\mu},\nu)$ such that $\int_{\R} y \hat{\pi}_{u,x}(dy)\leq x$, $\hat\mu\textrm{-a.e. }(u,x)$, where $\hat{\pi}_{u,x}$ denotes the disintegration of $\hat\pi\in\Pi(\hat\mu,\nu)$ with respect to $\hat\mu$: $\hat{\pi}(du,dx,dy) = \hat{\mu}(du,dx) \hat{\pi}_{u,x}(dy)$.

One of the insights of Beiglb\"{o}ck and Juillet \cite{BeiglbockJuillet:16,BeiglbockJuillet:16s} is that, when $\mu\leq_c\nu$ and $(\mu_u)_{0 \leq u \leq 1}$ as above, the (martingale) shadow measure induces a family of martingale couplings. In particular the idea is that for all $u\in[0,1]$, $\mu_u$ is mapped to $S^\nu(\mu_u)$. A crucial result making this possible is the fact that if $0<u<v<1$ and $\mu_u \leq \mu_v$ then $S^\nu(\mu_u) \leq S^\nu(\mu_v)$ (recall Lemma~\ref{prop:shadowSum}). More precisely, for any lift $\hat\mu\in\Pi(\lambda,\mu)$, there exists a unique lifted martingale transport plan $\hat{\pi}$
such that for all $u\in[0,1]$ and Borel $A,B\subseteq\R$,  $\hat\pi([0,u]\times A\times\R)=\mu_u(A)$ and $\hat\pi([0,u]\times \R\times B)=S^\nu(\mu_u)(B)$ (see Beiglb\"{o}ck and Juillet \cite[Theorem 2.9]{BeiglbockJuillet:16s}).

The proof of Beiglb\"{o}ck and Juillet \cite[Theorem 2.9]{BeiglbockJuillet:16s} relies on the associativity of the shadow measure. Since this property also holds when $\mu\leq_{cd}\nu$ (see Proposition \ref{thm:shadow_assoc}), we believe that, by replacing $\leq_c$ and $\leq_{pc}$ with $\leq_{cd}$ and $\leq_{pcd}$, respectively, the theorem can be extended to the supermartingale case (for arbitrary lifts). The rigorous proof, however, is left for future work.

In this paper we work with a particular lift instead. In particular, we choose the {\it quantile lift} $\hat{\mu}^Q$ whose support is of the form $\{(u,G(u)): 0<u<1 \}$ where $G$ is a quantile function of $\mu$. Then $\hat{\mu}^Q(du,dx) = du \delta_{G(u)}(dx)$ and for a Borel set $A$, $\hat{\mu}^Q([0,w] \times A) = \int_{0}^w du I_{ \{G(u) \in A \}}$. (This is precisely the lift used by Hobson and Norgilas \cite{HobsonNorgilas:21} to construct the lifted left-curtain martingale coupling.) By analogy with the correspondence between Lemma~\ref{lem:LC} and Corollary~\ref{cor:BJ} we have the following equivalent restatement of Theorem~\ref{thm:mainTHM_intro}:
\begin{cor}\label{cor:construction}
	Let $\mu,\nu$ be probability measures in convex-decreasing order and let $\hat{\mu}^Q$ be the quantile lift of $\mu$. Then there exists a unique regime-switching point $u^*\in[0,1]$ and a triple of measurable functions $R,S : (0,u^*] \mapsto \R$ and $T(u^*,1)\mapsto\R$ such that $(R,S,T)$ is second-order left-monotone and first-order right-monotone with respect to $((0,1),(0,u^*],G)$, and such that if
	$\hat{\pi}^Q(du,dx,dy) = du \delta_{G(u)}(dx) \hat{\pi}^Q_{u,x}(dy)$ (recall $\hat{\mu}$ has support on $\{(u,G(u)) : 0 < u < 1 \}$) then
	$$
	\hat{\pi}^Q_{u,x}(dy) = \hat{\pi}^Q_{u,G(u)} (dy) = I_{\{u\leq u^*\}}\chi_{R(u),G(u),S(u)}(dy)+I_{\{u>u^*\}}\delta_{T(u)}$$
	and $\hat{\pi}^Q$ is the lifted increasing supermartingale transport plan which transports a second marginal $\mu$ to third marginal $\nu$.
\end{cor}

\section{The geometric construction of $\pi_{I}$}\label{sec:Slc}
In this section we construct a generalized (or lifted) version of the supermartingale increasing coupling for two measures $\mu,\nu\in\sP$ with $\mu\leq_{cd}\nu$. In particular, we will split the construction into two parts, martingale and (strict) supermartingale, by explicitly determining the unique regime-switching point.

We begin by introducing a particular parametrization of an initial measure $\mu$. Let $G_\mu:(0,1)\mapsto\R$ be a quantile function of $\mu$. For now, we use an arbitrary version of $G=G_\mu$. For each $u\in(0,1)$ define $\mu_u\in\sM$ by
\begin{equation}\label{eq:muLift}
\mu_u=\mu\lvert_{(-\infty,G(u))}+(u-\mu\lvert_{(-\infty,G(u))}(\R))\delta_{G(u)}.
\end{equation}
It is easy to verify that $\mu_u$ does not depend on the choice of $G$. Note that $\mu_u(\R)=u$, $u\in(0,1)$. For $v,u\in(0,1)$ with $v<u$, $\mu_u-\mu_v\geq0$ (i.e., $(\mu_u-\mu_v)\in\sM$), $(\mu_u-\mu_v)(\R)=(u-v)$ and the support of $(\mu_u-\mu_v)$ is contained in $[G(v),G(u)]$. Furthermore, we treat $\mu_0$ as the zero measure, and set $\mu_1=\mu$.

Note that, for $u\in(0,1)$, we have $P_{\mu_u}(k)=P_\mu(k)$ for $k\leq G(u)$, while $P_{\mu_u}(k)\leq P_\mu(k)$ for $k> G(u)$. In particular,
	\begin{equation*}
	P_{\mu_u}(k)=P_\mu(k\wedge G(u))+u(k-G(u))^+,\quad k\in\R,
	\end{equation*}
	and thus, $P_{\mu_u}(\cdot)$ is linear on $[G(u),\infty)$ and $u\in\partial P_\mu(G(u))$, so that $P'_\mu(G(u)-)\leq u\leq P'_\mu(G(u)+)$.
\subsection{The regime-switching point $u^*$}\label{sec:u^*}
For $\eta,\chi\in\sM$ with $\eta\leq_{pcd}\chi$, recall the definition of $c_{\eta,\chi}:=\sup_{k\in\R}\{(\chi(\R)-\eta(\R))k-(\bar\chi-\bar\eta)-P_\chi(k)+P_\eta(k)\}\in[0,\infty)$, which was used in the proof of Theorem \ref{thm:shadow_potential}, see Figure \ref{fig:C}. Note that $c_{\eta,\chi}=\sup_{k\in\R}\{C_{\eta}(k)-C_{\chi}(k)\}$.

\begin{figure}[H]
	\centering
	\begin{tikzpicture}[scale=1]

	\begin{axis}[%
	width=6.028in,
	height=2.754in,
	at={(1.011in,0.642in)},
	scale only axis,
	xmin=-11,
	xmax=11,
	ymin=-0.75,
	ymax=3,
	axis line style={draw=none},
	ticks=none
	]
\addplot [color=black, dashed, line width=1.0pt, forget plot]
table[row sep=crcr]{%
	2.5	1.390625\\
	2.53128286567997	1.40046205913031\\
	2.56818181818182	1.41222236570248\\
	2.6038739617008	1.42375997552646\\
	2.63636363636364	1.43440082644628\\
	2.66981244424612	1.44549365546571\\
	2.70454545454545	1.45716038223141\\
	2.74216116198107	1.46996548989234\\
	2.77272727272727	1.48050103305785\\
	2.8107222959441	1.49375998905733\\
	2.84090909090909	1.50442277892562\\
	2.87379754287642	1.51616951984016\\
	2.90909090909091	1.52892561983471\\
	2.94186292595522	1.54090984219436\\
	2.97727272727273	1.55400955578512\\
	3.01530153897903	1.56825271068558\\
	3.04545454545455	1.57967458677686\\
	3.08201720135142	1.59367687683913\\
	3.11363636363636	1.60592071280992\\
	3.14951910985763	1.6199669139599\\
	3.18181818181818	1.6327479338843\\
	3.21559287091146	1.64625234446604\\
	3.25	1.66015625\\
	3.28409984781444	1.67408198815093\\
	3.31818181818182	1.68814566115702\\
	3.35362780135877	1.7029262143779\\
	3.38636363636364	1.71671616735537\\
	3.42318691966234	1.73238804293421\\
	3.45454545454545	1.74586776859504\\
	3.48805135860873	1.76040639251826\\
	3.52272727272727	1.77560046487603\\
	3.55704200638889	1.79078423970094\\
	3.59090909090909	1.80591425619835\\
	3.62502634462459	1.82130099995139\\
	3.65909090909091	1.83680914256198\\
	3.69566847163685	1.85362284076567\\
	3.72727272727273	1.86828512396694\\
	3.7597556933937	1.88348517962539\\
	3.79545454545455	1.90034220041322\\
	3.83122796384398	1.91739423193375\\
	3.86363636363636	1.93298037190083\\
	3.88936854816097	1.94544923146398\\
	3.93181818181818	1.96619963842975\\
	3.96527578564731	1.98271325351505\\
	4	2\\
};
\addplot [color=black, dashed, line width=1.0pt, forget plot]
table[row sep=crcr]{%
	-2.5	0.140625\\
	-2.45828951242671	0.148554451718342\\
	-2.40909090909091	0.158186983471074\\
	-2.36150138439893	0.167792357082914\\
	-2.31818181818182	0.176782024793388\\
	-2.27358340767184	0.186282140641623\\
	-2.22727272727273	0.196410123966942\\
	-2.1771184506919	0.207681071425493\\
	-2.13636363636364	0.217071280991736\\
	-2.08570360540786	0.229033167896779\\
	-2.04545454545455	0.238765495867769\\
	-2.00160327616477	0.24959934161471\\
	-1.95454545454545	0.261234504132231\\
	-1.9108494320597	0.271790902999819\\
	-1.86363636363636	0.282928719008264\\
	-1.81293128136129	0.294580010566356\\
	-1.77272727272727	0.303589876033058\\
	-1.72397706486477	0.314243942488765\\
	-1.68181818181818	0.323217975206612\\
	-1.63397452018982	0.333132954210653\\
	-1.59090909090909	0.341813016528926\\
	-1.54587617211805	0.350641678779853\\
	-1.5	0.359375\\
	-1.45453353624741	0.367770761995725\\
	-1.40909090909091	0.375903925619835\\
	-1.36182959818831	0.384088759093642\\
	-1.31818181818182	0.39139979338843\\
	-1.26908410711689	0.399339095566459\\
	-1.22727272727273	0.405862603305785\\
	-1.18259818852169	0.412591345281576\\
	-1.13636363636364	0.419292355371901\\
	-1.09061065814814	0.425660524520854\\
	-1.04545454545455	0.431689049586777\\
	-0.999964873833882	0.43750439084788\\
	-0.954545454545455	0.443310950413223\\
	-0.905775371150867	0.449832971148751\\
	-0.863636363636364	0.455707644628099\\
	-0.820325742141739	0.46197695966584\\
	-0.772727272727273	0.469137396694215\\
	-0.725029381541366	0.4765968798708\\
	-0.681818181818182	0.48360020661157\\
	-0.647508602452041	0.489327061277577\\
	-0.590909090909091	0.499096074380165\\
	-0.546298952470256	0.507077920974317\\
	-0.5	0.515625\\
};
\addplot [color=black, dashed, line width=1.0pt, forget plot]
table[row sep=crcr]{%
	0.5	0.765625\\
	0.541710487573291	0.778768262665004\\
	0.590909090909091	0.794550619834711\\
	0.638498615601069	0.810104684033047\\
	0.681818181818182	0.824509297520661\\
	0.726416592328157	0.839584214682642\\
	0.772727272727273	0.855501033057851\\
	0.822881549308096	0.873041265089005\\
	0.863636363636364	0.887525826446281\\
	0.914296394592138	0.905820217220796\\
	0.954545454545455	0.920583677685951\\
	0.998396723835225	0.936898932094113\\
	1.04545454545455	0.954416322314049\\
	1.0891505679403	0.970434723992356\\
	1.13636363636364	0.987474173553719\\
	1.18706871863871	1.00546360039619\\
	1.22727272727273	1.01949896694215\\
	1.27602293513523	1.03624680938067\\
	1.31818181818182	1.05049070247934\\
	1.36602547981018	1.06638613918693\\
	1.40909090909091	1.08044938016529\\
	1.45412382788195	1.0949071572651\\
	1.5	1.109375\\
	1.54546646375259	1.1234540699648\\
	1.59090909090909	1.13726756198347\\
	1.63817040181169	1.1513600593201\\
	1.68181818181818	1.1641270661157\\
	1.73091589288312	1.17820358217685\\
	1.77272727272727	1.18995351239669\\
	1.81740181147831	1.20226657171636\\
	1.86363636363636	1.21474690082645\\
	1.90938934185186	1.22683419225234\\
	1.95454545454545	1.23850723140496\\
	2.00003512616612	1.25000878161864\\
	2.04545454545455	1.26149276859504\\
	2.09422462884913	1.27411104975489\\
	2.13636363636364	1.28525309917355\\
	2.17967425785826	1.29693624189812\\
	2.22727272727273	1.31004648760331\\
	2.27497061845863	1.32346820717813\\
	2.31818181818182	1.3358729338843\\
	2.35249139754796	1.34588848597107\\
	2.40909090909091	1.36273243801653\\
	2.45370104752974	1.37629055191554\\
	2.5	1.390625\\
};
\addplot [color=black, dashed, line width=1.0pt, forget plot]
table[row sep=crcr]{%
	4	2\\
	4.06256573135994	2.03030424799482\\
	4.13636363636364	2.06353305785124\\
	4.2077479234016	2.09308416178138\\
	4.27272727272727	2.11776859504132\\
	4.33962488849224	2.14097617802528\\
	4.40909090909091	2.16270661157025\\
	4.48432232396214	2.18351913360905\\
	4.54545454545454	2.19886363636364\\
	4.62144459188821	2.21786114797205\\
	4.68181818181818	2.23295454545455\\
	4.74759508575284	2.24939877143821\\
	4.81818181818182	2.26704545454545\\
	4.88372585191044	2.28343146297761\\
	4.95454545454546	2.30113636363636\\
	5.03060307795806	2.32015076948951\\
	5.09090909090909	2.33522727272727\\
	5.16403440270284	2.35350860067571\\
	5.22727272727273	2.36931818181818\\
	5.29903821971527	2.38725955492882\\
	5.36363636363636	2.40340909090909\\
	5.43118574182293	2.42029643545573\\
	5.5	2.4375\\
	5.56819969562888	2.45454992390722\\
	5.63636363636364	2.47159090909091\\
	5.70725560271753	2.48931390067938\\
	5.77272727272727	2.50568181818182\\
	5.84637383932467	2.52409345983117\\
	5.90909090909091	2.53977272727273\\
	5.97610271721746	2.55652567930437\\
	6.04545454545454	2.57386363636364\\
	6.11408401277778	2.59102100319445\\
	6.18181818181818	2.60795454545455\\
	6.25005268924918	2.62501317231229\\
	6.31818181818182	2.64204545454545\\
	6.3913369432737	2.66033423581842\\
	6.45454545454546	2.67613636363636\\
	6.51951138678739	2.69237784669685\\
	6.59090909090909	2.71022727272727\\
	6.66245592768795	2.72811398192199\\
	6.72727272727273	2.74431818181818\\
	6.77873709632194	2.75718427408048\\
	6.86363636363636	2.77840909090909\\
	6.93055157129462	2.79513789282365\\
	7	2.8125\\
};
\addplot [color=black, dashed, line width=1.0pt, forget plot]
table[row sep=crcr]{%
	-7	0\\
	-6.90615140296009	0\\
	-6.79545454545455	0\\
	-6.6883781148976	0\\
	-6.59090909090909	0\\
	-6.49056266726165	0\\
	-6.38636363636364	0\\
	-6.27351651405678	0\\
	-6.18181818181818	0\\
	-6.06783311216769	0\\
	-5.97727272727273	0\\
	-5.87860737137074	0\\
	-5.77272727272727	0\\
	-5.67441122213433	0\\
	-5.56818181818182	0\\
	-5.45409538306291	0\\
	-5.36363636363636	0\\
	-5.25394839594574	0\\
	-5.15909090909091	0\\
	-5.0514426704271	0\\
	-4.95454545454546	0\\
	-4.85322138726561	0\\
	-4.75	0\\
	-4.64770045655668	0\\
	-4.54545454545454	0\\
	-4.4391165959237	0\\
	-4.34090909090909	0\\
	-4.23043924101299	0\\
	-4.13636363636364	0\\
	-4.03584592417381	0\\
	-3.93181818181818	0.000290547520661158\\
	-3.82887398083333	0.00183025715223957\\
	-3.72727272727273	0.00464876033057852\\
	-3.62492096612623	0.00879276760322985\\
	-3.52272727272727	0.0142368285123967\\
	-3.41299458508945	0.0215359598208941\\
	-3.31818181818182	0.0290547520661157\\
	-3.22073291981891	0.0379535738908722\\
	-3.11363636363636	0.0491025309917355\\
	-3.00631610846807	0.0617129797681272\\
	-2.90909090909091	0.0743801652892562\\
	-2.83189435551709	0.0852794247920518\\
	-2.70454545454545	0.104887654958678\\
	-2.60417264305808	0.121770875649217\\
	-2.5	0.140625\\
};
\addplot [color=black, dashed, line width=1.0pt, forget plot]
table[row sep=crcr]{%
	-0.5	0.515625\\
	-0.479144756213355	0.519562542034584\\
	-0.454545454545455	0.524276859504132\\
	-0.430750692199466	0.528908961877029\\
	-0.409090909090909	0.533186983471074\\
	-0.386791703835921	0.537652562925788\\
	-0.363636363636364	0.542355371900826\\
	-0.338559225345952	0.54752409048019\\
	-0.318181818181818	0.551782024793388\\
	-0.292851802703931	0.557147185470701\\
	-0.272727272727273	0.56146694214876\\
	-0.250801638082387	0.566230931833454\\
	-0.227272727272727	0.571410123966942\\
	-0.205424716029852	0.576281278114784\\
	-0.181818181818182	0.581611570247934\\
	-0.156465640680647	0.587413683374439\\
	-0.136363636363636	0.592071280991736\\
	-0.111988532432386	0.597786706354176\\
	-0.0909090909090909	0.602789256198347\\
	-0.0669872600949105	0.608533640789711\\
	-0.0454545454545455	0.613765495867769\\
	-0.0229380860590236	0.619298363222247\\
	0	0.625\\
	0.0227332318762936	0.630683307969073\\
	0.0454545454545455	0.636363636363636\\
	0.0690852009058446	0.642271300226461\\
	0.0909090909090909	0.647727272727273\\
	0.115457946441558	0.653864486610389\\
	0.136363636363636	0.659090909090909\\
	0.158700905739154	0.664675226434789\\
	0.181818181818182	0.670454545454545\\
	0.204694670925928	0.676173667731482\\
	0.227272727272727	0.681818181818182\\
	0.250017563083059	0.687504390770765\\
	0.272727272727273	0.693181818181818\\
	0.297112314424567	0.699278078606142\\
	0.318181818181818	0.704545454545455\\
	0.33983712892913	0.709959282232283\\
	0.363636363636364	0.715909090909091\\
	0.387485309229317	0.721871327307329\\
	0.409090909090909	0.727272727272727\\
	0.426245698773979	0.731561424693495\\
	0.454545454545455	0.738636363636364\\
	0.476850523764872	0.744212630941218\\
	0.5	0.75\\
};
\addplot [color=blue, dotted, line width=1.0pt, forget plot]
table[row sep=crcr]{%
	-4.25	0\\
	-4.01537850740024	0.0586553731499402\\
	-3.73863636363636	0.127840909090909\\
	-3.47094528724399	0.194763678189002\\
	-3.22727272727273	0.255681818181818\\
	-2.97640666815412	0.318398332961471\\
	-2.71590909090909	0.383522727272728\\
	-2.43379128514196	0.45405217871451\\
	-2.20454545454545	0.511363636363636\\
	-1.91958278041923	0.582604304895194\\
	-1.69318181818182	0.639204545454546\\
	-1.44651842842686	0.700870392893286\\
	-1.18181818181818	0.767045454545455\\
	-0.936028055335836	0.828492986166041\\
	-0.670454545454545	0.894886363636364\\
	-0.385238457657277	0.966190385585681\\
	-0.159090909090909	1.02272727272727\\
	0.115129010135656	1.09128225253391\\
	0.352272727272728	1.15056818181818\\
	0.621393323932257	1.21784833098306\\
	0.863636363636363	1.27840909090909\\
	1.11694653183598	1.341736632959\\
	1.375	1.40625\\
	1.6307488586083	1.47018721465208\\
	1.88636363636364	1.53409090909091\\
	2.15220851019075	1.60055212754769\\
	2.39772727272727	1.66193181818182\\
	2.67390189746752	1.73097547436688\\
	2.90909090909091	1.78977272727273\\
	3.16038518956548	1.85259629739137\\
	3.42045454545455	1.91761363636364\\
	3.67781504791669	1.98195376197917\\
	3.93181818181818	2.04545454545455\\
	4.18769758468441	2.1094243961711\\
	4.44318181818182	2.17329545454546\\
	4.71751353727637	2.24187838431909\\
	4.95454545454546	2.30113636363636\\
	5.19816770045272	2.36204192511318\\
	5.46590909090909	2.42897727272727\\
	5.73420972882982	2.49605243220746\\
	5.97727272727273	2.55681818181818\\
	6.17026411120727	2.60506602780182\\
	6.48863636363636	2.68465909090909\\
	6.73956839235481	2.7473920980887\\
	7	2.8125\\
};
\addplot [color=red, dashdotted, line width=1.0pt, forget plot]
table[row sep=crcr]{%
	-2.5	0\\
	-2.30187518402687	0.0495312039932831\\
	-2.06818181818182	0.107954545454546\\
	-1.84213157589492	0.164467106026269\\
	-1.63636363636364	0.215909090909091\\
	-1.42452118644125	0.268869703389686\\
	-1.20454545454545	0.323863636363636\\
	-0.966312640786543	0.383421839803364\\
	-0.772727272727273	0.431818181818182\\
	-0.532092125687346	0.491976968578164\\
	-0.340909090909091	0.539772727272727\\
	-0.13261556178268	0.59184610955433\\
	0.0909090909090908	0.647727272727273\\
	0.298465197716405	0.699616299429101\\
	0.522727272727273	0.755681818181818\\
	0.763576413533855	0.815894103383464\\
	0.954545454545455	0.863636363636364\\
	1.18610894189233	0.921527235473083\\
	1.38636363636364	0.971590909090909\\
	1.61362102909835	1.02840525727459\\
	1.81818181818182	1.07954545454545\\
	2.03208818243928	1.13302204560982\\
	2.25	1.1875\\
	2.46596570282479	1.2414914257062\\
	2.68181818181818	1.29545454545455\\
	2.90630940860552	1.35157735215138\\
	3.11363636363636	1.40340909090909\\
	3.3468504911948	1.4617126227987\\
	3.54545454545455	1.51136363636364\\
	3.75765860452197	1.56441465113049\\
	3.97727272727273	1.61931818181818\\
	4.19459937379631	1.67364984344908\\
	4.40909090909091	1.72727272727273\\
	4.62516684928906	1.78129171232227\\
	4.84090909090909	1.83522727272727\\
	5.07256698703338	1.89314174675835\\
	5.27272727272727	1.94318181818182\\
	5.47845272482674	1.99461318120668\\
	5.70454545454546	2.05113636363636\\
	5.93111043767851	2.10777760941963\\
	6.13636363636364	2.15909090909091\\
	6.29933413835281	2.1998335345882\\
	6.56818181818182	2.26704545454546\\
	6.78007997576628	2.32001999394157\\
	7	2.375\\
};
\addplot [color=black, line width=.5pt, forget plot]
table[row sep=crcr]{%
	-4	0\\
	-3.7705923183469	0\\
	-3.5	0\\
	-3.23825761419412	0\\
	-3	0\\
	-2.75470874219514	0\\
	-2.5	0\\
	-2.22415147880547	0\\
	-2	0\\
	-1.72136982974324	0\\
	-1.5	0\\
	-1.25881801890626	0\\
	-1	0\\
	-0.759671876328373	0\\
	-0.5	0\\
	-0.221122047487116	0\\
	0	0\\
	0.268126143243753	0\\
	0.5	0\\
	0.763140138955985	0\\
	1	0\\
	1.24768105335074	0\\
	1.5	0\\
	1.75006555063923	0\\
	2	0\\
	2.25993720996429	0\\
	2.5	0\\
	2.77003741085713	0\\
	3	0\\
	3.2457099631307	0\\
	3.5	0\\
	3.7516413801852	0\\
	4	0\\
	4.25019319391365	0\\
	4.5	0\\
	4.76823545867023	0\\
	5	0\\
	5.23820841822043	0\\
	5.5	0\\
	5.76233840152249	0\\
	6	0\\
	6.18870268651377	0\\
	6.5	0\\
	6.74535576141359	0\\
	7	0\\
};
	
	\node (R)[scale=1] at (-4.25,-0.25) {$\frac{\overline\chi-\overline\eta}{\chi(\R)-\eta(\R)}$};
	\draw[gray,dashed,thin] (-4.25,-0.1) -- (-4.25,0.05);
	
	\node (Q)[scale=1] at (0,-0.25) {$\overline k$};
	\draw[gray,dashed,thin] (0,-0.1) -- (0,0.65);
	
	\node (G)[scale=1] at (-2.5,-0.6) {$\frac{\overline\chi-\overline\eta}{\chi(\R)-\eta(\R)}+c_{\eta,\chi}$};
	\draw[gray,dashed,thin] (-2.5,-0.5) -- (-2.5,0.05);

	\node (E)[scale=1,red] at (7,1.7) {$k\mapsto(l(k)-c_{\eta,\chi})$};
	\node (D)[scale=1,blue] at (2,2) {$k\mapsto l(k)$};
	\node (P)[scale=1,black] at (-4,1.5) {$k\mapsto(P_\chi(k)-P_\eta(k))$};
	\draw [gray,->] (-4,1.3) to[out=320, in=180] (2.6,1.5);
	\end{axis}
	
	\end{tikzpicture}
	\caption{The geometrical representation of the case $c_{\eta,\chi}>0$ for $\eta\leq_{pcd}\chi$. The dashed curve represents $(P_\chi-P_\eta)$. The dotted line corresponds to $k\mapsto l(k)=(\chi(\R)k-\overline\chi)-(\eta(\R)k-\overline\eta)$. Note that $(P_\chi-P_\eta)$ converges to $l$ at $\infty$. The dash-dotted curve represents $k\mapsto (l(k)-c_{\eta,\chi})$. Note that $(l(k)-c_{\eta,\chi})\leq(P_\chi(k)-P_\eta(k))$ for all $k\in\R$, but there exists $\overline k\in\R$ such that $(l(\overline k)-c_{\eta,\chi})=(P_\chi(\overline k)-P_\eta(\overline k))$.}
	\label{fig:C}
\end{figure}
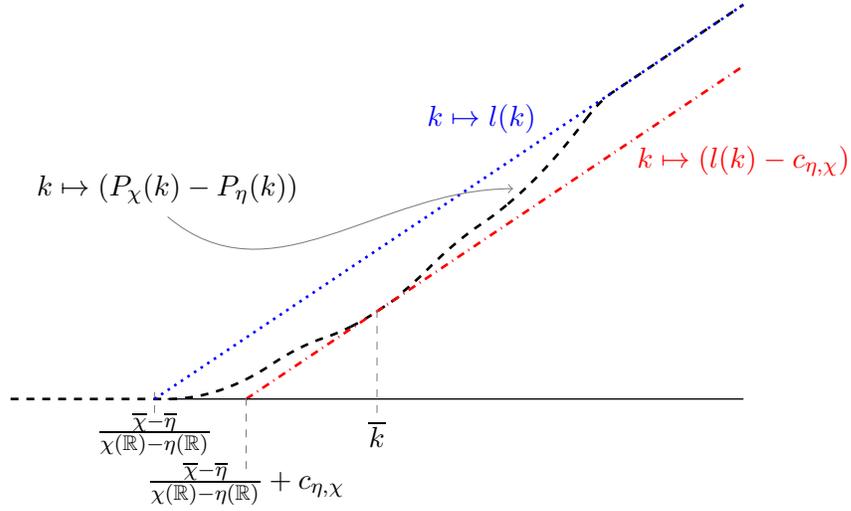
 We will use $c_{\eta,\chi}$ with $\eta=\mu_u$, for $u\in[0,1]$, and $\chi=\nu$. Define $c:[0,1]\mapsto[0,\infty)$ by
\begin{equation}\label{eq:c}
c(u):=c_{\mu_u,\nu}=\sup_{k\in\R}\{C_{\mu_u}(k)-C_{\nu}(k)\},\quad u\in(0,1),\quad c(0)=0,\quad c(1)=c_{\mu,\nu}.
\end{equation}
(That $c(\cdot)$ is non-negative is due to the fact that $\mu_u\leq\mu\leq_{cd}\nu$, so that $\mu_u\leq_{pcd}\nu$.) The following lemma summarizes the properties of $c(\cdot)$.
\begin{lem}\label{lem:cproperties}
$c(\cdot)$, defined in \eqref{eq:c}, is non-decreasing and lower semi-continuous, and thus also left-continuous. 
\end{lem}
\begin{proof}
Fix $u,v\in(0,1)$ with $u<v$. Then $\mu_u\leq\mu_v$. For each $k\in\R$, $x\mapsto (x-k)^+$ is non-negative, and therefore $C_{\mu_u}(k)\leq C_{\mu_v}(k)$. It follows that $c(\cdot)$ is non-decreasing:
$$
c(u)=\sup_{k\in\R}\{C_{\mu_u}(k)-C_\nu(k)\}\leq \sup_{k\in\R}\{C_{\mu_v}(k)-C_\nu(k)\}=c(v).
$$

We now turn to the continuity of $c(\cdot)$.

If $u<v$, then the support of $(\mu_v-\mu_u)$ is contained in $[G(u),G(v)]$ and we have that, for each $k\in\R$,
$$
0\leq C_{\mu_v}(k)-C_{\mu_u}(k)=\int_\R(x-k)^+(\mu_v-\mu_u)(dx)\leq(G(v)-k)^+(\mu_v-\mu_u)(\R),$$
and it follows that
$$
0\leq\lim_{v\downarrow u}\{C_{\mu_v}(k)-C_{\mu_u}(k)\}\leq(G(u+)-k)^+\lim_{v\downarrow u}(v-u)=0.
$$
On the other hand, if $v<u$, then the support of $(\mu_u-\mu_v)$ is contained in $[G(v),G(u)]$, and therefore
$$
0\leq C_{\mu_u}(k)-C_{\mu_v}(k)\leq(G(u)-k)^+(u-v),\quad k\in\R.
$$
It follows that, for each $k\in\R$,
$$
0\leq  \lim_{v\uparrow u}\{C_{\mu_u}(k)-C_{\mu_v}(k)\}\leq(G(u)-k)^+\lim_{v\uparrow u}(u-v)=0.
$$
Combining both cases we have that, for each $k\in\R$, $u\mapsto (C_{\mu_u}(k)-C_\nu(k))$ is continuous. Since the supremum of any collection of continuous functions is lower semi-continuous (l.s.c.), we conclude that $c(\cdot)$ is l.s.c.

Finally, fix $u\in(0,1]$. Then $c(v)\leq c(u)\leq\liminf_{v\uparrow u}c(v)$ for all $0<v<u$. Letting $v\uparrow u$ shows that $c(\cdot)$ is left-continuous.
\end{proof}
We have the following representation of $c(\cdot)$.
\begin{lem}\label{lem:cmeandiff}
	For each $u\in(0,1)$,
	$$
	c(u)=\overline{\mu_u}-\overline{S^\nu(\mu_u)}.
	$$
	\end{lem}
\begin{proof}
By Lemma \ref{lem:CSpotential} with $f=P_{\mu_u}$ and $g=P_\nu$ we have that $(P_\nu-P_{\mu_u})^c\in\sD(\nu(\R)-u,\overline\nu-\overline{\mu_u}+c(u))$. Let $\theta_u:=((P_\nu-P_{\mu_u})^c)''\in\sM$, then $\overline{\theta_u}=\overline\nu-\overline{\mu_u}+c(u)$. By Theorem \ref{thm:shadow_potential}, $S^\nu(\mu_u)=\nu-\theta_u$, and therefore $\overline{S^\nu(\mu_u)}=\overline\nu-\overline{\theta_u}=\overline{\mu_u}-c(u)$. It follows that $\overline{\mu_u}-\overline{S^\nu(\mu_u)}=c(u)$, as claimed.
\end{proof}
Let $u,v\in(0,1)$ with $u<v$. Then $\mu_u\leq\mu_v$ and, by Lemma \ref{prop:shadowSum}, $S^\nu(\mu_u)\leq S^\nu(\mu_v)$. Combining this with Lemma \ref{lem:cmeandiff} and the associativity of the shadow measure (see Proposition \ref{thm:shadow_assoc}) it follows that
$$
c(v)=c(u)+\overline{\mu_v-\mu_u}-\overline{S^{\nu-S^\nu(\mu_u)}(\mu_v-\mu_u)},\quad u,v\in(0,1)\textrm{ with }u<v,
$$
and thus we have the following property.
\begin{cor}\label{cor:pcVSpcd}
	Fix $u,v\in(0,1)$ with $u<v$. Then 
	$$c(u)=c(v)\quad\textrm{if and only if}\quad(\mu_v-\mu_u)\leq_{pc}{(\nu-S^\nu(\mu_u))}.$$
	\end{cor} 
\begin{proof}
	Since $(\mu_v-\mu_u)\leq_{pcd}{(\nu-S^\nu(\mu_u))}$, the statement follows immediately by noting that $(\mu_v-\mu_u)\leq_{pc}{(\nu-S^\nu(\mu_u))}$ if and only if $\overline{\mu_v-\mu_u}=\overline{S^{\nu-S^\nu(\mu_u)}(\mu_v-\mu_u)}$. See Remark \ref{rem:cVScd}.
\end{proof}
Corollary \ref{cor:pcVSpcd} motivates us to introduce a special point $u^*$ that will separate the construction of the supermartingale coupling into two parts.

Define
$$
u^*:=u^*_{\mu,\nu}:=\sup\{u\in(0,1):c(u)=0\}\in[0,1],
$$
with convention $\sup\emptyset=0$. Note that, since $c(\cdot)$ is non-decreasing and left-continuous, $c(u^*)=0$. Then by Lemma \ref{lem:cmeandiff} and Remark \ref{rem:cVScd} we have that $\mu_{u^*}\leq_{c}S^\nu(\mu_{u^*})\leq\nu$, and thus we can transport $\mu_{u^*}$ to $S^\nu(\mu_{u^*})$ using any martingale coupling $\pi\in\Pi_M(\mu_{u^*},S^\nu(\mu_{u^*}))$ (in Section \ref{sec:Constr} for this we will actually choose the generalised left-curtain martingale coupling).

Now consider the remaining initial mass $(\mu-\mu_{u^*})$. We will show that if $\mu_{u^*}$ is transported to $S^\nu(\mu_{u^*})$ then no portion of mass from $(\mu-\mu_{u^*})$ can be transported to $(\nu-S^\nu(\mu_{u^*}))$ using a martingale. This proves that $u^*$ defines a unique regime-switching point: to the left of (and including) $u^*$ we use a martingale, to the right of $u^*$ we must use a strict supermartingale. 

Before we proceed observe that if $c(u)=0$ for all $u\in(0,1)$, or equivalently if $u^*=1$, then $\overline\mu=\overline\nu$ and therefore $\mu\leq_{c}\nu$, so that $\Pi_S(\mu,\nu)=\Pi_M(\mu,\nu)$. Then by embedding $\mu_u$ to $S^\nu(\mu_u)$ for each $u\in(0,1)$, in Section \ref{sec:u=1} we will recover the generalised version of the martingale left-curtain coupling. In the rest of this section we focus on the supermartingale case $0\leq u^*<1$.

The following crucial result shows that there does not exist $\eta,\chi\in\sM$ with $\eta\leq(\mu-\mu_{u^*})$, $\chi\leq(\nu-S^\nu(\mu_{u^*}))$ and such that $\eta\leq_c\chi$.
\begin{prop}\label{prop:support} Suppose $\mu,\nu\in\sM$ with $\mu\leq_{cd}\nu$ and $u^*\in[0,1)$.
	
	Then the support of $(\mu-\mu_{u^*})$ is strictly to the right of the support of $(\nu-S^{\nu}(\mu_{u^*}))$:
	$$
	r_{\nu-S^{\nu}(\mu_{u^*})}\leq\ell_{\mu-\mu_{u^*}}\quad\textrm{and}\quad (\mu-\mu_{u^*})(\{\ell_{\mu-\mu_{u^*}}\})\wedge(\nu-S^{\nu}(\mu_{u^*}))(\{\ell_{\mu-\mu_{u^*}}\})=0.
	$$
\end{prop}
\begin{proof}
	(Note that if $u^*=1$ then there is nothing to prove since in this case both $\mu-\mu_{u^*}$ and $\nu-S^\nu(\mu_{u^*})$ are the zero measures.) First suppose that $u^*\in(0,1)$.
	
	\textit{Case 1.} Suppose $(\mu-\mu_{u^*})(\{G(u^*)\})>0$. We either have $(\nu-S^{\nu}(\mu_{u^*}))(\{G(u^*)\})>0$ or $(\nu-S^{\nu}(\mu_{u^*}))(\{G(u^*)\})=0$. In the former case let $h:=(\mu-\mu_{u^*})(\{G(u^*)\})\wedge(\nu-S^{\nu}(\mu_{u^*}))(\{G(u^*)\})>0$. Then 
	$$
	(\mu_{u^*+h}-\mu_{u^*})=h\delta_{G(u^*)}\leq(\nu-S^{\nu}(\mu_{u^*}))(\{G(u^*)\}) \delta_{G(u^*)}\leq(\nu-S^{\nu}(\mu_{u^*}))
	$$ and therefore $C_{\mu_{u^*+h}-\mu_{u^*}}\leq C_{\nu-S^{\nu}(\mu_{u^*})}$ everywhere. Using Lemma \ref{prop:pcd=pc} we then have that $(\mu_{u^*+h}-\mu_{u^*})\leq_{pc}(\nu-S^{\nu}(\mu_{u^*}))$. It follows that $(\mu_{u^*+h}-\mu_{u^*})\leq_c S^{\nu-S^\nu(\mu_{u^*})}(\mu_{u^*+h}-\mu_{u^*})$ and therefore $\overline{S^{\nu-S^\nu(\mu_{u^*})}(\mu_{u^*+h}-\mu_{u^*})}=\overline{(\mu_{u^*+h}-\mu_{u^*})}$. Hence, $c(u^*+h)=c(u^*)=0$, contradicting the maximality of $u^*$.
	
	Now suppose that $(\nu-S^{\nu}(\mu_{u^*}))(\{G(u^*)\})=0$. If $\sI_{\nu-S^{\nu}(\mu_{u^*})}\subseteq (-\infty,G(u^*))$ then we are done. On the other hand, if $\sI_{\nu-S^{\nu}(\mu_{u^*})}\subseteq (G(u^*),\infty)$, fix $v\in(u^*,\mu((-\infty,G(u^*)])]$. Then for any $\eta_v\in\sM$ with $\eta_v\leq(\nu-S^{\nu}(\mu_{u^*}))$ and $\eta_v(\R)=(v-u^*)$ we have that $\sI_{\eta_v}\subseteq(G(u^*),\infty)$, and therefore $\overline{\eta_v}>(v-u^*)\overline{\delta_{G(u^*)}}=\overline{(\mu_v-\mu_{u^*})}$, contradicting the fact that $(\mu_v-\mu_{u^*})\leq_{pcd}{\nu-S^\nu(\mu_{u^*})}$. Hence suppose that $\tilde h:=(\nu-S^\nu(\mu_{u^*}))((-\infty,G(u^*)))\wedge(\nu-S^\nu(\mu_{u^*}))((G(u^*),\infty))\in(0,1-u^*)$. Let $\tilde u:=(u^*+\tilde h)\wedge\mu((-\infty,G(u^*)])$. 

	Using Lemma \ref{lem:disjointSupport} with $\eta=(\mu_{\tilde u}-\mu_{u^*})$ and $\chi=(\nu-S^\nu(\mu_{u^*}))$ we have that $(\mu_{\tilde u}-\mu_{u^*})\leq_{pc}(\nu-S^\nu(\mu_{u^*}))$. But then $(\mu_{\tilde u}-\mu_{u^*})\leq_c S^{\nu-S^\nu(\mu_{u^*})}(\mu_{\tilde u}-\mu_{u^*})$. It follows that $\overline{\mu_{\tilde u}-\mu_{u^*}}=\overline{S^{\nu-S^\nu(\mu_{u^*})}(\mu_{\tilde u}-\mu_{u^*})}$, and hence $c(\tilde u)=c(u^*)=0$, contradicting the maximality of $u^*$. We conclude that $(\nu-S^\nu(\mu_{u^*}))(\R)=(\nu-S^\nu(\mu_{u^*}))((-\infty,G(u^*)))$.
	
	\textit{Case 2.} Now suppose $G(u^*)<G(u^*+)$ and $(\mu-\mu_{u^*})([G(u^*),G(u^*+)])=(\mu-\mu_{u^*})(\{G(u^*+)\})>0$. By replacing $G(u^*)$ with $G(u^*+)$ we can use the same arguments as in \textit{Case 1} and conclude that $(\nu-S^\nu(\mu_{u^*}))(\R)=(\nu-S^\nu(\mu_{u^*}))((-\infty,G(u^*+)))$.
	
	\textit{Case 3.} Suppose that $(\mu-\mu_{u^*})([G(u^*),G(u^*+)])=0$. (Note that in this case $G(u)>G(u^*+)$ for all $u\in(u^*,1)$.) Then either $\sI_{\mu-\mu_{u^*}}=(G(u^*+),r_\mu)$ or $\sI_{\mu-\mu_{u^*}}=(G(u^*+),r_\mu]$. To ease the notation let $\tilde\mu:=(\mu-\mu_{u^*})$ and $\tilde\nu:=(\nu-S^\nu(\mu_{u^*}))$. If $\tilde\nu(\R)=\tilde\nu((-\infty,G(u^*+)])$, i.e., $\sI_{\tilde\nu}\subseteq(-\infty,G(u^*+)]$, then we are done. Hence we can suppose that either $\tilde\nu((-\infty,G(u^*+)])\wedge\tilde\nu((G(u^*+),\infty))>0$ or $\tilde\nu(\R)=\tilde\nu((G(u^*+),\infty))$.
	
		We first argue that if $\tilde\nu((-\infty,G(u^*+)])\wedge\tilde\nu((G(u^*+),\infty))>0$ then there exists $\hat u\in(u^*,1)$ with $c(\hat u)=c(u^*)=0$, contradicting the maximality of $u^*$. Indeed, let $k>G(u^*+)$ belong to the support of $\tilde\nu\lvert_{(G(u^*+),\infty)}$. Then we can pick a small enough $\hat u>u^*$ such that $G(u^*)<G(\hat u)<k$. Let $\bar h:=\tilde\nu((-\infty,G(u^*+)])\wedge\tilde\nu((G(\hat u),\infty))\wedge(\hat u-u^*)>0$. Then using Lemma \ref{lem:disjointSupport} with $\eta=(\mu_{u^*+\bar h}-\mu_{u^*})$ and $\chi=\tilde{\nu}\lvert_{(-\infty,G(u^*+)]\cup(G(\hat u),\infty)}$ we have that $(\mu_{u^*+\bar h}-\mu_{u^*})\leq_{pc}\tilde{\nu}\lvert_{(-\infty,G(u^*+)]\cup(G(\hat u),\infty)}\leq\tilde\nu$ and therefore $(\mu_{u^*+\bar h}-\mu_{u^*})\leq_{pc}\tilde\nu$. But then $(\mu_{u^*+\bar h}-\mu_{u^*})\leq_c S^{\tilde\nu}(\mu_{u^*+\bar h}-\mu_{u^*})$. It follows that $\overline{\mu_{u^*+\bar h}-\mu_{u^*}}=\overline{S^{\tilde\nu}(\mu_{u^*+\bar h}-\mu_{u^*})}$ and therefore $c(u^*+\bar h)=c(u^*)=0$, contradicting the maximality of $u^*$.
		
			Finally suppose that $\tilde\nu(\R)=\tilde\nu((G(u^*+),\infty))$. Pick a small enough $\underline{u}\in(u^*,1)$ such that $G(u^*+)<G(\underline u)<(\overline{\tilde\nu}/\tilde{\nu}(\R))$. Define $g:\R\mapsto\R_+$ by
			$$
			g(k)=\begin{cases}C_{\mu_{\underline u}-\mu_{u^*}}(k),&k\in(-\infty,G(u^*+))\cup(G(\underline{u}),\infty),\\ \frac{C_{\mu_{\underline u}-\mu_{u^*}}(G(u^*+))}{G(\underline u)-G(u^*+)}(G(\underline u)-k),&k\in[G(u^*+),G(\underline u)].
			\end{cases}
			$$
			Then $g(\cdot)$ is convex. Furthermore, since the support of $(\mu_{\underline u}-\mu_{u^*})$ is contained in $(G(u^*+),G(\underline{u})]$ and $C_{\mu_{\underline u}-\mu_{u^*}}$ is convex, we have that
			$$
			 C_{\mu_{\underline u}-\mu_{u^*}}(k)\leq g(k)\leq (\overline{\tilde\nu}-\tilde{\nu}(\R)k)^+\leq C_{\tilde\nu}(k),\quad k\in\R.
			 $$
			 By Lemma \ref{prop:pcd=pc}, $(\mu_{\underline u}-\mu_{u^*})\leq_{pc}\tilde\nu$. But then $(\mu_{\underline u}-\mu_{u^*})\leq_{c}S^{\tilde\nu}(\mu_{\underline u}-\mu_{u^*})$.  It follows that $\overline{\mu_{\underline u}-\mu_{u^*}}=\overline{S^{\tilde\nu}(\mu_{\underline u}-\mu_{u^*})}$ and therefore $c(\underline u)=c(u^*)=0$, contradicting the maximality of $u^*$.	
	
		To finish the proof we must cover the case when $u^*=0$. However this can be achieved using the same arguments as in the above three cases. To see this first note that $G(u^*+)=G(0+)=\ell_\mu$. If $\mu(\{\ell_{\mu}\})>0$ then $\ell_{\mu}>-\infty$, and we can use the arguments of \textit{Case 2} to conclude that $(\nu-S^\nu(\mu_{u^*}))(\R)=\nu(\R)=\nu((-\infty,\ell_\mu))$. On the other hand, if $\mu(\{\ell_{\mu}\})=\mu(\{G(0+)\})=0$, then the arguments of \textit{Case 3} show that $\nu(\R)=\nu((-\infty,\ell_\mu])$, as required. (Note that if $\ell_\mu=G(0+)=-\infty$, then $\nu(\R)=\nu(\{-\infty\})$, a contradiction to the fact that $\nu$ is integrable. Hence, if $u^*=0$ then we must have that $\ell_\mu>-\infty$.)
		\end{proof}
\begin{cor}\label{cor:cincresing}
$c(\cdot)$ is strictly increasing on $(u^*,1)$.
\end{cor}
\begin{proof}
	By Proposition \ref{prop:support}, the support of $\nu-S^{\nu}(\mu_{u^*})$ is (strictly) to the left of the support of $\mu-\mu_{u^*}$. Since, for any $u,v\in(0,1)$ with $u^*< u<v$, $(\mu_v-\mu_u)\leq (\mu-\mu_{u^*})$ and $(S^{\nu-S^\nu(\mu_{u^*})}(\mu_v-\mu_u))\leq(\nu-S^\nu(\mu_{u^*}))$, we have that $\overline{\mu_v-\mu_u}>\overline{S^{\nu-S^\nu(\mu_{u^*})}(\mu_v-\mu_u)}$ and therefore $c(v)>c(u)$, as claimed.
\end{proof}
\subsection{Proof of Theorem \ref{thm:mainTHM_intro}}\label{sec:Constr}
 We will prove Theorem \ref{thm:mainTHM_intro} by explicitly constructing the functions that support the lifted increasing supermartingale coupling. 
 
 Fix $\mu,\nu\in\sM$ with $\mu\leq_{cd}\nu$. In this section we work with a left-continuous version of the quantile function $G_\mu=G^-_\mu$ of the initial measure $\mu$. Define $D_{\mu,\nu}:\R\mapsto\R$ by $D_{\mu,\nu}(k)=P_\nu(k)-P_\mu(k)$, $k\in\R$. Then $D_{\mu,\nu}\geq0$ everywhere, $\lim_{k\to-\infty}D_{\mu,\nu}(k)=0$ and $\lim_{k\to\infty}D_{\mu,\nu}(k)=(\overline\mu-\overline\nu)\geq0$. 

Let $m_{\mu,\nu}:=\mu(\R)=\nu(\R)$. For each $u \in (0,m_{\mu,\nu})$ define $\sE^{\mu,\nu}_u:\R\mapsto\R_+$ by $\sE^{\mu,\nu}_u=P_\nu-P_{\mu_u}$, so that 
\[
\sE^{\mu,\nu}_u(k)=P_\nu(k)-P_{\mu_u}(k)=D_{\mu,\nu}(k)+P_\mu(k)-P_{\mu_u}(k),\quad k \in \R.
\]
Then, by Theorem \ref{thm:shadow_potential}, we have that
\begin{equation*}
P_{S^\nu(\mu_u)}(k)=P_\nu(k)-(\sE_u^{\mu,\nu})^c(k),\quad k\in\R.
\end{equation*}
The underlying idea is that using the graph of $k\mapsto(\sE_u^{\mu,\nu})^c(k)$, for each $u\in(0,m_{\mu,\nu})$, we can define candidate functions that characterise the increasing supermartingale coupling.

Note that, $\sE^{\mu,\nu}_u(k)=D_{\mu,\nu}(k)$ for $k\leq G_\mu(u)$. Since $P_\mu-P_{\mu_u}$ is non-negative on $\R$, we have that $\sE^{\mu,\nu}_u(k)\geq D_{\mu,\nu}(k)$ for $k> G_\mu(u)$. Moreover, since $P_{\mu_u}$ is linear on $[G_\mu(u),+\infty)$, $\sE^{\mu,\nu}_u$ is convex on $(G_\mu(u),+\infty)$. It is also easy to see that $k\mapsto\sE^{\mu,\nu}_u(k)-D_{\mu,\nu}(k)$ is non-decreasing.

Let $Q_{\mu,\nu},S_{\mu,\nu}:(0,m_{\mu,\nu})\mapsto[-\infty,+\infty]$ be given by
\begin{align*}
 Q_{\mu,\nu}(u) &:= X^{\sE^{\mu,\nu}_u}(G_\mu(u)), \\
 S_{\mu,\nu}(u) &:= Z^{\sE^{\mu,\nu}_u}(G_\mu(u)).
\end{align*}

By definition, $Q_{\mu,\nu}(u)\leq G_{\mu}(u) \leq S_{\mu,\nu}(u)$, for $u\in(0,m_{\mu,\nu})$. Furthermore, for each $u\in(0,m_{\mu,\nu})$, either $Q_{\mu,\nu}(u)<G_{\mu}(u)< S_{\mu,\nu}(u)$ or $Q_{\mu,\nu}(u)= G_{\mu}(u)=S_{\mu,\nu}(u)$, see Hobson and Norgilas \cite[Lemma 4.1]{HobsonNorgilas:21}.

We now introduce a function $\phi_{\mu,\nu}:(0,m_{\mu,\nu})\mapsto\R$ which represents the slope of $(\sE_u^{\mu,\nu})^c(\cdot)$ at $G_\mu(u)$. If $Q_{\mu,\nu}(u)< G_{\mu}(u) < S_{\mu,\nu}(u)$, then this slope is well defined.
If $Q_{\mu,\nu}(u)=G_{\mu}(u)=S_{\mu,\nu}(u)$ then the slope of $(\sE^{\mu,\nu}_u)^c$ may not be well defined at $G_\mu(u)$. To cover all cases we define:
\begin{defn}
	\label{def:slope}
	$\phi_{\mu,\nu}:(0,m_{\mu,\nu}) \mapsto \R$ is given by $\phi_{\mu,\nu}(u) = \inf \{ \psi: \psi \in \partial (\sE_u^{\mu,\nu})^c (G_\mu(u)) \}$.
\end{defn}

Now we can introduce our second candidate lower function.

Recall the definition of $ L^f_{a,b}$ for any $f:\R\mapsto\R$ (see \eqref{eq:L1}), so that (in the case $a<b$) $L^f_{a,b}$ is the line passing through $(a,f(a))$ and $(b, f(b))$. Define also $L^{f,\psi}_a$ by $L^{f,\psi}_a(y) = f(a) + \psi(y-a)$ so that $L^{f,\psi}_a$ is the line passing through $(a,f(a))$ with slope $\psi$. (Note that, in the case $a=b$, $L^f_{a,a}=L^{f,0}_a$.)
Define $R_{\mu,\nu}:(0,m_{\mu,\nu})\mapsto[-\infty,\infty]$ by
\begin{equation}    \label{eq:R}
R_{\mu,\nu}(u):=  \inf\{k:k \leq G_\mu(u), D_{\mu,\nu}(k) = {L}^{(\sE^{\mu,\nu}_u)^c ,\phi_{\mu,\nu}(u)}_{G_\mu(u)}(k) \}, \quad u\in(0.m_{\mu,\nu}).
\end{equation}
If $Q_{\mu,\nu}(u)<G_\mu(u)$ then the definition of $R_{\mu,\nu}$ can be rewritten as $R_{\mu,\nu}(u)=   \inf\{k:k \leq G_\mu(u), D_{\mu,\nu}(k) = L^{\sE^{\mu,\nu}_u}_{Q_{\mu,\nu}(u),S_{\mu,\nu}(u)}(k) \}$.
Note that 
$Q_{\mu,\nu}(u) \in \{k:k \leq G_\mu(u), D_{\mu,\nu}(k) = {L}^{(\sE^{\mu,\nu}_u)^c,\phi_{\mu,\nu}(u)}_{G_\mu(u)}(k) \}$ so that $R_{\mu,\nu}(u)$ exists in all cases and satisfies $R_{\mu,\nu}(u) \leq Q_{\mu,\nu}(u)$. (See, for example, Figure~\ref{fig:RS} which corresponds to the martingale case when $\overline\mu=\overline\nu$).

If $Q_{\mu,\nu}(u)<S_{\mu,\nu}(u)$, then by construction, $(\sE^{\mu,\nu}_u)^c<\sE^{\mu,\nu}_u$ on $(Q_{\mu,\nu}(u),S_{\mu,\nu}(u))$ and $(\sE^{\mu,\nu}_u)^c \leq \sE^{\mu,\nu}_u$ on $[R_{\mu,\nu}(u),S_{\mu,\nu}(u)]$. In particular, $(\sE^{\mu,\nu}_u)^c$ is linear on $(R_{\mu,\nu}(u),S_{\mu,\nu}(u))$, whilst $(\sE^{\mu,\nu}_u)^c(S_{\mu,\nu}(u))=\sE^{\mu,\nu}_u(S_{\mu,\nu}(u))$, $(\sE^{\mu,\nu}_u)^c(Q_{\mu,\nu}(u))=\sE^{\mu,\nu}_u(Q_{\mu,\nu}(u))=D_{\mu,\nu}(Q_{\mu,\nu}(u))$ and $(\sE^{\mu,\nu}_u)^c(R_{\mu,\nu}(u))=\sE^{\mu,\nu}_u(R_{\mu,\nu}(u))=D_{\mu,\nu}(R_{\mu,\nu}(u))$ (provided that $S_{\mu,\nu},Q_{\mu,\nu}$ and $R_{\mu,\nu}$ are finite, respectively). Then we have that
\begin{equation*}
\phi_{\mu,\nu}(u) = \frac{\sE^{\mu,\nu}_u(S_{\mu,\nu}(u))-D_{\mu,\nu}(Q_{\mu,\nu}(u))}{S_{\mu,\nu}(u)-Q_{\mu,\nu}(u)} = \frac{\sE^{\mu,\nu}_u(S_{\mu,\nu}(u))-D_{\mu,\nu}(R_{\mu,\nu}(u))}{S_{\mu,\nu}(u)-R_{\mu,\nu}(u)}.
\end{equation*}
Further, $\phi_{\mu,\nu}(u)$ is an element of each of $\partial \sE^{\mu,\nu}_u(R_{\mu,\nu}(u))$, $\partial \sE^{\mu,\nu}_u(Q_{\mu,\nu}(u))$ and $\partial\sE^{\mu,\nu}_u(S_{\mu,\nu}(u))$ together with $\partial (\sE^{\mu,\nu}_u)^c(R_{\mu,\nu}(u))$, $\partial (\sE^{\mu,\nu}_u)^c(Q_{\mu,\nu}(u))$ and $\partial(\sE^{\mu,\nu}_u)^c(S_{\mu,\nu}(u))$.
\subsubsection{Martingale case: $u^*=1$}\label{sec:u=1}
Suppose $\mu,\nu\in\sP$ with $\mu\leq_{cd}\nu$ and $u^*=1$. Then $\overline\mu=\overline\nu$ and therefore $\mu\leq_c\nu$, so that we are in the martingale set-up. Hobson and Norgilas \cite{HobsonNorgilas:21} showed how to construct the upper and lower functions that support the generalised left-curtain martingale coupling, on a single irreducible component only. In this section we will extend their result by \textit{gluing} together the constructions obtained on separate irreducible components.
 
Let $(\sI_i)_{i\geq 0}=(\sI^{\mu,\nu}_i)_{i\geq 0}:=((\ell_i,r_i))_{i\geq0}$ denote the collection of irreducible components associated to $\mu$ and $\nu$, and let $\overline{\sI_i}=[\ell_i,r_i]$ denote the closure of $\sI_i$. Define
$$\mu_i:=\mu\lvert_{\sI_i},\quad\nu_i:=S^\nu(\mu_i),\quad i\geq 0.
$$
Since each $\sI_i$ is irreducible, $\mu_i$ is embedded in $\nu\lvert_{\overline{\sI_i}}$ under any $\pi\in\Pi_M(\mu,\nu)$. In particular,
$$\nu_i={\nu\lvert_{\sI_i}}+\alpha_i\delta_{\ell_i}+\beta_i\delta_{r_i},\quad i\geq0,$$
where $\alpha_i\in[0,\nu(\{\ell_i\})]$ and $\beta_i\in[0,\nu(\{r_i\})]$, recall Lemma \ref{lem:irreducibleComp}.

Let $G=G_\mu$ and, for $i\geq0$, define
$$
u^\ell_i:=\inf\{u\in(0,m_{\mu,\nu}):G(u)>\ell_i\},\quad
u^r_i:=\sup\{u\in(0,m_{\mu,\nu}):G(u)<r_i\},
$$
and let $(\sU_i)_{i\geq1}$ be such that $\sU_i:=(u^\ell_i,u^r_i)$. Since each $\sI_i$ is irreducible, it is easy to see that there exists $u\in(0,m_{\mu,\nu})$ with $G(u)\in(\ell_i,r_i)$, and therefore both $u^\ell_i$ and $u^r_i$ are well-defined (for example, we can take $u\in(0,m_{\mu,\nu})$ with $G(u)\leq x\leq G(u+)$, where $x\in\argsup_{k\in\overline{\sI_i}}D_{\mu,\nu}(k)$).
\begin{lem}
	\label{lem:Uirreducible}
	If $u\in\bigcup_{i\geq0}\sU_i$ then $G(u)\in\bigcup_{i\geq0}\sI_i$.
	\end{lem}
\begin{proof}
 Since $G$ is non-decreasing and left-continuous, we have $G(u^r_i)\leq r_i$. Suppose $G(u^r_i)=r_i$. If $\mu(\{r_i\})=0$ then $G(u)<r_i$ for $u<u^r_i$. If $\mu(\{r_i\})>0$, then $u^r_i=F_\mu(r_i-)=\mu((-\infty,r_i))$ and therefore $G(u)<r_i$ for $u<u^r_i$. On the other hand, we have that either $G(u^\ell_i)=\ell_i=G(u^\ell_i+)$ or $G(u^\ell_i)\leq\ell_i<G(u^\ell_i+)$. If $G(u^\ell_i)=\ell_i=G(u^\ell_i+)$ and $\mu(\{\ell_i\})>0$ then $u^\ell_i=F_\mu(\ell_i)=\mu((-\infty,\ell_i])$. In either case $\ell_i<G(u)$ for $u>u^\ell_i$. We conclude that $G(u)\in\sI_i$ for all $u\in\sU_i$. It follows that $G(u)\in\bigcup_{i\geq1}\sI_i$ provided that $u\in\bigcup_{i\geq0}\sU_i$.
 \end{proof}
\begin{rem}\label{rem:G(u)}
 The reverse statement of Lemma \ref{lem:Uirreducible} is almost true. In fact we have that if $G(u)\in\bigcup_{i\geq0}\sI_i$ then $u\in\Big(\bigcup_{i\geq0}\sU_i\Big)\setminus\sN$ for some Lebesgue null-set $\sN$. 
 
 Suppose $G(u)\in\bigcup_{i\geq0}\sI_i$ for some $u\in(0,1)$. Then $G(u)\in\sI_i$ for some $i\geq0$, and therefore $\ell_i<G(u)<r_i$. It follows that $u^\ell_i\leq u\leq u^r_i$. If $u^\ell_i=u$, then $G(u^\ell_i)>\ell_i$. But his cannot happen since $G$ is left-continuous. It follows that $u^\ell_i<u \leq u^r_i$. Now suppose that $u=u^r_i$, so that $G(u^r_i)<r_i$. This, however is a possible situation. It happens when $G(u^r_i)<r_i\leq G(u^r_i+)$. Since there are countably many $u^r_i$'s the assertion follows.
\end{rem}

We will now define candidate functions that support the left-curtain martingale coupling on each irreducible component.

For $i\geq0$, let $G_i:(0,u^r_i-u^\ell_i)\mapsto\R$ be a (left-continuous) quantile function of $\mu_i$, i.e., $G_i=G_{\mu_i}$. By construction, $G(u)=G_i(u-u^\ell_i)$ for $u\in\sU_i$, and therefore, for a uniform random variable $U$ on $[0,1]$ we have that $\sL(I_{\{U\in\sU_i\}}G(U))=\sL(I_{\{U\in\sU_i\}}G_i(U-u^\ell_i))=\mu_i$.

Define $R_i,S_i:(0,u^r_i-u^\ell_i)\mapsto[-\infty,\infty]$ by
$$
R_i(u)=R_{\mu_i,\nu_i}(u),\quad S_i(u)=S_{\mu_i,\nu_i}(u),\quad u\in(0,u^r_i-u^\ell_i).
$$
The next result shows that $R_i$ and $S_i$ are real-valued. The proof is presented in Section \ref{sec:proofs}.
\begin{lem}\label{lem:RSbound}
	For $i\geq0$, $R_i(u),S_i(u)\in[\ell_i,r_i]$ for all $u\in\sU_i$. If $\ell_i=-\infty$ (resp. $r_i=\infty$) then $R_i(u),S_i(u)\in(\ell_i,r_i]$ (resp. $R_i(u),S_i(u)\in[\ell_i,r_i)$).
	\end{lem}

Let $\sN:=\bigcup_{i\geq0}\{u^r_i:G(u^r_i)<r_i)\}$, recall Remark \ref{rem:G(u)}. Finally, define $\tilde R_{\mu,\nu},\tilde S_{\mu,\nu}:(0,1)\mapsto\R$ by
\begin{align}
\label{eq:Rgeneral}
\tilde R_{\mu,\nu}(u)&=I_{\{u\notin(\bigcup_{i\geq0}\sU_i)\bigcup \sN\}}G(u)+\sum_{i\geq0}\Big(I_{\{u\in\sU_i\}}R_i(u-u^\ell_i)+I_{\{u=u^r_i\in\sN\}}\ell_i\Big),\\ \label{eq:Sgeneral} \tilde S_{\mu,\nu}(u)&=I_{\{u\notin(\bigcup_{i\geq0}\sU_i)\bigcup \sN\}}G(u)+\sum_{i\geq0}I_{\{u\in\sU_i\}}\Big(S_i(u-u^\ell_i)+I_{\{u=u^r_i\in\sN\}}r_i\Big).
\end{align}

\begin{figure}[H]
	\centering
	\begin{tikzpicture}[scale=1]

	\begin{axis}[%
	width=6.028in,
	height=2.754in,
	at={(1.011in,0.642in)},
	scale only axis,
	xmin=-11,
	xmax=11,
	ymin=-0.5,
	ymax=2,
	axis line style={draw=none},
	ticks=none
	]
	\draw[black, thin] (-11,0) -- (11,0);
	\addplot [color=black, dashed, line width=0.5pt, forget plot]
	table[row sep=crcr]{%
		5.10778631336004
		0.579120487313966\\
		5.28116588395275	0.550692356321193\\
		5.45454545454545	0.516281441424168\\
		5.79769279116164	0.441487336695587\\
		6.14084012777783	0.37232841678\\
		6.81818181818182	0.25309834481619\\
		7.50052689249177	0.156184065001039\\
		8.18181818181818	0.0826473956587641\\
		8.913369432737	0.0295195446611451\\
		9.54545454545455	0.00516182995117376\\
		10.1951138678739	-4.02019102274664e-06\\
		10.9090909090909	5.23608497005057e-06\\
		11.6245592768795	4.99949232768415e-06\\
		12.2727272727273	5.69424161334098e-06\\
		12.7873709632194	-0.000900998195628944\\
		13.6363636363636	2.55228360046544e-06\\
		14.3055157129462	-7.12227511812102e-06\\
		15	5.82376882718449e-06\\
	};
	\addplot [color=red, line width=1.0pt, forget plot]
	table[row sep=crcr]{%
		-15	0\\
		-14.3743426864006	0\\
		-13.6363636363636	0\\
		-12.922520765984	0\\
		-12.2727272727273	0\\
		-11.6037511150776	0\\
		-10.9090909090909	0\\
		-10.1567767603786	0\\
		-9.54545454545455	0.00516529194875333\\
		-8.78555408111793	0.0241111129606366\\
		-8.18181818181818	0.0392044808885227\\
		-7.52404914247162	0.0556486746725015\\
		-6.81818181818182	0.0732953232255121\\
		-6.16274148089556	0.0896812995719799\\
		-5.45454545454546	0.107386165562501\\
		-4.69396922041941	0.12640053418326\\
		-4.09090909090909	0.141477007899491\\
		-3.35965597297158	0.159758300050986\\
		-2.72727272727273	0.17556785023648\\
		-2.00961780284731	0.193509188215844\\
		-1.36363636363636	0.20965869257347\\
		-0.688142581770708	0.226546004052749\\
		0	0.243749534910459\\
		0.681996956288809	0.260799348155818\\
		1.36363636363636	0.277840010546298\\
		2.07255602717534	0.295562666422883\\
		2.72727272727273	0.31193027388147\\
		3.46373839324672	0.330341566774406\\
		4.09090909090909	0.346020537216643\\
		4.76102717217463	0.362773171911045\\
		5.45454545454545	0.380110800551816\\
		6.14084012777783	0.397267842384896\\
		6.81818181818182	0.414201063886988\\
		7.50052689249177	0.431259367617357\\
		8.18181818181818	0.449124682527086\\
		8.913369432737	0.487440956412348\\
		9.54545454545455	0.542093630196611\\
		10.1951138678739	0.618135397881325\\
		10.9090909090909	0.707391398215165\\
		11.6245592768795	0.796824991527279\\
		12.2727272727273	0.877846509150189\\
		12.7873709632194	0.942167254247437\\
		13.6363636363636	1.04829872537446\\
		14.3055157129462	1.13389926036962\\
		15	1.21874929168467\\
	};
	\addplot [color=blue, dotted, line width=1.5pt, forget plot]
	table[row sep=crcr]{%
		-15	0\\
		-14.3743426864006	0\\
		-13.6363636363636	0\\
		-12.922520765984	0\\
		-12.2727272727273	0\\
		-11.6037511150776	0\\
		-10.9090909090909	0\\
		-10.1567767603786	0\\
		-9.54545454545455	0.00516529194875333\\
		-8.78555408111793	0.0368719582992022\\
		-8.18181818181818	0.0826446406924073\\
		-7.52404914247162	0.153258338675332\\
		-6.81818181818182	0.253099065950084\\
		-6.16274148089556	0.368114076983243\\
		-5.80864346772051	0.43918676050732\\
		-5.45454545454546	0.516270463100197\\
		-5.26440139601394	0.553708788226936\\
		-5.07425733748243	0.58391669559351\\
		-4.88411327895092	0.606892956316372\\
		-4.69396922041941	0.622638584241668\\
		-4.54320418804183	0.626617195164236\\
		-4.39243915566425	0.632781558477449\\
		-4.24167412328667	0.631034673976273\\
		-4.09090909090909	0.624741555295822\\
		-3.90809581142471	0.611011938733503\\
		-3.72528253194034	0.590598341633919\\
		-3.54246925245596	0.563502600307279\\
		-3.35965597297158	0.532183251804731\\
		-2.72727272727273	0.435950274547483\\
		-2.00961780284731	0.350962188489822\\
		-1.36363636363636	0.296342723908122\\
		-0.688142581770708	0.261838896997191\\
		0	0.249996711366964\\
		0.681996956288809	0.261628059973472\\
		1.36363636363636	0.296487735376558\\
		2.07255602717534	0.357385147678734\\
		2.72727272727273	0.435946494150791\\
		3.46373839324672	0.549937949132153\\
		3.62053106766232	0.575890831516237\\
		3.77732374207791	0.597092931388165\\
		3.9341164164935	0.61337407422424\\
		4.09090909090909	0.624743965908923\\
		4.25843861122547	0.63190866150059\\
		4.42596813154186	0.632554779333081\\
		4.59349765185824	0.628031280250534\\
		4.76102717217463	0.617918105339684\\
		4.93440674276733	0.601515203917605\\
		5.10778631336004	0.580572723208292\\
		5.45454545454545	0.542107891512015\\
		6.14084012777783	0.483683429598855\\
		6.81818181818182	0.449121067486052\\
		7.50052689249177	0.437499922994881\\
		8.18181818181818	0.449124682527086\\
		8.913369432737	0.487440956412348\\
		9.54545454545455	0.542093630196611\\
		10.1951138678739	0.618135397881325\\
		10.9090909090909	0.707391398215165\\
		11.6245592768795	0.796824991527279\\
		12.2727272727273	0.877846509150189\\
		12.7873709632194	0.942167254247437\\
		13.6363636363636	1.04829872537446\\
		14.3055157129462	1.13389926036962\\
		15	1.21874929168467\\
	};
	
	\node (R)[scale=1] at (-9.5,-0.2) {$R(u)$};
	\draw[gray,dashed,thin] (-9.5,-0.1) -- (-9.5,0.05);
	
	\node (Q)[scale=1] at (0.5,-0.2) {$Q(u)$};
	\draw[gray,dashed,thin] (0.5,-0.1) -- (0.5,0.27);
	
	\node (G)[scale=1] at (5,-0.2) {$G(u)$};
	\draw[gray,dashed,thin] (5,-0.1) -- (5,0.64);
	
	\node (S)[scale=1] at (8,-0.2) {$S(u)$};
	\draw[gray,dashed,thin] (8,-0.1) -- (8,0.45);

	\node (E)[scale=1,blue] at (7,0.7) {$y\mapsto\sE_{u}(y)$};
	\node (D)[scale=1,black] at (9.5,0.2) {$y\mapsto D(y)$};
	\node (P)[scale=1,red] at (0.5,0.9) {slope $\phi(u)$};
	\draw [gray,->] (0.5,0.8) to[out=230, in=70] (-4.5,0.25);
	\draw [gray,->] (0.55,0.8) to[out=310, in=110] (4,0.45);
	\end{axis}
	
	\end{tikzpicture}
	\caption{Plot of locations of $R(u)=R_{\mu,\nu}(u)$, $Q(u)=Q_{\mu,\nu}(u)$, $G(u)=G_\mu(u)$ and $S(u)=S_{\mu,\nu}(u)$ in the case where $R(u)<Q(u)<G(u)<S(u)$ and such that $\{k:D_{\mu,\nu}(u)(k)>0\}$ is an interval. The dashed curve represents $D=D_{\mu,\nu}(u)$. The dotted curve corresponds to the graph of $\sE_{u}=\sE_u^{\mu,\nu}$. Note that $D=\sE_u$ on $(-\infty, G(u)]$, while $\sE_u$ is convex and $D\leq\sE_u$ on $(G(u),\infty)$. The solid curve below $\sE_{u}$ represents $\sE^c_{u}$. The convex hull $\sE^c_{u}$ is linear on $[R(u),S(u)]$, and its slope is given by $\phi(u)=\phi_{\mu,\nu}(u)$.}
	\label{fig:RS}
\end{figure}
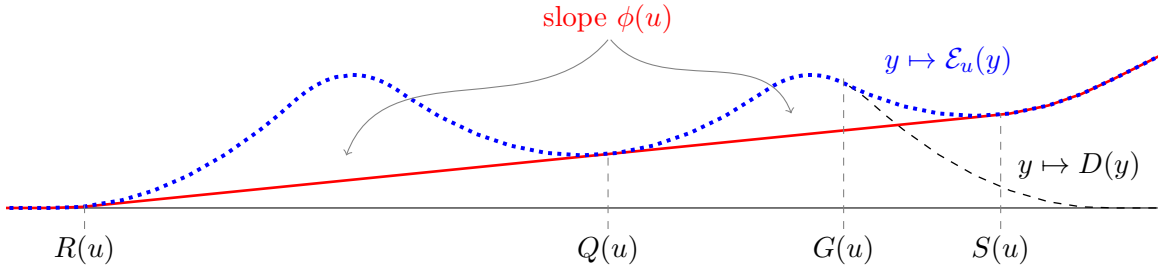
\begin{rem}
	\label{rem:intuitionRS}
	The intuition behind our choice of $R$ and $S$ is because (at least in regular cases) they do satisfy the mean and mass preservation conditions
	\begin{equation}\label{eq:meanMass}
	\int^{G(u)}_{R(u)}x^i\mu(dx)=\int^{S(u)}_{R(u)}x^i\nu(dx),\quad i=0,1,
	\end{equation}
	which are also satisfied by the pair of upper and lower functions $(T_d,T_u)$ constructed by Henry-Labord\`{e}re and Touzi~\cite{HenryLabordereTouzi:16} when $\mu$ is continuous.
	
	Indeed, suppose $\mu,\nu$ are atomless with positive density everywhere. In this case $D_{\mu,\nu}$ and $\sE^{\mu,\nu}_u$, for $u\in(0,1)$, are differentiable. Then if $R(u)<G(u)\leq S(u)$, by construction (see Figure \ref{fig:RS}) we have that
	$$
	D'(R(u))=\sE_u'(S(u))\quad\textrm{and}\quad D(R(u))+D'(R(u))(S(u)-R(u))=\sE_u(S(u)),
	$$
	which, using the definition of $D$ and $\sE_u$, can be easily shown to be equivalent to \eqref{eq:meanMass}.
\end{rem}

We are now ready to present the main result of this section. Since $u^*=1$, the function $T$ in Theorem \ref{thm:mainTHM_intro} is irrelevant. In particular, Theorem \ref{thm:mainTHM_intro} is now a direct consequence of the following result.
\begin{thm}\label{thm:constructionMART}
	Suppose $\mu\leq_c\nu$. Define $\tilde R$ and $\tilde S$ as in \eqref{eq:Rgeneral} and \eqref{eq:Sgeneral}, respectively.
	Then they are second-order left-monotone with respect to $G$ on $(0,1)$ and define a construction as in Theorem~\ref{thm:mainTHM_intro} such that $\sL(Y)=\nu$. In particular, $\tilde R$ and $\tilde S$ define the (lifted) left-curtain martingale coupling.
\end{thm}
\begin{proof}
We have $\sL(X(U,V))=\sL(X(U))=\sL(G(U))=\mu$. Furthermore, since $\sL(I_{\{U\in\sU_i\}}G(U))=\mu_i$ and $\Prob[\{U\in\sN\}]=0$, we still have that $\sL(I_{\{U\in\sU_i\cup\sN\}}G(U))=\mu_i$. Since $\sU_i$'s are disjoint, we conclude that $\sL(I_{\{U\in(\bigcup_{i\geq0}\sU_i)\cup\sN\}}G(U))=\sum_{i\geq0}\mu_i$, and therefore $\sL(I_{\{U\notin(\bigcup_{i\geq0}\sU_i)\cup\sN\}}G(U))=\mu_{-1}$.

We now turn to $Y(U,V)$. For $y\in\R$ we have
\begin{align*}
\Prob[I_{\{U\notin\bigcup_{i\geq0}\sU_i\}}&Y(U,V)\leq y]\\&=\Prob[I_{\{U\notin(\bigcup_{i\geq0}\sU_i)\cup\sN\}}G(U)\leq y]+\Prob[I_{\{U\in\sN\}}Y(U,V)\leq y]=\mu_{-1}((-\infty,y]),
\end{align*}
where we used that $\Prob[U\in\sN]=0$ and therefore $$\Prob[I_{U\in\sN}Y(U,V)\leq y]\leq\int_{\sN}\sum_{i\geq 0}I_{\{u=u^r_i:\ell_i\leq y\}}du\leq\int_{\sN}du=0.$$

We will now show that $\sL(I_{\{U\in\bigcup_{i\geq0}\sU_i\}}Y(U,V))=\sum_{i\geq0}\nu_i$. Note that it is enough to show that $\sL(I_{\{U\in\sU_i\}}Y(U,V))=\nu_i$, $i\geq0$. But this follows from the results of Hobson and Norgilas \cite{HobsonNorgilas:21}. Indeed, if for a fixed $i\geq0$ we define
$$
Y_i(u,v) =I_{\{R_i(u)=S_i(u)\}}G_i(u)+I_{\{R_i(u)<S_i(u)\}}\Big\{R_i(u) I_{\{ v \leq \frac{S_i(u) - G_i(u)}{S_i(u)-R_i(u)} \}} +  S_i(u) I_{ \{ v > \frac{S_i(u) - G_i(u)}{S_i(u)-R_i(u)} \} }\Big\},
$$
for $u\in(0,u^r_i-u^\ell_i)$ and $v\in(0,1)$, then $\sL(I_{\{U\in\sU_i\}}Y(U,V))=\sL(Y_i(U_i,V))$ where $U_i$ is a uniform random variable on $[0,u^r_i-u^\ell_i]$ that is independent of $V$. By Hobson and Norgilas \cite[Theorem 7.8]{HobsonNorgilas:21} we have that $\sL(Y_i(U_i,V))=\nu_i$, and therefore $\sL(I_{\{U\in\sU_i\}}Y(U,V))=\nu_i$ as required.

The martingale property of $Y$ follows by construction. Hence we are left to show that $\tilde{R}=\tilde{R}_{\mu,\nu}$ and $\tilde{S}=\tilde{S}_{\mu,\nu}$ are second-order left-monotone with respect to $G$ on $(0,1)$. Fix $u,v\in(0,1)$ with $u<v$. If $u\notin(\bigcup_{i\geq0}\sU_i)\cup\sN$, then $\tilde{R}(u)=G(u)=\tilde{S}(u)$ and the second-order left-monotonicity property trivially holds. Hence we can assume that either $u\in\bigcup_{i\geq0}\sU_i$ or $u\in\sN$.

\textit{Case 1: $u\in\bigcup_{i\geq0}\sU_i$}. Suppose $u\in\sU_i$ for some $i\geq0$. We know that $u\mapsto R_i(u-u^\ell_i)$ and $u\mapsto S_i(u-u^\ell_i)$ are second-order left-monotone with respect to $G(\cdot)=G_i(\cdot-u^\ell_i)$ on $\sU_i$ (see Hobson and Norgilas \cite[Theorem 7.8]{HobsonNorgilas:21}). Hence if $v\in\sU_i$ we are done. On the other hand if $v\in\sU_j$ with $i\neq j$, then using Lemma \ref{lem:RSbound} we have that either $\ell_i\leq \tilde{R}(u)\leq\tilde{S}(u)\leq r_i\leq\ell_j\leq \tilde{R}(v)$ or  $\tilde{R}(v)\leq r_j\leq\ell_i\leq \tilde{R}(u)\leq\tilde{S}(u)\leq r_i$. In either case we have that $\tilde{R}(v)\notin(\tilde{R}(u),\tilde{S}(u))$ as required.

Now suppose that $v\in\sN$, so that $v=u^r_j$ for some $j\geq0$. If $i=j$, then (using Lemma \ref{lem:RSbound} again) we have that $\tilde{R}(v)=\ell_i\leq\tilde{R}(u)\leq\tilde{S}(u)$ as required.

Suppose $v\notin\bigcup_{i\geq0}\sU_i$ and $v\notin\sN$. Then $\tilde{R}(v)=G(v)=\tilde{S}(v)$. Since $v\notin\sU_i$ and $u<v$, $v\geq u^r_i$. Recall that $G(u^r_i)\leq r_i$. If $G(u^r_i)= r_i$ then (using Lemma \ref{lem:RSbound}) $\tilde{R}(u)\leq\tilde{S}(u)\leq r_i\leq G(v)$. On the other hand if $G(u^r_i)< r_i$, then $u^r_i<v$. By the definition of $u^r_i$ we have that $G(u^r_i)< r_i\leq G(u^r_i+)$, and then it follows that $\tilde{R}(u)\leq\tilde{S}(u)\leq r_i\leq G(u^r_i+)\leq G(v)$. In both cases the second-order left-monotonicity property holds.

\textit{Case 2: $u\in\sN$}. In this case $u=u^r_i$ for some $i\geq0$ and $\tilde{R}(u)=\ell_i<r_i=\tilde{S}(u)$. If $v\in\sN$ then $u^r_i=u<v=u^r_j$ for some $i\neq j$. Then $r_i<r_j$ and we have that $\tilde{R}(u)=\ell_i<r_i=\tilde{S}(u)<\ell_j=\tilde{R}(v)$. On the other hand, if $v\in\sU_j$ for some $j\geq0$, then $i\neq j$ and $\ell_j\leq\tilde{R}(v)\leq r_j$. Then either $r_i\leq\ell_j$ or $r_j\leq\ell_i$. In either case $\tilde{R}(v)\notin (\tilde{R}(u),\tilde{S}(u))$. Finally, suppose $v\notin \Big(\bigcup_{i\geq0}\sU_i\Big)$ and $v\notin\sN$, so that $\tilde{R}(v)=G(v)=\tilde{S}(v)$. Then $u^r_i=u<v$. By definition of $u^r_i$ and since $u^r_i\in\sN$ we have that $G(u)<r_i\leq G(u+)\leq G(\tilde{u})$ for all $\tilde{u}>u$. It follows that $\tilde{R}(u)=\ell_i<r_i=\tilde{S}(u)\leq G(u+)\leq G(v)=\tilde{R}(v)$ as required.
\end{proof}
We finish this section with a remark that will be useful in Section \ref{sec:u^*General}.
\begin{rem}\label{rem:R<x<S}
Suppose $x\in\R$ with $\tilde{R}_{\mu,\nu}(u)<x<\tilde{S}_{\mu,\nu}(u)$ for some $u\in(0,1)$. Then $u\in\sU_i\cup(\{u^r_i\}\cap\sN)$ for some $i\geq0$. Combining Lemma \ref{lem:RSbound} and the definitions of $\tilde{R}_{\mu,\nu}$ and $\tilde{S}_{\mu,\nu}$ on $\sN$, we have that $\ell_i\leq\tilde{R}_{\mu,\nu}(u)<x<\tilde{S}_{\mu,\nu}(u)\leq r_i$ and therefore $x\in\sI_i^{\mu,\nu}$.
\end{rem}

\subsubsection{The (strict) supermartingale case: $u^*=0$}\label{sec:u=0}
In this section we consider $\mu,\nu\in\sP$ with $r_\nu\leq\ell_{\mu}$ and $\mu(\{\ell_\mu\})\wedge\nu(\{\ell_\mu\})=0$. Then the support of $\nu$ is strictly to the left of the support of $\mu$ and we automatically have that $\mu\leq_{cd}\nu$ and $u^*=0$. Note that the converse is also true (recall Proposition \ref{prop:support}): if $\mu\leq_{cd}\nu$ and $u^*=0$ the the support of $\nu$ is strictly to the left of the support of $\mu$.

Recall the definition of $D(k)=D_{\mu,\nu}(k)=P_\nu(k)-P_\mu(k)$, $k\in\R$. Note that $D(k)=P_\nu(k)$ for $k\leq\ell_\mu$ and $D(k)=(k-\overline\nu)-P_\mu(k)$ for $k>\ell_\mu$. It follows that $D$ is continuous, non-decreasing, convex on $(-\infty,\ell_\mu)$ and concave on $(\ell_\mu,\infty)$. Furthermore $\lim_{k\to-\infty}D(k)=0$ but $\lim_{k\to\infty}D(k)=\overline\mu-\overline\nu>0$.

Now, as in Section \ref{sec:u=1}, for each $u\in(0,1)$ we will use the function $\sE_u(k)=\sE_u^{\mu,\nu}(k)=P_\nu(k)-P_{\mu_u}(k)$, $k\in\R$, and its convex hull $\sE^c_u$. First note that $\sE_u(k)=D(k)+P_\mu(k)-P_{\mu_u}(k)$ for all $k\in\R$, and therefore $\sE_u\geq D$ everywhere. Similarly as $D$, $\sE_u$ is continuous, non-decreasing, convex on $(-\infty,\ell_\mu)$ and concave on $(\ell_\mu,\infty)$. In fact, since $\mu_u$ coincides with $\mu$ on $(-\infty,G(u))$ and does not charge $(G(u),\infty)$, we have that $\sE_u=D$ on $(-\infty,G(u+))$ and $\sE_u=(k-\overline\nu)-(uk-\overline{\mu_u})$ (so that it is linear with slope $(1-u)$) on $(G(u+),\infty)$.

We now introduce our candidate supporting function for the increasing supermartingale coupling. Define $T:(0,1)\mapsto\R$ by
\begin{equation}\label{eq:T}
T(u):=T_{\mu,\nu}:=R_{\mu,\nu}(u)=:R(u),\quad u\in(0,1),
\end{equation}
where $R_{\mu,\nu}$ is given by \eqref{eq:R}. (Note that we chose the lower function from the martingale set-up to be the only supporting function in the strict supermartingale case.) See Figure \ref{fig:T}.

Let $F_\nu$ and $G_\nu$ denote the (right-continuous) cumulative distribution function of $\nu$ and (left-continuous) quantile function of $\nu$, respectively. Let $U$ be a uniform random variable on $[0,1]$.
\begin{prop}\label{prop:U=0Construct}
	$T$, defined in \eqref{eq:T}, is non-increasing, right-continuous, $T<G$ on $(0,1)$ and $\sL(T(U))=\nu$.
\end{prop}
\begin{proof}
	Fix $u\in(0,1)$. By definition of $G_\nu$ we have that $(1-u)\in\partial P_\nu(G_\nu(1-u))$. Moreover, since $G_\nu$ is non-decreasing and left-continuous, we immediately have that $u\mapsto G_\nu(1-u)$ is non-increasing and right-continuous. Furthermore, $G_\nu(1-u)\leq r_\nu\leq\ell_\mu\leq G(u)$. Since $\mu$ (and thus also $\mu_u$, for each $u\in(0,1)$) and $\nu$ cannot have atom at $\ell_\mu$ simultaneously, it follows that $G_\nu(1-u)<G(u)$ for all $u\in(0,1)$. Finally, since $\tilde U=1-U$ is still a uniform random variable on $[0,1]$, $\sL(G_\nu(1-U))=\nu$. We are left to show that $T(u)=G_\nu(1-u)$, $u\in(0,1)$.
	
	Let $h_u:\R\mapsto\R$ be defined by
	$$
h_u(k)=\begin{cases}
P_\nu(k),&k\leq G_\nu(1-u),\\
L^{P_\nu,1-u}_{G_\nu(1-u)}(k),&k>G_\nu(1-u).
\end{cases}
$$
Note that $h_u$ is continuous and convex. Furthermore, $h_u\leq \sE_u$ everywhere and $h_u<\sE_u$ on $(G_\nu((1-u)+),\infty)$. To see this first observe that, $h_u=P_\nu=\sE_u$ on $(-\infty,G_\nu(1-u)]$. Also, if $G_\nu(1-u)<G_\nu((1-u)+)$, then $h_u=L^{P_\nu,1-u}_{G_\nu(1-u)}=P_\nu=\sE_u$ on $(G_\nu(1-u),G_\nu((1-u)+)]$. On the other hand, for all $k\in (G_\nu((1-u)+),G(u))$, $(1-u)<\inf\{\theta:\theta\in\partial\sE_u(k)\}$, and therefore $h_u=L^{P_\nu,1-u}_{G_\nu(1-u)}<\sE_u$ on $(G_\nu((1-u)+),G(u)]$. Finally, since $L^{\sE_u,1-u}_{G(u)}=L^{D,1-u}_{G(u)}$ is parallel to $L^{P_\nu,1-u}_{G_\nu(1-u)}$, we also have that $h_u=L^{P_\nu,1-u}_{G_\nu(1-u)}<L^{D,1-u}_{G(u)}=\sE_u$ on $(G(u),\infty)$, as claimed. It follows that $h_u\leq\sE^c_u\leq\sE_u$ everywhere.

We claim that $\sE^c_u=h_u$ everywhere. First, since $\sE_u(k)=P_\nu(k)=h_u(k)$ for $k\leq G_\nu((1-u)+)$ we must have that $\sE^c_u=h_u$ on $(-\infty,G_\nu((1-u)+)]$. On the other hand, if $\sE^c_u(k)>h_u(k)$ for some $k>G_\nu((1-u)+)$, then by convexity of $\sE^c_u$ we have that $(1-u)<\inf\{\theta:\theta\in\partial\sE^c_u(k)\}$. Then, since $\sE_u=L^{D,1-u}_{G(u)}$ on $[G(u),\infty)$, for large enough $\hat k\geq k$ we have that $\sE_u(\hat k)<\sE^c_u(\hat k)$, a contradiction.

Finally, $T(u)=R(u)=\inf\{k:k \leq G(u), D(k) = {L}^{\sE^c_u ,\phi(u)}_{G(u)}(k) \}=\inf\{k:k \leq G(u), D(k) = {L}^{P_\nu ,1-u}_{G_\nu(1-u)}(k) \}=\inf\{k:k \leq G(u), P_\nu(k) = {L}^{P_\nu ,1-u}_{G_\nu(1-u)}(k) \}=G_\nu(1-u)$, which finishes the proof.	
\end{proof}

Using Proposition \ref{prop:U=0Construct} we immediately have the following.
\begin{thm}\label{thm:constructionSUPM}
	Define $T$ as in \eqref{eq:T}. Then it defines a construction as in Theorem~\ref{thm:mainTHM_intro} such that $\sL(Y)=\nu$. In particular, $T$ defines the (lifted) increasing supermartingale coupling, which coincides with the (lifted) antitone coupling.
\end{thm}
\begin{figure}[H]
	\centering
	\begin{tikzpicture}[scale=1]

	\begin{axis}[%
	width=4.028in,
	height=2.754in,
	at={(1.011in,0.642in)},
	scale only axis,
	xmin=-2,
	xmax=2,
	ymin=-0.5,
	ymax=2,
	axis line style={draw=none},
	ticks=none
	]
	\draw[black, thin] (-2,0) -- (2,0);
\addplot [color=black, dashed, line width=1.0pt, forget plot]
table[row sep=crcr]{%
	0.545454545454545	0.89669401425384\\
	0.634803622956617	0.933312919164825\\
	0.727272727272727	0.962809410137732\\
	0.773025705488219	0.974241993218211\\
	0.818778683703711	0.983895490286957\\
	0.86393479639731	0.990742557411827\\
	0.909090909090909	0.995868226707272\\
	1.00007025233224	1.00013968779498\\
	1.09090909090909	1.00000024012117\\
	1.18844925769827	1.00000015911441\\
	1.27272727272727	1.00000060644527\\
	1.35934851571652	0.999999558130977\\
	1.45454545454545	1.00000009224857\\
	1.54994123691727	0.99999819611556\\
	1.63636363636364	1.00002907474622\\
	1.70498279509592	0.999509618190581\\
	1.81818181818182	1.00033762370232\\
	1.90740209505949	1.0000009156506\\
	2	0.999998969134129\\
};
\addplot [color=blue, dotted, line width=1.0pt, forget plot]
table[row sep=crcr]{%
		-0.545454545454545	0.103305838975067\\
	-0.447954129729544	0.15237743659336\\
	-0.363636363636364	0.202508554005905\\
	-0.267949040379642	0.26794919827413\\
	-0.181818181818182	0.335048153804228\\
	-0.0917523442360944	0.412456743706039\\
	0	0.5\\
	0.0909329275051746	0.586798352495356\\
	0.181818181818182	0.665288842775788\\
	0.276340803623378	0.738158725561154\\
	0.363636363636364	0.797520730795605\\
	0.46183178576623	0.855187073053214\\
	0.545454545454545	0.897727072783314\\
	0.634803622956617	0.942398935407721\\
	0.727272727272727	0.988635862509857\\
	0.818778683703711	1.03470531148313\\
	0.909090909090909	1.07954590825459\\
	1.00007025233224	1.12517487582868\\
	1.09090909090909	1.17045475642469\\
	1.18844925769827	1.21922501403464\\
	1.27272727272727	1.26136440028013\\
	1.35934851571652	1.30467359645597\\
	1.45454545454545	1.35227286470508\\
	1.54994123691727	1.39996926004764\\
	1.63636363636364	1.44318095599006\\
	1.70498279509592	1.47700032634203\\
	1.81818181818182	1.53409158310667\\
	1.90740209505949	1.57870114905301\\
	2	1.62499938433451\\
};
\addplot [color=red, line width=1.0pt, forget plot]
table[row sep=crcr]{%
	-2	0\\
	-1.91657902485342	0\\
	-1.81818181818182	0\\
	-1.72300276879786	0\\
	-1.63636363636364	0\\
	-1.54716681534369	0\\
	-1.45454545454545	0\\
	-1.35423690138381	0\\
	-1.27272727272727	0\\
	-1.17140721081572	0\\
	-1.09090909090909	0\\
	-1.00320655232955	0\\
	-0.956148730710229	0.000961467324401677\\
	-0.909090909090909	0.0041322321180808\\
	-0.865394886605159	0.00905926621532028\\
	-0.821698864119408	0.0158956463592492\\
	-0.774485795696068	0.0254283302165158\\
	-0.727272727272727	0.0371900935858536\\
	-0.676567644997657	0.0522734653412368\\
	-0.625862562722587	0.0699894241241011\\
	-0.545454545454545	0.103305838975067\\
	-0.447954129729544	0.151023049974868\\
	-0.363636363636364	0.19321112905674\\
	-0.267949040379642	0.241025370513682\\
	-0.181818181818182	0.284428651449477\\
	-0.0917523442360944	0.329123694206483\\
	0	0.374999671405176\\
	0.0909329275051746	0.420466289804514\\
	0.181818181818182	0.465908899271805\\
	0.276340803623378	0.513170429123079\\
	0.363636363636364	0.556818055811411\\
	0.46183178576623	0.605915572347986\\
	0.545454545454545	0.647727505014768\\
	0.634803622956617	0.692399392099147\\
	0.727272727272727	0.738636155960953\\
	0.818778683703711	0.784403178686012\\
	0.909090909090909	0.829545998711309\\
	1.00007025233224	0.87503482275064\\
	1.09090909090909	0.920454170811474\\
	1.18844925769827	0.969224799778874\\
	1.27272727272727	1.01136496434949\\
	1.35934851571652	1.05467445676585\\
	1.45454545454545	1.10227298177649\\
	1.54994123691727	1.14997087120432\\
	1.63636363636364	1.19375611457704\\
	1.70498279509592	1.2270011232803\\
	1.81818181818182	1.28409099284289\\
	1.90740209505949	1.32870093859336\\
	2	1.37500042464913\\
};
	\node (lNU)[scale=1] at (-1,-0.2) {$\ell_\nu$};
\draw[gray,dashed,thin] (-1,-0.1) -- (-1,0);

	\node (rNU)[scale=1] at (0,-0.2) {$r_\nu=\ell_\mu$};
\draw[gray,dashed,thin] (0,-0.1) -- (0,0.5);

\node (rMU)[scale=1] at (1,-0.2) {$r_\mu$};
\draw[gray,dashed,thin] (1,-0.1) -- (1,1);

\node (G)[scale=1] at (0.5,-0.3) {$G(u)$};
\draw[gray,dashed,thin] (0.5,-0.2) -- (0.5,0.85);

\node (T)[scale=1] at (-0.5,-0.3) {$T(u)$};
\draw[gray,dashed,thin] (-0.5,-0.2) -- (-0.5,0.13);
\node (E)[scale=1,blue] at (1.1,1.5) {$y\mapsto\sE_{u}(y)$};
\node (D)[scale=1,black] at (1.7,0.8) {$y\mapsto D(y)$};
\node (Ec)[scale=1,red] at (-1.5,0.2) {$y\mapsto \sE^c_u(y)$};
	\end{axis}
	
	\end{tikzpicture}
	\caption{Plot of locations of $G(u)=G_\mu(u)$ and $T(u)=T_{\mu,\nu}(u)$ in the case where $\mu\leq_{cd}\nu$ and $u^*=0$. The dashed curve represents $D=D_{\mu,\nu}$. Since the support of $\mu$ is (strictly) to the right of the support of $\nu$, $D$ is convex to the left of $r_\nu$, linear on $(r_\nu,\ell_\mu)$ and concave to the right of $\ell_\nu$. In particular, $D=P_\nu$ on $(-\infty,\ell_\mu]$ and $D=P_{\delta_{\overline\nu/\nu(\R)}}-P_\mu$ on $(\ell_\mu,\infty)$. The dotted curve corresponds to the graph of $\sE_{u}=\sE_u^{\mu,\nu}$. Note that $D=\sE_u$ on $(-\infty, G(u+)]$, while $\sE_u$ is linear (with slope $1-u$) and $D<\sE_u$ on $(G(u+),\infty)$. The solid curve below $\sE_{u}$ represents $\sE^c_{u}$. The convex hull $\sE^c_{u}$ is linear on $[T(u),\infty)$, and its slope is given by $\phi(u)=1-u$. Note that the linear section of $\sE_u$ on $(G(u+),\infty)$ is parallel to the linear section of $\sE^c_u$ on $(T(u),\infty)$.}
	\label{fig:T}
\end{figure}
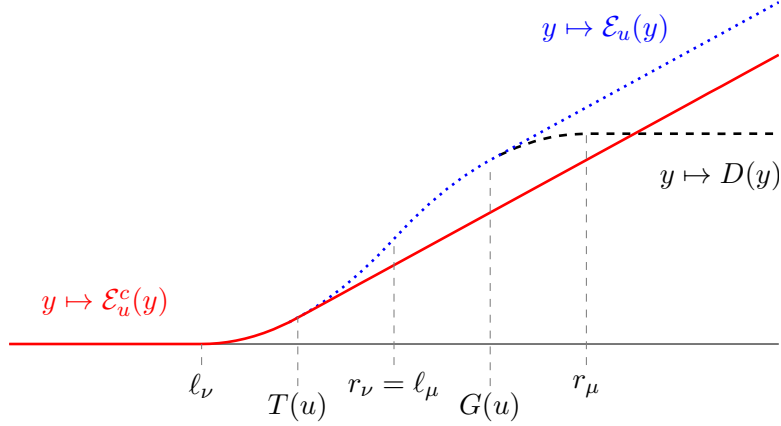
\subsubsection{The general supermartingale case: $u^*\in(0,1)$}\label{sec:u^*General}
In this section we combine the constructions of Sections \ref{sec:u=0} and \ref{sec:u=1},	and show how to build the increasing supermartingale coupling when $u^*\in(0,1)$.

Let $\mu,\nu\in\sP$ with $\mu\leq_{cd}\nu$ and suppose that $u^*\in(0,1)$. Recall that in this case $c(u^*)=\overline{\mu_{u^*}}-\overline{S^\nu(\mu_{u^*})}=0$ and therefore $\mu_{u^*}\leq_c{S^\nu(\mu_{u^*})}$. Hence we can embed $\mu_{u^*}$ into $S^\nu(\mu_{u^*})$ using a construction similar the one provided in Section \ref{sec:u=0}. In particular, define $\hat D:=P_{S^\nu(\mu_{u^*})}-P_{\mu_{u^*}}$ and $\hat{\sE}_u:=\hat D+P_{\mu_{u^*}}-P_{\mu_u}$, $u\in(0,u^*]$. ($\hat{\sE}_u$ serves the same role for $\mu_{u^*}$ and $S^\nu(\mu_{u^*})$ as $\sE_u$ did for $\mu_u$ and $P_\nu$ in Section \ref{sec:u=0}.)

Define $\hat R,\hat S:(0,u^*]\mapsto\R$ by
\begin{align}
\label{eq:hR} \hat{R}(u)&:= \tilde{R}_{\mu_{u^*},S^\nu(\mu_{u^*})},\\
\label{eq:hS} \hat{S}(u) &:= \tilde{S}_{\mu_{u^*},S^\nu(\mu_{u^*})},
\end{align}
where $\tilde{R}$ and $\tilde{S}$ are defined as in \eqref{eq:Rgeneral} and \eqref{eq:Sgeneral}, respectively. Furthermore,  for $u\in(0,u^*]$ and $v\in(0,1)$, define $\hat{Y}(u,v) \in \{\hat{R}(u),\hat{S}(u) \}$ by $\hat{Y}(u,v) = G(u)$ on $\hat{R}(u)=\hat{S}(u)$ and
\begin{equation}
\hat{Y}(u,v) = \hat{R}(u) I_{\{ v \leq \frac{\hat{S}(u) - G(u)}{\hat{S}(u)-\hat{R}(u)} \}} +  \hat{S}(u) I_{ \{ v > \frac{\hat{S}(u) - G(u)}{\hat{S}(u)-\hat{R}(u)} \} }
\label{eq:YUVdef3}
\end{equation}
otherwise.

 We now deal with $\tilde\mu:=(\mu-\mu_{u^*})$ and $\tilde\nu:=(\nu-S^\nu(\mu_{u^*}))$. Recall that, by Proposition \ref{prop:support}, the support of $\tilde\nu$ is strictly to the left of the support of $\tilde\mu$. Let $\tilde G=G_{\tilde\mu}$ be the (left-continuous) quantile function of $\tilde\mu$. Consider the lift $(\tilde{\mu}_u)_{u\in(0,1-u^*)}$ of $\tilde\mu$ defined as in \eqref{eq:muLift} but with $\tilde\mu$ and $\tilde G$ in place of $\mu$ and $G$, respectively, and set $\tilde{\mu}_0$ to be the zero measure and $\tilde{\mu}_{1-u^*}=\tilde\mu$. Let $\tilde{D}:=P_{\tilde\nu}-P_{\tilde\mu}$ and $\tilde{\sE}_u:=\tilde D+P_{\tilde\mu}-P_{\tilde{\mu}_u}$ for $u\in(0,1-u^*)$.

Define $\tilde T:(0,1-u^*)\mapsto\R$ by
\begin{align}
\label{eq:tRT} \tilde{T}(u):=T_{\tilde\mu,\tilde\nu}(u),\quad u\in(0,1-u^*),
\end{align}
where $T_{\tilde\mu,\tilde\nu}$ is defined as in \eqref{eq:T}.

We are now ready to prove Theorem \ref{thm:mainTHM_intro} in the case when $u^*\in(0,1)$.
\begin{thm}\label{thm:construction3}
	Define $\hat R$ and $\hat S$ as in \eqref{eq:hR} and \eqref{eq:hS}, and $\tilde T$ as in \eqref{eq:tRT}.
	Then $(\hat R,\hat S, \tilde T)$ is second-order left-monotone and first-order right-monotone with respect to $((0,1),(0,u^*],G)$.
	
	Furthermore, let $\hat Y$ be as in \eqref{eq:YUVdef3} and set $$Y(u,v)=I_{\{u\in(0,u^*]\}}\hat{Y}(u,v)+I_{\{u\in(u^*,1)\}}\tilde{T}(u-u^*),\quad u,v\in(0,1).$$
	Then $(\hat R,\hat S,\tilde T)$ defines a construction as in Theorem~\ref{thm:mainTHM_intro} such that $\sL(Y)=\nu$. In particular, $(\hat R,\hat S,\tilde T)$ defines the (lifted) increasing supermartingale coupling.
\end{thm}
\begin{proof}
	Let $U,V$ be two independent uniform random variables on $[0,1]$. Fix $y\in\R$. Then using Theorem \ref{thm:constructionMART} we have that $\Prob[Y(U,V)\leq y,U\leq u^*]=\Prob[\hat{Y}(U_{u^*},V)\leq y]=S^\nu(\mu_{u^*})((-\infty,y])$, where $U_{u^*}$ is a uniform random variable on $[0,u^*]$, independent of $V$. On the other hand, using Theorem \ref{thm:constructionSUPM} we have that $\Prob[Y(U,V)\leq y,U> u^*]=\Prob[\tilde{T}(U_{1-u^*})\leq y]=(\nu-S^\nu(\mu_{u^*}))((-\infty,y])$, where $U_{{1-u^*}}$ is a uniform random variable on $[0,1-u^*]$, independent of $V$. It follows that $\sL(Y(U,V))=\nu$.
	
	Furthermore, from Theorem \ref{thm:constructionMART} we have that $(\hat R,\hat S)$ is second-order left-monotone with respect to $G$ on $(0,u^*]$. On the other hand, from Theorem \ref{thm:constructionSUPM} it follows that $u\mapsto\tilde T(u-u^*)$ is non-increasing and $\tilde T(u-u^*)<\tilde G(u-u^*)=G(u)$ for all $u\in(u^*,1)$. We are left to show that for all $u,v\in(0,1)$ with $u\leq u^*<v$, $\tilde T(v-u^*)\notin (\hat R(u),\hat S(u))$.

	If $\hat R(u)<\tilde T(v-u^*) < \hat S(u)$, then by Remark \ref{rem:R<x<S} we have that $\tilde T(v-u^*)$ belongs to an interior of an irreducible component of $\mu_{u^*}$ and $S^\nu(\mu_{u^*})$. We will show that this cannot happen by proving that $D_{\mu_{u^*},S^\nu(\mu_{u^*})}(\tilde T(v-u^*))=0$.
	
By construction we have that, for $v\in(u^*,1)$, $\tilde T(v-u^*)=G_{\nu-S^\nu(\mu_{u^*})}(1-v)$ and 
$$
(\sE^{\nu-S^\nu(\mu_{u^*}),\mu-\mu_{u^*}}_{v-u^*})^c(k)=\begin{cases}P_{\nu-S^\nu(\mu_{u^*})}(k),&k\leq G_{\nu-S^\nu(\mu_{u^*})}(1-v),\\
L^{P_{\nu-S^\nu(\mu_{u^*})},1-v}_{G_{\nu-S^\nu(\mu_{u^*})}}(k),&k>G_{\nu-S^\nu(\mu_{u^*})}(1-v).
\end{cases}
$$
(see the proof Proposition \ref{prop:U=0Construct}). Furthermore, since $G_{\nu-S^\nu(\mu_{u^*})}$ is taken to be left-continuous, $P_{\nu-S^\nu(\mu_{u^*})}>L^{P_{\nu-S^\nu(\mu_{u^*})},1-v}_{G_{\nu-S^\nu(\mu_{u^*})}}$ on $(-\infty,\tilde{T}(v-u^*)=G_{\nu-S^\nu(\mu_{u^*})}(1-v))$. Note that this implies that if $P_{\nu-S^\nu(\mu_{u^*})}$ is linear on $(x,\tilde{T}(v-u^*))$ for some $x<\tilde{T}(v-u^*)$, then the slope of this linear section must be strictly smaller than $1-v$. In particular, there does not exist $\epsilon>0$ such that $P_{\nu-S^\nu(\mu_{u^*})}$ is linear on $(\tilde{T}(v-u^*)-\epsilon,\tilde{T}(v-u^*)+\epsilon)$.
\begin{figure}[H]
	\centering
	\begin{tikzpicture}[scale=1]

	\begin{axis}[%
	width=4.528in,
	height=3.754in,
	at={(1.011in,0.642in)},
	scale only axis,
	xmin=-5,
	xmax=9,
	ymin=-0.5,
	ymax=3,
	axis line style={draw=none},
	ticks=none
	]
	\addplot [color=black, dashed, line width=1.0pt, forget plot]
	table[row sep=crcr]{%
		4.5	2.1875\\
		4.55213810946661	2.19985493175197\\
		4.61363636363636	2.21268078512397\\
		4.67312326950134	2.22328790076463\\
		4.72727272727273	2.23140495867769\\
		4.7830207404102	2.23823000022697\\
		4.81196491565964	2.24116070176428\\
		4.84090909090909	2.24367252066116\\
		4.8722555137721	2.24592033655959\\
		4.90360193663512	2.24767685334487\\
		4.92907369559029	2.2487423648357\\
		4.95454545454545	2.24948347107438\\
		4.98620797389281	2.24995244500396\\
		5.01787049324017	2.25\\
		5.06818181818182	2.25\\
		5.12299590479403	2.25\\
		5.18181818181818	2.25\\
		5.23643820992537	2.25\\
		5.29545454545454	2.25\\
		5.35883589829838	2.25\\
		5.40909090909091	2.25\\
		5.47002866891903	2.25\\
		5.52272727272727	2.25\\
		5.58253184976272	2.25\\
		5.63636363636364	2.25\\
		5.69265478485244	2.25\\
		5.75	2.25\\
		5.80683307969073	2.25\\
		5.86363636363636	2.25\\
		5.92271300226461	2.25\\
		5.97727272727273	2.25\\
		6.03864486610389	2.25\\
		6.09090909090909	2.25\\
		6.14675226434789	2.25\\
		6.20454545454546	2.25\\
		6.26173667731482	2.25\\
		6.31818181818182	2.25\\
		6.37504390770765	2.25\\
		6.43181818181818	2.25\\
		6.49278078606142	2.25\\
		6.54545454545454	2.25\\
		6.59959282232283	2.25\\
		6.65909090909091	2.25\\
		6.71871327307329	2.25\\
		6.77272727272727	2.25\\
		6.81561424693495	2.25\\
		6.88636363636364	2.25\\
		6.94212630941218	2.25\\
		7	2.25\\
	};
	\addplot [color=black!30!green, dashdotted, line width=1.0pt, forget plot]
	table[row sep=crcr]{%
		0.5	0.765625\\
		0.572993353253259	0.788435422891643\\
		0.659090909090909	0.815340909090909\\
		0.74237257730187	0.841366430406834\\
		0.818181818181818	0.865056818181818\\
		0.896229036574275	0.889446573929461\\
		0.977272727272727	0.914772727272727\\
		1.06504271128917	0.942200847277865\\
		1.13636363636364	0.964488636363636\\
		1.22501869053624	0.992193340792575\\
		1.29545454545455	1.01420454545455\\
		1.37219426671164	1.03818570834739\\
		1.45454545454545	1.06392045454545\\
		1.53101349389552	1.08781671684235\\
		1.61363636363636	1.11363636363636\\
		1.70237025761774	1.14136570550554\\
		1.77272727272727	1.16335227272727\\
		1.85804013648665	1.19001254265208\\
		1.93181818181818	1.21306818181818\\
		2.01554458966781	1.23923268427119\\
		2.09090909090909	1.26278409090909\\
		2.16971669879342	1.28741146837294\\
		2.25	1.3125\\
		2.32956631156703	1.3373644723647\\
		2.40909090909091	1.36221590909091\\
		2.49179820317046	1.38806193849077\\
		2.56818181818182	1.41222236570248\\
		2.65410281254545	1.4402663587226\\
		2.72727272727273	1.46487603305785\\
		2.80545317008704	1.49191046809696\\
		2.88636363636364	1.52069344008264\\
		2.96643134824075	1.54998218398909\\
		3.04545454545455	1.57967458677686\\
		3.12506147079071	1.61037557476379\\
		3.20454545454545	1.6418194731405\\
		3.28989310048598	1.67646228828908\\
		3.36363636363636	1.70712809917355\\
		3.43942995125196	1.73935489934806\\
		3.52272727272727	1.77560046487603\\
		3.60619858230261	1.81279176343758\\
		3.68181818181818	1.84723657024793\\
		3.74185994570893	1.8750947408313\\
		3.84090909090909	1.92203641528926\\
		3.91897683317705	1.95989871368615\\
		4	2\\
	};
	\addplot [color=black!30!green, dashdotted, line width=1.0pt, forget plot]
	table[row sep=crcr]{%
		4	2\\
		4.06256573135994	2.03128286567997\\
		4.13636363636364	2.06818181818182\\
		4.2077479234016	2.1038739617008\\
		4.27272727272727	2.13636363636364\\
		4.33962488849224	2.16981244424612\\
		4.40909090909091	2.20454545454545\\
		4.48432232396214	2.24216116198107\\
		4.54545454545454	2.27272727272727\\
		4.62144459188821	2.3107222959441\\
		4.68181818181818	2.34090909090909\\
		4.74759508575284	2.37379754287642\\
		4.81818181818182	2.40909090909091\\
		4.88372585191044	2.44186292595522\\
		4.95454545454546	2.47727272727273\\
		5.03060307795806	2.51530153897903\\
		5.09090909090909	2.54545454545455\\
		5.16403440270284	2.58201720135142\\
		5.22727272727273	2.61363636363636\\
		5.29903821971527	2.64951910985763\\
		5.36363636363636	2.68181818181818\\
		5.43118574182293	2.71559287091146\\
		5.5	2.75\\
		5.56819969562888	2.78409984781444\\
		5.63636363636364	2.81818181818182\\
		5.70725560271753	2.85362780135877\\
		5.77272727272727	2.88636363636364\\
		5.84637383932467	2.92318691966234\\
		5.90909090909091	2.95454545454545\\
		5.97610271721746	2.98805135860873\\
		6.04545454545454	3.02272727272727\\
		6.11408401277778	3.05704200638889\\
		6.18181818181818	3.09090909090909\\
		6.25005268924918	3.12502634462459\\
		6.31818181818182	3.15909090909091\\
		6.3913369432737	3.19566847163685\\
		6.45454545454546	3.22727272727273\\
		6.51951138678739	3.2597556933937\\
		6.59090909090909	3.29545454545455\\
		6.66245592768795	3.33122796384398\\
		6.72727272727273	3.36363636363636\\
		6.77873709632194	3.38936854816097\\
		6.86363636363636	3.43181818181818\\
		6.93055157129462	3.46527578564731\\
		7	3.5\\
	};
	\addplot [color=black, dashed, line width=1.0pt, forget plot]
	table[row sep=crcr]{%
		-2.5	0.140625\\
		-2.45828951242671	0.148554451718342\\
		-2.40909090909091	0.158186983471074\\
		-2.36150138439893	0.167792357082914\\
		-2.31818181818182	0.176782024793388\\
		-2.27358340767184	0.186282140641623\\
		-2.22727272727273	0.196410123966942\\
		-2.1771184506919	0.207681071425493\\
		-2.13636363636364	0.217071280991736\\
		-2.08570360540786	0.229033167896779\\
		-2.04545454545455	0.238765495867769\\
		-2.00160327616477	0.24959934161471\\
		-1.95454545454545	0.261234504132231\\
		-1.9108494320597	0.271790902999819\\
		-1.86363636363636	0.282928719008264\\
		-1.81293128136129	0.294580010566356\\
		-1.77272727272727	0.303589876033058\\
		-1.72397706486477	0.314243942488765\\
		-1.68181818181818	0.323217975206612\\
		-1.63397452018982	0.333132954210653\\
		-1.59090909090909	0.341813016528926\\
		-1.54587617211805	0.350641678779853\\
		-1.5	0.359375\\
		-1.45453353624741	0.367770761995725\\
		-1.40909090909091	0.375903925619835\\
		-1.36182959818831	0.384088759093642\\
		-1.31818181818182	0.39139979338843\\
		-1.26908410711689	0.399339095566459\\
		-1.22727272727273	0.405862603305785\\
		-1.18259818852169	0.412591345281576\\
		-1.13636363636364	0.419292355371901\\
		-1.09061065814814	0.425660524520854\\
		-1.04545454545455	0.431689049586777\\
		-0.999964873833882	0.43750439084788\\
		-0.954545454545455	0.443310950413223\\
		-0.905775371150867	0.449832971148751\\
		-0.863636363636364	0.455707644628099\\
		-0.820325742141739	0.46197695966584\\
		-0.772727272727273	0.469137396694215\\
		-0.725029381541366	0.4765968798708\\
		-0.681818181818182	0.48360020661157\\
		-0.647508602452041	0.489327061277577\\
		-0.590909090909091	0.499096074380165\\
		-0.546298952470256	0.507077920974317\\
		-0.5	0.515625\\
	};
	\addplot [color=black, dashed, line width=1.0pt, forget plot]
	table[row sep=crcr]{%
		0.5	0.765625\\
		0.541710487573291	0.778768262665004\\
		0.590909090909091	0.794550619834711\\
		0.638498615601069	0.810104684033047\\
		0.681818181818182	0.824509297520661\\
		0.726416592328157	0.839584214682642\\
		0.772727272727273	0.855501033057851\\
		0.822881549308096	0.873041265089005\\
		0.863636363636364	0.887525826446281\\
		0.914296394592138	0.905820217220796\\
		0.954545454545455	0.920583677685951\\
		0.998396723835225	0.936898932094113\\
		1.04545454545455	0.954416322314049\\
		1.0891505679403	0.970434723992356\\
		1.13636363636364	0.987474173553719\\
		1.18706871863871	1.00546360039619\\
		1.22727272727273	1.01949896694215\\
		1.27602293513523	1.03624680938067\\
		1.31818181818182	1.05049070247934\\
		1.36602547981018	1.06638613918693\\
		1.40909090909091	1.08044938016529\\
		1.45412382788195	1.0949071572651\\
		1.5	1.109375\\
		1.54546646375259	1.1234540699648\\
		1.59090909090909	1.13726756198347\\
		1.63817040181169	1.1513600593201\\
		1.68181818181818	1.1641270661157\\
		1.73091589288312	1.17820358217685\\
		1.77272727272727	1.18995351239669\\
		1.81740181147831	1.20226657171636\\
		1.86363636363636	1.21474690082645\\
		1.90938934185186	1.22683419225234\\
		1.95454545454545	1.23850723140496\\
		2.00003512616612	1.25000878161864\\
		2.04545454545455	1.26149276859504\\
		2.09422462884913	1.27411104975489\\
		2.13636363636364	1.28525309917355\\
		2.17967425785826	1.29693624189812\\
		2.22727272727273	1.31004648760331\\
		2.27497061845863	1.32346820717813\\
		2.31818181818182	1.3358729338843\\
		2.35249139754796	1.34588848597107\\
		2.40909090909091	1.36273243801653\\
		2.45370104752974	1.37629055191554\\
		2.5	1.390625\\
	};
	\addplot [color=blue, dotted, line width=1.0pt, forget plot]
	table[row sep=crcr]{%
		4	2\\
		4.06256573135994	2.03030424799482\\
		4.13636363636364	2.06353305785124\\
		4.2077479234016	2.09308416178138\\
		4.27272727272727	2.11776859504132\\
		4.33962488849224	2.14097617802528\\
		4.40909090909091	2.16270661157025\\
		4.48432232396214	2.18351913360905\\
		4.54545454545454	2.19886363636364\\
		4.62144459188821	2.21786114797205\\
		4.68181818181818	2.23295454545455\\
		4.74759508575284	2.24939877143821\\
		4.81818181818182	2.26704545454545\\
		4.88372585191044	2.28343146297761\\
		4.95454545454546	2.30113636363636\\
		5.03060307795806	2.32015076948951\\
		5.09090909090909	2.33522727272727\\
		5.16403440270284	2.35350860067571\\
		5.22727272727273	2.36931818181818\\
		5.29903821971527	2.38725955492882\\
		5.36363636363636	2.40340909090909\\
		5.43118574182293	2.42029643545573\\
		5.5	2.4375\\
		5.56819969562888	2.45454992390722\\
		5.63636363636364	2.47159090909091\\
		5.70725560271753	2.48931390067938\\
		5.77272727272727	2.50568181818182\\
		5.84637383932467	2.52409345983117\\
		5.90909090909091	2.53977272727273\\
		5.97610271721746	2.55652567930437\\
		6.04545454545454	2.57386363636364\\
		6.11408401277778	2.59102100319445\\
		6.18181818181818	2.60795454545455\\
		6.25005268924918	2.62501317231229\\
		6.31818181818182	2.64204545454545\\
		6.3913369432737	2.66033423581842\\
		6.45454545454546	2.67613636363636\\
		6.51951138678739	2.69237784669685\\
		6.59090909090909	2.71022727272727\\
		6.66245592768795	2.72811398192199\\
		6.72727272727273	2.74431818181818\\
		6.77873709632194	2.75718427408048\\
		6.86363636363636	2.77840909090909\\
		6.93055157129462	2.79513789282365\\
		7	2.8125\\
	};
	\addplot [color=red, line width=1.0pt, forget plot]
	table[row sep=crcr]{%
		-7	0\\
		-6.70802658698696	0\\
		-6.36363636363636	0\\
		-6.03050969079252	0\\
		-5.72727272727273	0\\
		-5.4150838537029	0\\
		-5.09090909090909	0\\
		-4.73982915484333	0\\
		-4.45454545454545	0\\
		-4.09992523785504	0\\
		-3.95905352801843	0.000104788347983613\\
		-3.81818181818182	0.00206611570247935\\
		-3.66470237566762	0.00702653105518359\\
		-3.51122293315342	0.0149314388171964\\
		-3.3465205574858	0.0266897113617917\\
		-3.18181818181818	0.0418388429752066\\
		-3.02888210311806	0.0589418731027757\\
		-2.87594602441793	0.0789685837513662\\
		-2.71070028493624	0.103893359703969\\
		-2.54545454545455	0.132231404958678\\
		-2.19051896952906	0.198652693213302\\
		-1.90909090909091	0.251420454545455\\
		-1.56783945405341	0.315405102364987\\
		-1.27272727272727	0.370738636363636\\
		-0.937821641328747	0.43353344225086\\
		-0.636363636363636	0.490056818181818\\
		-0.32113320482633	0.551162107246044\\
		0	0.625\\
		0.318265246268111	0.704566311567028\\
		0.636363636363636	0.784090909090909\\
		0.967192812681825	0.866798203170456\\
		1.27272727272727	0.943181818181818\\
		1.61641125018181	1.02910281254545\\
		1.90909090909091	1.10227272727273\\
		2.22181268034816	1.18045317008704\\
		2.54545454545455	1.26136363636364\\
		2.86572539296299	1.34143134824075\\
		3.18181818181818	1.42045454545455\\
		3.50024588316283	1.50006147079071\\
		3.81818181818182	1.57954545454545\\
		4.15957240194393	1.66489310048598\\
		4.45454545454545	1.73863636363636\\
		4.75771980500782	1.81442995125196\\
		5.09090909090909	1.89772727272727\\
		5.42479432921044	1.98119858230261\\
		5.72727272727273	2.05681818181818\\
		5.96743978283571	2.11685994570893\\
		6.36363636363636	2.21590909090909\\
		6.67590733270821	2.29397683317705\\
		7	2.375\\
	};
	\node (lNU)[scale=1] at (-4,-0.2) {$\ell_\nu$};
	\draw[gray,dashed,thin] (-4,-0.1) -- (-4,0);
	
	\node (rNU)[scale=1] at (4,-0.2) {$\ell_{\tilde\mu}$};
	\draw[gray,dashed,thin] (4,-0.1) -- (4,2);
	
	
	\node (G)[scale=1] at (4.5,-0.4) {$G_\mu(u)$};
	\draw[gray,dashed,thin] (4.5,-0.3) -- (4.5,2.2);
	
	\node (T)[scale=1] at (0,-0.4) {$\tilde T(u-u^*)$};
	\draw[gray,dashed,thin] (0,-0.3) -- (0,0.6);
	\node (E)[scale=1,blue] at (7.5,2.9) {$y\mapsto\sE^{\mu,\nu}_{u}(y)$};
	\node (ET)[scale=1,red] at (7.5,1.9) {$y\mapsto P_{\tilde\nu-S^{\tilde\nu}(\tilde\mu_{u-u^*})}(y)$};
	
	\node (D)[scale=1,black] at (-2.8,0.5) {$y\mapsto D_{\mu,\nu}(y)$};
	\node (Ec)[scale=1,black!30!green] at (3.7,2.8) {$y\mapsto (\sE^{\mu,\nu}_{u^*})^c(y)$};
	\end{axis}
	
	\end{tikzpicture}
	\caption{Plot of locations of $u^*$, $G_\mu(u)=G_{\tilde\mu}(u-u^*)$ and $\tilde T(u-u^*)$ in the case $\mu\leq_{cd}\nu$ and $0<u^*<u<1$. The dashed curve represents $D_{\mu,\nu}$. The dash-dotted curve corresponds to $(\sE^{\mu,\nu}_{u^*})^c=(P_\nu-P_{\mu_{u^*}})^c=P_{\nu-S^\nu(\mu_{u^*})}=P_{\tilde\nu}$ (note that it is linear on $(r_{\tilde\nu}=r_\nu=r_{\nu-S^\nu(\mu_{u^*})}=\ell_{\tilde\mu},\infty)$). The dotted curve represents $\sE^{\mu,\nu}_u$. Note that $\sE^{\mu,\nu}_u\geq P_{\tilde\nu}$ on $(-\infty,r_\nu)$ and $\sE^{\mu,\nu}_u\leq P_{\tilde\nu}$ on $[r_\nu,\infty)$. The solid curve below $\sE^{\mu,\nu}_u$ and $P_{\tilde\nu}$ corresponds to $P_{\tilde\nu-S^{\tilde\nu}(\tilde\mu_{u-u^*})}=(P_{\tilde\nu}-P_{\tilde\mu_{u-u^*}})^c$. Note that $P_{\tilde\nu-S^{\tilde\nu}(\tilde\mu_{u-u^*})}=P_{\tilde\nu}$ on $(-\infty,\tilde{T}(u-u^*)]$ and $P_{\tilde\nu-S^{\tilde\nu}(\tilde\mu_{u-u^*})}<(P_{\tilde\nu}\wedge\sE^{\mu,\nu}_u)$ on $(\tilde{T}(u-u^*),\infty)$. Furthermore, $\sE^{\mu,\nu}_u$ is linear on $(G_\mu(u),\infty)$ while $P_{\tilde\nu-S^{\tilde\nu}(\tilde\mu_{u-u^*})}$ is linear on $(\tilde{T}(u-u^*),\infty)$, and both of these linear sections are parallel to each other.}
	\label{fig:general}
\end{figure}
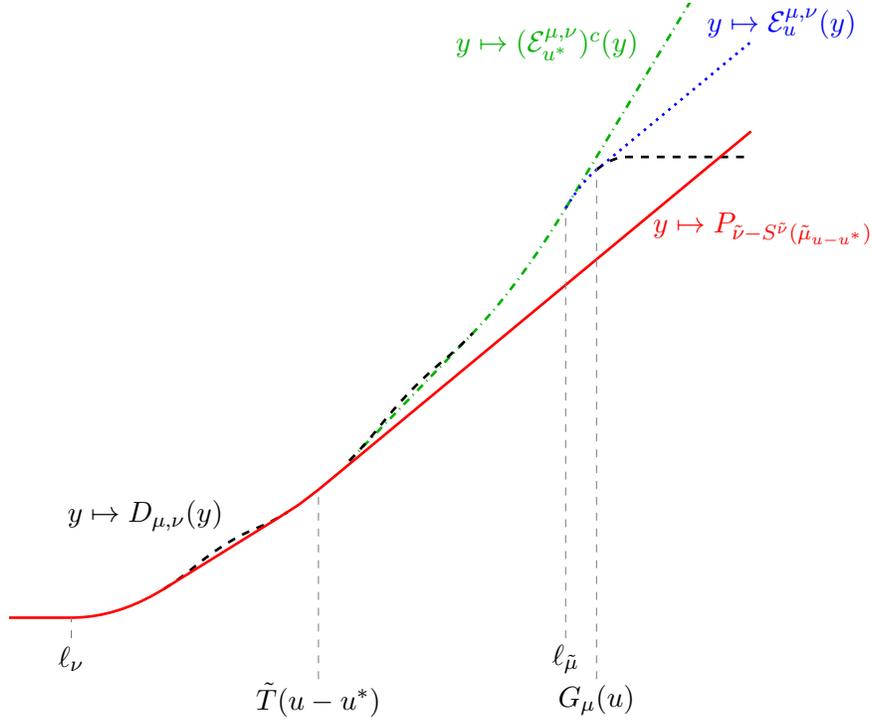

We claim that $(\sE^{\mu,\nu}_{u^*})^c(\tilde{T}(v-u^*))=\sE^{\mu,\nu}_{u^*}(\tilde{T}(v-u^*))$, see Figure \ref{fig:general}. Suppose not. Since $(\sE^{\mu,\nu}_{u^*})^c$ and $\sE^{\mu,\nu}_{u^*}$ are both continuous, $(\sE^{\mu,\nu}_{u^*})^c < \sE^{\mu,\nu}_{u^*}$ on $(\tilde{T}(v-u^*)-\epsilon,\tilde{T}(v-u^*)+\epsilon)$ for some $\epsilon>0$. Then by Lemma \ref{lem:linear} we have that $(\sE^{\mu,\nu}_{u^*})^c$ is linear on $(\tilde{T}(v-u^*)-\epsilon,\tilde{T}(v-u^*)+\epsilon)$. But $(\sE^{\mu,\nu}_{u^*})^c=(P_\nu-P_{\mu_{u^*}})^c=P_\nu-P_{S^\nu(\mu_{u^*})}=P_{\nu-S^\nu(\mu_{u^*})}$, and we have a contradiction since $P_{\nu-S^\nu(\mu_{u^*})}$ cannot be linear on an open interval including $\tilde{T}(v-u^*)$.

Finally, note that
$$
D_{\mu,\nu}-D_{\mu_{u^*},S^\nu(\mu_{u^*})}=P_{\nu-S^\nu(\mu_{u^*})}-(P_{\mu}-P_{\mu_{u^*}})=(\sE^{\mu,\nu}_{u^*})^c-P_{\mu-\mu_{u^*}}.
$$
Then since $(\sE^{\mu,\nu}_{u^*})^c(\tilde{T}(v-u^*))=\sE^{\mu,\nu}_{u^*}(\tilde{T}(v-u^*))=P_\nu(\tilde{T}(v-u^*))-P_{\mu_{u^*}}(\tilde{T}(v-u^*))$ we conclude that $D_{\mu_{u^*},S^\nu(\mu_{u^*})}(\tilde{T}(v-u^*))=0$, which finishes the proof.

\end{proof}
\section{Appendix}\label{sec:proofs}
\begin{proof}[Proof of Lemma \ref{lem:disjointSupport}]
	Let $G_\chi:[0,\chi(\R)]\mapsto\R$ be a quantile function of $\chi$. Then each $\zeta\in[0,\eta(\R)]$ defines a measure
	$$
	\theta^\zeta=\chi\lvert_{(-\infty,G_\chi(\zeta))}+\chi\lvert_{(G_\chi(\zeta+\chi(\R)-\eta(\R)),\infty)}+\alpha^\zeta\delta_{G_\chi(\zeta)}
	+\beta^\zeta\delta_{G_\chi(\zeta+\chi(\R)-\eta(\R))},
	$$
	where $0\leq\alpha^\zeta = \zeta - \chi((-\infty,G_\chi(\zeta))) \leq\chi(\{G_\chi(\zeta)\})$ and $0\leq\beta^\zeta = \eta(\R) - \zeta  - \chi((G_\chi(\zeta+\chi(\R)-\eta(\R)),\infty)) \leq\chi(\{G_\chi(\zeta+\chi(\R)-\eta(\R))\})$. By construction, $\theta^\zeta\leq\chi$ and $\theta^\zeta(\R)=\eta(\R)$.  Furthermore, $\overline{\theta^{0}}\geq\overline\eta\geq \overline{\theta^{\eta(\R)}}$ and $\overline{\theta^{\zeta}}$ is continuous and decreasing in $\zeta$, and therefore there exists $\zeta_*$ such that $\overline{\theta^{\zeta_*}}=\overline{\eta}$.
	
	Now let $f:\R\mapsto\R$ be convex and $g:\R\mapsto\R$ be linear with $g=f$ on $\{\ell_{\eta},r_\eta\}$. Then
	$$
	\int_{\R}fd\eta\leq\int_{\R}gd\eta=\int_{\R}gd\theta^{\zeta_*}\leq\int_{\R}fd\theta^{\zeta_*}.
	$$
	For the the first inequality we use that, by convexity of $f$, $g\geq f$ on $[\ell_{\eta},r_\eta]$, and that $\eta$ does not charge $\R\setminus[\ell_{\eta},r_\eta]$. For the quality we use that $\overline{\theta^{\zeta_*}}=\overline{\eta}$, $\theta^{\zeta_*}(\R)=\eta(\R)$ and the linearity of $g$. To deduce the second inequality we use that $g\leq f$ on $\R\setminus[\ell_{\eta},r_\eta]$ and that $\theta^{\zeta_*}$ does not charge $(\ell_{\eta},r_\eta)$. Since $f$ was arbitrary, $\eta\leq_c\theta^{\zeta_*}$. By Lemma \ref{lem:order} and Remark \ref{rem:orderCX} it follows that $\eta\leq_{pc}\chi$.
\end{proof}
\begin{proof}[Proof of Lemma \ref{lem:CSpotential}]
	
	Since $g\in\sD(a,b)$ and $f\in\sD(\alpha,\beta)$ with $g\geq f$, we have that
	\begin{align*}
	0\leq\lim_{k\to\infty}\{g(k)-f(k)\}&=\lim_{k\to\infty}\{g(k)-(ak-b)-f(k)+(\alpha k-\beta)+(a-\alpha)k-(b-\beta)\}\\
	&=\lim_{k\to\infty}\{(a-\alpha)k-(b-\beta)\}=\lim_{k\to\infty}h(k),
	\end{align*}
	and therefore $a\geq\alpha$. Also, if $\alpha = a$ then $\beta \geq b$.
	
	

	Now suppose $f \leq g$ and $g-f \geq h$. Note that in this case $\eta=0$, since $\lim_{k\to\infty}\{g(k)-f(k)\}=\lim_{k\to\infty}h(k)$. Then $g - f\geq h^+$ and since $h^+$ is convex, we have that $(g-f)\geq (g-f)^c\geq h^+$. Then, $\lim_{\lvert k\lvert\to\infty}\{g(k)-f(k)-  h^+(k) \}=0$, and it follows that $(g-f)^c\in\sD(a-\alpha,b-\beta)$.
	
	Now suppose that $\{k\in\R:h(k)>g(k)-f(k)\neq\emptyset$. Then $\eta>0$. We claim that $\eta<\infty$. Let $\{k_n\}_{n\geq1}$ be such that $\lim_{n\to\infty}\{h(k_n)-g(k_n)+f(k_n)\}=\eta$. Then (up to a subsequence) $\lim_{n\to\infty}k_n$ exists. Set $\bar k:=\lim_{n\to\infty}k_n$.
	
	Suppose $\bar k=\infty$. Then $\eta=\lim_{n\to\infty}\{h(k_n)-g(k_n)+f(k_n)\}=\lim_{k\to\infty}\{h(k)-g(k)+f(k)\}=0$, contradicting the fact that $\eta>0$. Hence $\bar k<\infty$. Then, by the continuity (and finiteness on $\R$) of $k\mapsto h(k)-g(k)+f(k)$,  $\eta=h(\bar k)-g(\bar k)+f(\bar k)<\infty$.
	
	Finally, since $g-f\geq 0$ and $g-f\geq \tilde h:=h-\eta$, we have that $g-f\geq(g-f)^c\geq \tilde h^+$. Then, since $g(\bar k)-f(\bar k)=\tilde h(\bar k)$, convexity of $(g-f)^c$ ensures that $(g-f)^c=\tilde h$ on $[\bar k,\infty)$. It follows that $(g-f)^c\in\sD(a-\alpha,b-\beta+\eta)$.
\end{proof}

\begin{proof}[Proof of Lemma \ref{lem:RSbound}]
	Fix $u\in(0,u^r_i-u^\ell_i)$. Consider the case $S_i(u) \leq r_i < \infty$ first.
	We will prove it by contradiction: suppose $r_i < S_i(u) \leq \infty$.
	Since $G_i(u)=G(u-u^\ell_i) < r_i < S_i(u)$, $R_i(u) \leq Q_{\mu_i,\nu_i}(u) < G_i(u) < r_i < S_i(u)$ (see Hobson and Norgilas \cite[Lemma 4.1]{HobsonNorgilas:21}). Therefore $(\sE^{\mu_i,\nu_i}_u)^c$ is linear on $( R_i(u),S_i(u) )$.
	But since $(\sE^{\mu_i,\nu_i}_u)^c = P_{\nu_i-S^{\nu_i}(\mu_{i,u})}$ (where $\mu_{i,u}$ is defined as in \eqref{eq:muLift} just with respect to $\mu_i$),  $(\sE^{\mu_i,\nu_i}_u)^c$ is linear on $(r_i,\infty)$, and we conclude that $(\sE^{\mu_i,\nu_i}_u)^c$ is linear on $(R_i(u),\infty)$.
	It follows that $(\nu_i-S^{\nu_i}(\mu_{i,u}))(R_i(u),\infty)=0$.
	But $(\mu_i-\mu_{i,u})(R_i(u),\infty)=(\mu_i-\mu_{i,u})(\R)=1-u >0$.
	Therefore $\overline{\mu_i-\mu_{i,u}} > (1-u)R_i(u) \geq \overline{(\nu_i-S^{\nu_i}(\mu_{i,u}))}$, contradicting the fact that $(\mu_i-\mu_{i,u}) \leq _{c} (\nu_i-S^{\nu_i}(\mu_{i,u}))$.
	
	Consider the case $r_i=S_i(u)=\infty$.
	Then again, $R_i(u)<G_i(u)<S_i(u)$ and therefore $(\sE^{\mu_i,\nu_i}_u)^c$ is linear on $(R_i(u),\infty)$.
	It follows that $(\nu_i-S^{\nu_i}(\mu_{i,u}))$ does not charge $(R_i(u),\infty)$, while the support of $(\mu_i-\mu_{i,u})$ is contained $(R_i(u),\infty)$. We conclude as in the previous case.
	
	We now turn to $R_i$.	We need to show that either $-\infty < \ell_i \leq R_i(u)$ or $-\infty = \ell_i < R_i(u)$.
	
	Suppose $-\infty \leq R_i(u) < \ell_i$.
	Then $R_i(u) < G_i(u) \leq S_i(u)$, and $(\sE^{\mu_i,\nu_i}_u)^c$ is linear on $(R_i(u),S_i(u))$.
	Since $0 \leq (\sE^{\mu_i,\nu_i}_u)^c \leq D_{\mu_i,\nu_i} = 0$ on $(-\infty , \ell _i]$, it follows that $\phi_{\mu_i,\nu_i}(u)=0$ and then in fact we have that $R_i(u)=-\infty$.
	Hence, $(\sE^{\mu_i,\nu_i}_u)^c$=0 on $(-\infty, S_i(u)]$.
	Since $(\sE^{\mu_i,\nu_i}_u)^c(S_i(u))=\sE^{\mu_i,\nu_i}_u(S_i(u))\geq D_{\mu_i,\nu_i}(S_i(u))$ we also have that $D_{\mu_i,\nu_i}(S_i(u))=0$ and therefore $S_i(u)=r_i$.
	It follows that $(\nu_i-S^{\nu_i}(\mu_{i,u}))$ is supported on $\{ S_i(u)=r_i\}$. But the support of $(\mu_i-\mu_{i,u})$ is contained in $[G_i(u),r_i)$. Hence $(\mu_i-\mu_{i,u})$ and $(\nu_i-S^{\nu_i}(\mu_{i,u}))$ cannot be in convex order, a contradiction.
	
	Finally suppose $-\infty = R_i(u) = \ell_i$. Then either $\phi_{\mu_i,\nu_i}(u)=0$ or $\phi_{\mu_i,\nu_i}(u)>0$.
	If $\phi_{\mu_i,\nu_i}(u)=0$ then we can argue as in the previous case when $-\infty \leq R_i(u) < \ell_i$.
	If $\phi_{\mu_i,\nu_i}(u)>0$, then, since $(\sE^{\mu_i,\nu_i}_u)^c$ is linear on $(-\infty, S_i(u))$, we must have that $(\nu_i-S^{\nu_i}(\mu_{i,u}))$ has an atom at $R_i(u)=-\infty$, contradicting the fact that $\nu_i$ is integrable.		
\end{proof}

\bibliographystyle{plainnat}

\end{document}